\documentclass[a4paper,10pt]{report}

\usepackage{amsmath, amsthm, amssymb, amscd}
\usepackage[normalem]{ulem}
\usepackage[multiple]{footmisc}
\usepackage{verbatim}
\usepackage{stmaryrd}
\usepackage[UglyObsolete]{diagrams}
\usepackage{pxfonts}
\usepackage[usenames,dvipsnames]{color}
\usepackage[bookmarks]{hyperref}
\usepackage{rotating}
\usepackage{textcomp}

\theoremstyle{definition}
\newtheorem{theorem}{Theorem}[section]
\newtheorem{lemma}[theorem]{Lemma}
\newtheorem{proposition}[theorem]{Proposition}
\newtheorem{corollary}[theorem]{Corollary}
\newtheorem{notation-proposition}[theorem]{Notation-Proposition}
\newtheorem{proposition-notation}[theorem]{Proposition-Notation}
\newtheorem{calculation}[theorem]{Calculation}
\newtheorem{application}[theorem]{Application}

\theoremstyle{plain}
\newtheorem{thm}[theorem]{Theorem}

\newtheorem{cor}[theorem]{Corollary}

\theoremstyle{definition}
\newtheorem{construction}[theorem]{Construction}

\newtheorem{definition}[theorem]{Definition}
\newtheorem{notation}[theorem]{Notation}
\newtheorem{example}[theorem]{Example}

\newtheorem{remark}[theorem]{Remark}

\newtheorem{convention}[theorem]{Convention}

\newtheorem{openquiz}[theorem]{Open Problem}

\theoremstyle{plain}
\newtheorem*{exercise*}{Exercise}
\newtheorem*{proposition*}{Proposition}
\newtheorem*{corollary*}{Corollary}
\newtheorem*{lemma*}{Lemma}
\newtheorem*{theorem*}{Theorem}

\theoremstyle{definition}
\newtheorem*{definition*}{Definition}
\newtheorem*{rule*}{Rule}
\newtheorem*{notation*}{Notation}
\newtheorem*{remark*}{Remark}
\newtheorem*{construction*}{Construction}
\newtheorem*{examples*}{Examples}

\renewcommand{\Im}{\mathrm{Im}}

\newcommand{\bs}{\backslash}

\newcommand{\realizes}{\Vdash}

\newcommand{\N}{\mathbf{N}}

\newcommand{\E}{\mathcal{E}}
\newcommand{\F}{\mathcal{F}}
\newcommand{\xto}{\xrightarrow}

\renewcommand{\to}{\rightarrow}

\newcommand{\emto}{\hookrightarrow}
\newcommand{\onto}{\twoheadrightarrow}

\newcommand{\emfrom}{\hookleftarrow}
\newcommand{\xemfrom}[1]{\stackrel{#1}{\emfrom}}

\newcommand{\BI}{\Leftrightarrow}

\newcommand{\ri}{\rightarrow}
\newcommand{\RI}{\Rightarrow}
\newcommand{\LI}{\Leftarrow}

\newcommand{\Dom}{\mathrm{Dom}}

\newcommand{\dbl}{\llbracket}
\newcommand{\dbr}{\rrbracket}

\newcommand{\Sub}{\mathrm{Sub}}
\newcommand{\id}{\mathrm{id}}
\newcommand{\Ob}{\mathrm{Ob}}
\newcommand{\op}{\mathrm{op}}
\newcommand{\Set}{\mathbf{Set}}
\newcommand{\Eff}{\mathbf{Eff}}

\newcommand{\true}{\mathrm{true}}
\newcommand{\false}{\mathrm{false}}

\newenvironment{indentation}{\itemize\item[]}{\enditemize}

\newenvironment{indmath}{\itemize\item[]\math}{\endmath\enditemize}

\newarrow{Imto} >--->
\newarrow{Emto} C--->
\newarrow{Onto} ----{>>}
\newarrow{Dato} {}{dash}{dash}{dash}>
\newarrow{Doto} {}...>
\newarrow{Dot}  {}...{}
\newarrow{Dash} {}{dash}{dash}{dash}{}
\newarrow{Dato} {}{dash}{dash}{dash}{>}
\newarrow{Corresponds} <--->
\newarrow{Eq}   {}==={}

\newcommand{\defn}{\textbf}
\newcommand{\term}{\emph}

\newcommand{\PN}{{\mathrm{P}\N}}

\newcommand{\ex}{\mathrm{ex}}

\newcommand{\fr}{\mathrm{fr}}

\newcommand{\ev}{\mathrm{ev}}
\newcommand{\Ev}{\mathrm{Ev}}
\renewcommand{\P}{\mathrm{P}}

\newcommand{\la}{\langle}
\newcommand{\ra}{\rangle}

\newcommand{\sr}{\stackrel}

\newcommand{\proves}{\vdash}
\newcommand{\leftadjto}{\dashv}
\newcommand{\pto}{\rightharpoonup}

\newcommand{\ol}{\overline}

\newcommand{\eqv}{\simeq}
\newcommand{\eqq}{\equiv}

\newcommand{\imto}{\rightarrowtail}

\newcommand{\sub}{\mathrm{sub}}
\newcommand{\en}{\ \&\ }
\newcommand{\incl}{\mathrm{incl}}

\newcommand{\PA}{\mathrm{PA}}
\newcommand{\negneg}{{\neg\neg}}

\newcommand{\Heytpre}{\mathbf{Heytpre}}

\newcommand{\ph}{{(-)}}
\renewcommand{\P}{\mathrm{P}}
\newcommand{\PPN}{{\P\P\N}}

\newcommand{\Sh}{\mathrm{Sh}}
\newcommand{\K}{\mathcal{K}}
\renewcommand{\L}{\mathcal{L}}
\renewcommand{\O}{\mathcal{O}}
\newcommand{\A}{\mathcal{A}}
\newcommand{\B}{\mathcal{B}}
\newcommand{\C}{\mathcal{C}}
\newcommand{\D}{\mathcal{D}}
\newcommand{\itpr}[1]{\dbl #1 \dbr}

\newcommand{\ext}{\mathrm{ext}}

\newcommand{\otherwise}{\text{otherwise}}

\renewcommand{\a}{\mathbf{a}}
\newcommand{\ClSub}{\mathrm{ClSub}}
\newcommand{\T}{\mathrm{T}}

\renewcommand{\r}{\mathrm{r}}
\renewcommand{\l}{\mathrm{l}}
\newcommand{\rl}{\mathrm{rl}}
\newcommand{\lr}{\mathrm{lr}}

\newcommand{\lomo}{{\lo/\mo}}
\newcommand{\loar}{{\lo/\ar}}

\newcommand{\lo}{\mathrm{lo}}
\newcommand{\mo}{\mathrm{mo}}
\newcommand{\ar}{\mathrm{ar}}
\newcommand{\gr}{\mathrm{gr}}

\newcommand{\PPsN}{{\mathrm{PP}^*\N}}
\newcommand{\PsPN}{{\P^*\PN}}

\newcommand{\nil}{\mathrm{Nil}}

\newcommand{\Lop}{\mathbf{Lop}}

\newcommand{\enc}[1]{{ \ulcorner #1 \urcorner }}
\newcommand{\singt}[1]{{ \{ #1 \} }}

\newcommand{\bbl}{\textbf{[}}
\newcommand{\bbr}{\textbf{]}}

\newcommand{\len}{\mathrm{len}}

\newcommand{\textif}{\text{if }}

\newcommand{\nab}[1]{\nabla_{#1}}

\newcommand{\ini}{\preceq}
\newcommand{\Out}{\mathrm{Out}}

\newcommand{\loF}{\lo'}
\newcommand{\loS}{\lo''}

\newcommand{\stil}{\textnormal{\textlbrackdbl}}
\newcommand{\stir}{\textnormal{\textrbrackdbl}}

\newcommand{\claim}{\emph}

\newcommand{\nabeq}{\stackrel{\nabla}{=}}
\newcommand{\nabin}{\stackrel{\nabla}{\in}}

\newcommand{\NNO}{N}

\newcommand{\Id}{\mathrm{Id}}

\newcommand{\tofrom}{\rightleftarrows}

\newcommand{\lex}{\mathrm{lex}}

\newcommand{\stage}{\mathrm{stage}}
\newcommand{\Pitts}{{\F^*}}
\newcommand{\isLeaf}{\mathrm{isLeaf}}
\newcommand{\PsN}{{\P^*\N}}

\newcommand{\cvg}{\textnormal{\textdownarrow}}

\newcommand{\ups}{\textnormal{\textuparrow}}

\newcommand{\Ptl}{\mathrm{Ptl}}
\newcommand{\mon}{\mathrm{mon}}
\newcommand{\lop}{\mathrm{lop}}

\newcommand{\cons}{\textnormal{:}}
\newcommand{\conc}{*}

\newcommand{\Mo}{\mathrm{Mo}}
\newcommand{\Lo}{\mathrm{Lo}}

\newcommand{\EffTrip}{\mathbf{ET}}
\newcommand{\ET}{\mathbf{ET}}

\newcommand{\st}{\mathrm{st}}

\newcommand{\dsim}{\approx}

\newcommand{\Preord}{\mathbf{Preord}}

\newcommand{\cha}{\mathrm{char}}

\newcommand{\el}{\textnormal{!`}}
\newcommand{\er}{\textnormal{!}}

\newcommand{\fst}{\mathrm{fwd}}
\newcommand{\snd}{\mathrm{bwd}}
\newcommand{\fwd}{\mathrm{fwd}}

\newcommand{\fn}{\mathrm{fn}}
\newcommand{\Nds}{\mathrm{Nds}}
\newcommand{\Lvs}{\mathrm{Lvs}}
\newcommand{\Nil}{\mathrm{Nil}}

\newcommand{\suc}{\mathrm{suc}}
\newcommand{\subtree}{\mathrm{Subtree}}
\newcommand{\Subtree}{\mathrm{Subtree}}
\newcommand{\Subsight}{\mathrm{Subsight}}

\newcommand{\softdefn}{\textit}

\newcommand{\Pdense}[1]{\P_{#1-\mathrm{dense}}}

\newcommand{\LET}{\mathsf{let}}
\newcommand{\IN}{\mathsf{in}}

\renewcommand{\implies}{$\RI$}
\newcommand{\lLeaf}{\mathrm{lLeaf}}

\everymath{\displaystyle}

\title{Subtoposes of the Effective Topos}
\author{Sori Lee}
\date{August 2011\\[60pt]
Master's Thesis\\
Supervisor: Dr. J. van Oosten\\
Second examiner: Dr. B. van den Berg\\[240pt]
Utrecht University\\
Department of Mathematics}

\begin{document}

\maketitle

\tableofcontents

\addcontentsline{toc}{chapter}{\protect\numberline{}Introduction}
\chapter*{Introduction}
The effective topos $\Eff$ was discovered by Martin Hyland \cite{hyl82} in late 1970s.
It is an elementary topos whose internal logic is governed by Kleene's number realizability \cite{kle45}.
Because of this nature, it has been of interest for logic (constructivism) and computer science.
This thesis deals with a topos-theoretic aspect of the effective topos: we study its (geometric) subtoposes.

A distinguished example of subtopos of $\Eff$ is the category of sets, whose internal logic is of course classical mathematics.
It is the subtopos corresponding to the (Lawvere-Tierney) topology $\negneg$, and is the least non-degenerate subtopos.
This gives us a rudimentary motivation for studying subtoposes of $\Eff$: in the viewpoint of considering toposes as universes of mathematics, the non-degenerate subtoposes of $\Eff$ are the universes of mathematics between classical and `effective' mathematics.
A result that is sensible under this perspective has been known already since \cite{hyl82}: the poset of Turing degrees embeds order-reversingly into the poset of subtoposes of $\Eff$, with the degree of computable sets (the least degree) corresponding to $\Eff$ itself - so `taking some oracle for granted' gets one closer to classical mathematics, as one might have expected.\footnote{In Section \ref{sec:turing}, we will present a more complete overview of results we know about subtoposes given by Turing degrees.}

On the other hand, not much more than this is known.
Speaking of examples, the following is the list of previously known (non-trivial) subtoposes of $\Eff$ (as far as the author is aware of).
\begin{itemize}
 \item The category of sets
 \item The Turing degrees (references: \cite{hyl82} and \cite{phoa89})
 \item The Lifschitz subtopos (references: \cite{jvo91}, \cite{jvo96} and \cite{jvo08})
 \item An example by Andrew Pitts (reference: \cite[Example 5.8]{pitts81})
\end{itemize}
In this thesis, we attempt to make a little advance upon this situation.
%Our approach is along the following line of thoughts.

In any topos, we can associate to any subobject of any object the least topology for which the subobject is dense, and any topology on the topos arises in this way.\footnote{This will be explained in Chapter \ref{chap:lop}.}
In the effective topos, we may pay special attention to those topologies given by subobjects of $\negneg$-sheaves: we will call them \emph{basic} topologies.
The canonical topologies $\id,\top,\negneg$, as well as the example of Pitts, are instances of basic topologies (Section \ref{sec:examples}).
Every subobject of a $\negneg$-sheaf is as usual given by a set $X$ and a function $R: X \to \PN$ (where $\PN$ is the power set of $\N$),
and for the purpose of studying the associated topology we may view it just as the collection $\{R(x) \mid x \in X\} \in \PPN$.
Internally in the effective topos, every topology will turn out to be an NNO-indexed join of basic\footnote{With
a suitable definition of when \emph{internally} a topology is `basic': Definition \ref{def:basic}.} topologies.
As consequence, we can represent every topology by a function $\N \to \PPN$:
this is the final kind of data we will use in this thesis to present concrete subtoposes of $\Eff$.

One fruit of our study is the establishment of new (basic) examples of subtoposes.
Specifically, we deal with the subtoposes given by the following collections $\in \PPN$.
\begin{itemize}
\item For an ordinal $\alpha < \omega$ and a number $m \in \N$ with $2m < \alpha$, we consider the collection
$\O_m^\alpha = \{A \subseteq \alpha \mid \text{$A$ is a co-$m$-ton in $\alpha$}\}$.
\item We consider the collection $\O^\omega = \{\text{the cosingletons in $\N$}\}$.
\end{itemize}
The reader can expect following kinds of results.
\begin{itemize}
\item The collections $\O_m^\alpha$ yield infinitely many subtoposes. (Subsection \ref{subs:between})
\item The collection $\O^\omega$ gives the largest proper \emph{basic} subtopos. (Proposition \ref{prop:atom})
\item No subtopos of the form $\O_m^\alpha$ is a subtopos of a non-recursive Turing degree subtopos. (Proposition \ref{OnlyRecEffe})
\item No non-recursive Turing degree subtopos is a subtopos of a basic subtopos. (Proposition \ref{prop:TuringBasic})
\item The example of Pitts is a subtopos of every subtopos given by an arithmetical Turing degree. (Subsection \ref{subs:arith})
\end{itemize}

It is crucial for making such calculations about subtoposes, to have an understandable representation of corresponding topologies.
The way we settle this issue will be the core of this thesis.
We shall study a particular representation, due to Andrew Pitts, in more depth.
We will introduce a tree-like structure, to be called \emph{sights}, that clarifies, given a function $j: \PN \to \PN$ representing a topology on $\Eff$ in the way of Pitts and a set $p \in \PN$, what it means for an element $z \in \N$ to belong to the set $j(p)$.
Namely, the number $z$ will turn out to be (a code of) a function acting on sights in an appropriate way.
(Section \ref{sec:sights})

%The key ingredient for proofs of these results will be a notion of trees called \emph{sights}.
%They clarify            and provide a formalism  

%In this thesis, we make a little advance upon this situation.
%We will study a particular representation of subtoposes in a systematic way.
%We introduce the notion of \emph{sights}, which clarifies the general structure behind objects that represent subtoposes.
%With this improved understanding of the representation, we will be able to make concrete calculations that yield some new results.
%These include establishment of a new family of examples of subtoposes, as well as their comparisons to other examples.
%We also prove a new fact about known examples, namely that the example by Pitts is a subtopos of each subtopos determined by an arithmetical degree.

The text consists of three chapters, of which the first two have the nature of overview while the third one is an exposition of the research done.
In Chapter 1, we review some theory of triposes, that we use to present the effective topos and its subtoposes.
In Chapter 2, we review some topos theory about Lawvere-Tierney topologies.
In Chapter 3, we explain the way we present subtoposes of $\Eff$, survey examples of subtoposes and prove results about them.

\subsection*{Prerequisites}
It will be assumed that the reader is familar with basics of categorical logic, elementary topos theory and recursion theory.

\subsection*{Acknowledgements}
I owe my deepest gratitude to my supervisor, Dr. Jaap van Oosten.
Needless to say, I have benefited enormously from his expertise, and have learned a lot.
Without his teaching and guidance, and without his encouragements and patience, this work would not have been possible.

I am grateful to my colleague friends Jeroen Goudsmit and Daniel van Dijk who have constantly given encouragements while I have been struggling with the thesis.

I thank my family for their prayers for me.

\addcontentsline{toc}{chapter}{\protect\numberline{}Preliminaries}
\chapter*{Preliminaries}
\section*{Notations and rules}
\paragraph{Formal langauge}
\begin{itemize}
 \item A formula (a mathematical object) in a formal language is always context-sensitive.
In other words, when we speak of a formula as mathematical object, the specification of free variables belongs to its data.

 \item We may use the notation $x,y,z . \cdots$ to indicate which are the free variables in the formula being described.
For example, one may write down a formula as follows.
$$p,q . (p \RI q) \RI (h(p) \RI h(q))$$
Using this notation (i.e. prepending a list of free variables) is not a requirement, but when this notation is used, we list \emph{all} free variables of the formula to avoid confusion.

 \item Given formal formulas $\phi,\psi$, we write $\phi \eqq \psi$ if they are equivalent (in an appropriate sense depending on the context). We may append a subscript to clarify the meaning: for example, if $M$ is some structure in which $\phi,\psi$ are interpreted, the notation $\phi \eqq_M \psi$ would mean that the interpretations of $\phi$ and $\psi$ in $M$ are (in an appropriate sense) equivalent.

 % 
%  \item Let $t$ be a term in a multi-sorted language, and let $X_1,\ldots,X_n,Y$ be sorts.
% We may write $t: X_1 \ldots X_n \to Y$ in the obvious sense.
\end{itemize}

\paragraph{Recursion theory}
\begin{itemize}
 \item We denote the relation of Turing reducibility by $\leq_\T$.

 \item Let $1 \leq k \in \N$. Given $x_1,\ldots,x_k \in \N$, we denote as usual by $\la x_1,\ldots,x_k \ra$ the standard code of the $k$-tuple $(x_1,\ldots,x_k)$. Let $1 \leq i \leq k$. We denote by $\pi^k_i$ the `$i$-th projection' function $\N \to \N$, i.e. the function determined by the equation $\pi^k_i(\la x_1,\ldots,x_k \ra) = x_i$. Given $n \in \N$, we write $n_1 = \pi_1^2(n)$ and $n_2 = \pi_2^2(n)$.
\end{itemize}

\paragraph{Order theory}
\begin{itemize}
 \item Given a preorder $(P,\leq)$ and $x,y \in P$, we write $x \cong y$ if $x \leq y$ and $y \leq x$, and speak of $x$ and $y$ being \defn{isomorphic}.

 \item Let $P$ be a preorder, and let $x,y \in P$. We denote by $x \vee_P y$, if exists, an arbitrary join of $x$ and $y$ in $P$. The subscript may be omitted. Similarly we employ notations $\leq_P, \wedge_P, \RI_p, \top_P, \bot_P, \neg_P$, etc.
\end{itemize}
 
\paragraph{Set theory}
\begin{itemize}
 \item As usual, the function application notation $f(x)$ may just be denoted $f{x}$. We will make use of this simplication ubiquitously.

 \item Let $X$ be a set. We denote by $\P(X)$ the power set of $X$, and by $\P^*(X)$ the set $\P(X)\bs\{\emptyset\}$ of non-empty subsets of $X$.

 \item Let $X$ be a set. We denote by $X^*$ the set of finite tuples on $X$.

 \item Let $X$ be a set. Given a finite sequence $s$ in $X$ and an element $x \in X$, we denote by $x \cons s$ the sequence obtained by prepending $x$ to $s$ and by $s \cons x$ the sequence obtained by appending $x$ to $s$. Given sequences $s,s'$, we denote by $s \conc s'$ the result of concatenating $s'$ at the end of $s$.

 \item Let $X,Y$ be sets. We denote by $\Ptl(X,Y)$ the set of partial functions $X \pto Y$.
 
 \item (Dependent sum) Let $X$ be a set, and let $(Y_x \mid x \in X)$ be a family of sets. We denote by $\sum_{x \in X} Y_x$ the set $\{(x,y) \mid x \in X \en y \in Y_x\}$.
\end{itemize}
 
\paragraph{Miscellaneous}
\begin{itemize}
 \item We (very) often use brackets just as an alternative for parentheses (for visual reasons).
 
 \item The symbols $\wedge$ and $\vee$, when used as binary connectives, preceed $\RI$. For example, `$P \wedge Q \RI R$' stands for `$(P \wedge Q) \RI R$'.
 
 \item On the notation `$\BI$': `$P \BI Q$' means `$P \RI Q \wedge Q \RI P$'.

 \item In the ordinary context as well as internally in a topos, we allow ourselves not to distinguish notationally between a function $1 \to X$ and an element of $X$.
  
% \item The \defn{language of arithmetic} in this thesis is the first-order language with equality with a function symbol for each primitive recursive function.
 
 % \item We may use the underscore symbol to mean ``anything that we do not want to name''. For example, given $n \in \N$, we may write `let $\la a,\_ \ra = n$' in order to name the first component of the coded pair $n$ as `$a$'.

\end{itemize}

\subsection*{Possibly-undefined elements}
% Given two (informal) terms `$x$' and `$y$' that may be possibly undefined and that can be compared (in terms of (in)equality) if both were defined,  we write $x = y$ if $x$ and $y$ are both defined and are equal, and we write $x \simeq x'$ if either of $x,y$ being defined implies $x = y$.

\begin{definition*}
Let $X$ be a set.
A \defn{possibly-undefined element} of $X$ is either an element of $X$ or something written $\uparrow$ (pronounce: \emph{undefined}).
Let $x$ be a possibly-undefined element of $X$.
If $x$ is an element of $X$, we say that $x$ is \defn{defined} and write $x\cvg$.
\end{definition*}

\begin{notation*}
Let $X$ be a set, and let $x,x'$ be possibly-undefined elements of $X$.
We write $x = x'$ if $x$ and $x'$ are both defined and are equal (as elements of $X$).
We write $x \eqv x'$ if either of $x,x'$ being defined implies $x = x'$. 
\end{notation*}

The following Notation reinterprets our usual notation of partial function application in terms of the notion of possibly-undefined elements.

\begin{notation*}
Let $f: X \pto Y$ be a partial function, and let $x \in X$.
We denote by $f(x)$ a possibly-undefined element: if $f$ is defined on $x$, then $f(x)$ denotes the value, and if $f$ is not defined on $x$, then $f(x)$ denotes $\uparrow$.
\end{notation*}

\subsection*{Recursion-theoretic operations on the set $\PN$}
\begin{notation*}
Given $A,B \in \PN$, define
\begin{quote}
$A \RI B = \{e \in \N \mid \text{for all $n \in A$ one has $e(n) \in B$}\}$ \\
$A \wedge B = \{\la a,b \ra \mid a \in A \en b \in B\}$ \\
$A \vee B = \{\la 0,a \ra \mid a \in A\} \cup \{\la 1,b \ra \mid b \in B\}.$
\end{quote}
Let $X$ be a set. We may denote the notation $\bigcap_{x \in X}$ as $\forall_{x \in X}$, and $\bigcup_{x \in X}$ as $\exists_{x \in X}$. 
\end{notation*}

\subsection*{Preassemblies}
Recursion theory provides a notion of when a function $\N^k \to \N$ is computable (effective).
In addition to this, we will need to speak of effectivity of functions between sets other than (powers of) $\N$.
In order to do this, we introduce the notion of \emph{pre-assemblies}.  

\begin{definition*}
A \defn{preassembly} is a pair $(X,E)$, where $X$ is a set and $E$ is a function $X \to \PN$.
The function $E$ is called the \defn{encoding function}.
Given $x \in X$, a number $\in E(x)$ is called a \defn{code} or \defn{index} of $x$.

A preassembly $(X,E)$ for which $E(x) \neq \emptyset$ for all $x \in X$ is called an \defn{assembly}.\footnote{The assemblies form a full subcategory of the effective topos with rich structure, but we do not discuss this here. Our interest now is just in having a rigorous notion of effectivity of functions between sets other than powers of $\N$.}
\end{definition*}

\begin{notation*}
Let $X$ be a set, and let $x \in X$.
As a ``default symbolism'', we may denote by $E_X(x)$ or just by $E(x)$ the code set of $x$ under some encoding function.
Under this practice, we may simply say that $X$ is a preassembly.
\end{notation*}

\begin{definition*}
Let $X,Y$ be preassemblies, and let $f: X \pto Y$ be a partial function.
We say that a partial recursive function $\phi: \N \pto \N$ \defn{tracks} $f$ (or that $\phi$ is a \defn{tracking} of $f$) if for all $x \in \Dom(f)$ and $d \in E(x)$ we have $\phi(d)\cvg$ and $\phi(d) \in E(f(x))$.
We say that $f$ is \defn{computable} (or \defn{effective}) if $f$ has a tracking.
\end{definition*}

%\begin{construction*}
%Let $(X,D),(Y,E)$ be preassemblies.
%The preassembly $(X \times Y,G)$, where $G(x,y) = D(x) \wedge E(y)$, is called the \defn{product} and is denoted $(X,D) \times (X,E)$.
%The preassembly $(\{Y^X,G)$, where $G(f) = \forall_{x \in X} D(x) \RI E(f(x))$, is denoted $(Y,E)^{(X,D)}$ and is called the \defn{exponential}.
%\end{construction*}

%\begin{definition*}
%A \defn{partitioned assembly} is a pair $(X,\enc{\ph})$, where $X$ is a set and $\enc{\ph}$ is a function $X \to \N$.
%A partitioned assembly may be viewed as an assembly: given a partitioned assembly $(X,\enc{\ph})$, we view it as the assembly $(X,\{\enc{ph}\})$.
%\end{definition*}

\begin{examples*}[of preassemblies]\mbox{}
\begin{enumerate}
\item We view the set $\N$ as an assembly with $E(n) = \{n\}$.
\item We view the set $\N^*$ as an assembly with $E(x_1,\ldots,x_n) = \{\la n, \la x_1,\ldots,x_n \ra \ra\}$.
\item Let $X,Y$ be preassemblies. We view the set $X \times Y$ as a preassembly by $E(x,y) = E(x) \wedge E(y)$.
\item Let $X$ be a preassembly, and let $(Y_x \mid x \in X)$ be a family of preassemblies. We view the set $\sum_{x \in X} Y_x$ as a preassembly by $E(x,y) = E(x) \wedge E(y)$.
\item Let $X,Y$ be preassemblies. We view the set $\Ptl(X,Y)$ as a preassembly by $E(f) = \forall_{x \in \Dom(f)} E(x) \RI E(f(x))$.
\item Let $X$ be a preassembly, and let $U \subseteq X$. We view the set $U$ as a preassembly by $E_U(u) = E_X(u)$.
\end{enumerate}
\end{examples*}

\begin{notation*}
Let $X$ be a preassembly, and let $x \in X$ with $E(x) \neq \emptyset$.
We denote by $\enc{x}$ an arbitrary index of $x$.
\end{notation*}

%\begin{notation*}
%Let $(X,E)$ be an assembly.
%We may denote $(X,E)$ just by $X$.\footnote{Of course, as always, only if it the clear from the context what is meant.}
%\end{notation*}

%\begin{notation*}
%Let $X$ be a preassembly, and let $(Y_x \mid x \in X)$ be a family of preassemblies indexed by the set $X$.
%Let $Z$ be a preassembly, and let $(z_{x,y} \in Z \mid x \in X, y \in Y_x)$ be a family of elements of $Z$.
%We write that the assignment $(x \in X, y \in Y_x) \mapsto (z_{x,y} \in Z)$ is \defn{effective} (or \defn{computable}) if there is a partial recursive function $t: \N \times \N \pto \N$ such that for every $x \in X, y \in Y_x$ and every $a \in E(x), b \in E(y)$ we have $p(a,b) \in E(z_{x,y})$.
%\end{notation*}

\subsection*{Arithmetic in toposes}
For this discussion, let $\E$ be a topos with a natural numbers object $(\NNO,0,\suc)$.

\begin{definition*}
The \defn{language of arithmetic} is the untyped language consisting of a constant symbol $0$ and a unary function symbol $\suc$.
\end{definition*}

\begin{notation*}
Let $n \in \N$.
We denote the arithmetic term $\suc^n(0)$ by $\ol{n}$ or just by $n$.
\end{notation*}

\begin{definition*}
We define a notion of `interpretation' of the language of arithmetic in $\E$ w.r.t. $(\NNO,0,\suc)$.
The \defn{interpretation} of the symbols $0$ and $\suc$ are respectively the maps $0: 1 \to \NNO$ and $\suc: \NNO \to \NNO$.
The \defn{interpretation} of a (first- or higher-order) arithmetic term/formula is the term/formula in the language of $\E$ obtained in the obvious way from the notion of interpretation of the symbols $0$ and $\suc$.
\end{definition*}

\begin{rule*}
Under the notion of interpretation just defined, we may notationally as well as narationally regard symbols/terms/formulas
in the language of arithmetic as symbols/terms/formulas in the language of $\E$.
\end{rule*}

Now we can state as follows that the natural numbers object in every topos is a `model' of the Heyting arithmetic.

\begin{proposition*}
The (second-order) Peano axioms (not including the Law of Excluded Middle) are true in $\E$.
\end{proposition*}

\begin{proof}
Cf. \cite[Proposition 8.1.10]{fb94}.
\end{proof}

\begin{remark*}
For convenience, the \defn{language of arithmetic} is often definitionally extended\footnote{in the sense of model theory} with a function symbol for each primitive recursive function.
The way we interpret primitive recursive functions in $\E$ is as follows.
If $f: \N^k \to \N$ is a primitive recursive function, choose a formula $\phi(\vec{x},y)$ in the unextended language of arithmetic that provably defines $f$.\footnote{I.e., for all $x_1,\ldots,x_k \in \N$ we have $\PA \proves \phi(\ol{x_1},\ldots,\ol{x_k},\ol{f(x_1,\ldots,x_k)})$ and $\PA \proves \forall \vec{x} \exists! y \phi(\vec{x},y)$, where $\PA$ denotes as usual the Peano arithmetic.}
As (the interpretation in $\N$ of) the formula $\phi$ is primitive recursive, we have by \cite[p.128, (1)]{td88} that
$$\text{Heyting arithmetic} \proves \forall \vec{x} \exists! y \phi(\vec{x},y).$$
In particular, this last sentence is true in $\E$.
So by topos theory, there is a unique map $\ol{f}: \NNO^k \to \NNO$ in $\E$ such that
\begin{equation}\label{de1}
\E \models \forall \vec{x} \in \NNO^k \phi(x,\ol{f}(\vec{x})).
\end{equation}
We define the \defn{interpretation} of $f$ in $\E$ to be the map $\bar{f}$.
This makes (the natural numbers object of) $\E$ again a `model' of the extended theory.\footnote{By which we just mean that \eqref{de1} holds.}
\end{remark*}

\chapter{Triposes and their toposes}
We will not only deal with the effective topos, but also with its subtoposes.
While every subtopos can be presented as a subcollection of the ambient topos, it is useful to present a subtopos of the effective topos as given by the tripos-to-topos construction on a suitable tripos, for dealing with its logic.
For this reason, we review here some theory of triposes.

\section{Definition of tripos and first example}
We will present triposes basically in the style of \cite{jvo08}.
In l.c., the notion of tripos is relative to the choice of a category: a tripos is over some `base category' $\C$.
Although the examples we will deal with are all triposes over the category of sets, the author believes that the generality of base category, if restricted to cartesian closed ones, adds little cognitive or notational complication.
Therefore, we will also allow a (cartesian closed) base category.

Following \cite{jvo08}, we will use the language of preorder-enriched categories:

\begin{definition}
A \defn{preorder-enriched category} is a category $\C$ together with, for each pair $A,B$ of objects of $\C$, a preordering on the collection of arrows $\C(A,B)$, such that for every triple $A,B,C$ of objects of $C$ the composition mapping
$$\C(B,C) \times \C(A,B) \to \C(A,C)$$
is order-preserving.
Let $\C,\D$ be preorder-enriched categories.
A \defn{pseudofunctor} $F: \C \to \D$ is the data
\begin{itemize}\setlength{\itemsep}{0mm}
 \item for each object $X$ of $\C$, an object $F(X)$ of $\D$,
 \item for each arrow $f: X \to Y$ in $\C$, an arrow $F(f): F(X) \to F(Y)$,
\end{itemize}
such that
\begin{itemize}\setlength{\itemsep}{0mm}
 \item for $X,Y \in \Ob(C)$, the assignment $\C(X,Y) \to \D(F(X),F(Y)): f \mapsto F(f)$ is monotone,
 \item for each object $X$ of $\C$, we have $F(\id_X) \cong \id_{F(X)}$,
 \item for arrows $X \xto{f} Y$ and $Y \xto{g} Z$, we have $F(g \circ f) \cong F(g) \circ F(g)$.
\end{itemize}
Let $F,G$ be pseudofunctors $\C \to \D$.
A \defn{pseudo-natural transformation} $\Phi: F \to G$ is the data
\begin{itemize}
 \item for each object $X$ of $\C$ an arrow $\Phi_X: F(X) \to G(X)$
\end{itemize}
such that for every arrow $a: X \to Y$ in $\C$ the diagram in $\D$
\begin{diagram}
F(X)       & \rTo{\Phi_X} & G(X)       \\
\dTo{F(a)} &              & \dTo_{G(a)} \\
F(Y)       & \rTo{\Phi_Y} & G(Y)
\end{diagram}
\term{pseudo-commutes}, i.e. commutes up to isomorphism.
\end{definition}

\begin{example}
A prototypical example of a preorder-enriched category is that of preorder sets and monotone functions with the pointwise preorder on hom-sets.
We denote this preorder-enriched category by $\Preord$.
\end{example}

The notion of \emph{Heyting prealgebras} will base our definition of tripos:

\begin{definition}
A \defn{Heyting prealgebra} is a preorder, that is (considered as a category) cartesian closed and has finite coproducts.
A \defn{morphism} of Heyting prealgebras is a function that preserves all the structures up to isomorphism.
We denote by $\Heytpre$ the preorder-enriched category of Heyting prealgebras and their morphisms, with pointwise preordering on the hom collections.

A \defn{Heyting prealgebra with specific structure} is a tuple $(H,\leq,\wedge,\vee,\RI,0,1)$ where $(H,\leq)$ is a Heyting prealgebra and $\wedge,\vee,\RI,0,1$ are operations on $H$ for meet, join, Heyting implication, bottom and top respectively.
\end{definition}

Now we define tripos.

\begin{definition}%\footnote{For the \improve{continuity} with the reference [JvO, Realizability] I do not try to paraphrase the definition there.}
Let $\C$ be a cartesian closed category.
A $\C$-\defn{tripos} (or a \defn{tripos over} $\C$) is a pseudofunctor $P: \C^\op \to \Heytpre$ satisfying the following conditions:
\begin{itemize}
 \item[(i)] For every $\C$-morphism $f: X \to Y$, the map $P(f): P(Y) \to P(X)$ has (as a map of preorders) both a left adjoint $\exists_f$ and a right adjoint $\forall_f$ that satisfy the so-called \defn{Beck-Chevalley condition}: if
 \begin{diagram}
  X       & \rTo{f} & Y \\
  \dTo{g} &         & \dTo{h}  \\
  Z       & \rTo{k} & W
 \end{diagram}
 is a pullback square in $\C$, then the composite maps of preorders $\forall_f \circ P(g)$ and $P(h) \circ \forall_k$ are isomorphic.
 \item[(ii)] There is a $\C$-object $\Sigma$ and an element $\sigma$ of $P(\Sigma)$ such that for every object $X$ of $\C$ and every $\phi \in P(X)$ there is a morphism $[\phi]: X \to \Sigma$ such that $\phi$ and $P([\phi])(\sigma)$ are isomorphic elements of $P(X)$. (Such a pair $(\Sigma,\sigma)$ is called a \defn{generic element}.)
\end{itemize}
We may call elements of $P(X)$ \defn{predicates}.

A \defn{tripos with specific structure} is a tripos $P$ equipped with
\begin{itemize}\setlength{\itemsep}{0mm}
 \item for each object $X$ of $\C$, particular top, bottom, meet, join and implication operations on $P(X)$, (so each $P(X)$ is a Heyting prealgebra with specific structure)
 \item for each arrow $a: X \to Y$ in $\C$, particular left and right adjoints to the map $P(a): P(Y) \to P(X)$, and
 \item a particular generic element.
\end{itemize}
\end{definition}

We have defined triposes in two ways, one without and one with specific structure.
These are clearly not the same notion, but it is obvious that ``essentially'' they should be.
In order that all the discussions about triposes below can be applied to both notions of tripos, we introduce the following notation.

\begin{notation}
Let $(H,\leq)$ be a Heyting prealgebra, and let $x,y \in H$.
\begin{itemize}
 \item We denote by $0$ an arbitrary bottom element.
 \item We denote by $1$ an arbitrary top element.
 \item We denote by $x \wedge y$ an arbitrary meet of $x$ and $y$.
 \item We denote by $x \vee y$ an arbitrary join of $x$ and $y$.
 \item We denote by $x \RI y$ an arbitrary relative pseudo-complement of $x$ w.r.t. $y$.
% \item Let $\neg x$ be any negation of $x$. (So we have $\neg x \cong x \RI \bot)$.)
\end{itemize}
In a situation where $(H,\leq)$ underlies a Heyting prealgebra with specific structure whose top, bottom, meet, join or implication operation is named by the symbol `$0$', `$1$', `$\wedge$', `$\vee$' or `$\RI$' respectively, we face a collision of notation.
In such a case, we give priority to the namings for specific structure.

Let $P$ be a tripos without specific structure.
\begin{itemize}
 \item For each arrow $a: X \to Y$ in $\C$, we let $\exists_a$ be any left adjoint of the map $P(a)$ and $\forall_a$ any right adjoint.
 \item We let $(\Sigma,\sigma)$ be any generic element of $P$.
\end{itemize}
Similarly to the case of Heyting prealgebras, if $P$ underlies a tripos with specific structure and the notations $\exists_\ph$, $\forall_\ph$ and $(\Sigma,\sigma)$ name the specific structure, these namings have precedence over the notation just introduced.
\end{notation}

From now on, the word `tripos' can be understood as referring to either the notion of tripos without specific structure or the one with specific structure (unless explicitly designated).\footnote{But for psychological ease, one should not have two notions in mind simultaneously - just make a choice.}
However, the notion of tripos we need in this thesis is always the one with specific structure.

Let us proceed.
The notion of tripos is designed to interpret a certain kind of logic.
In the following, we survey the notion of formal language linked to triposes and how we interpret these languages in triposes.

\begin{definition}
The notion of language we will consider depends on a category, which serves as a means of specifying the types and the function symbols in the language.
Having this in mind, let $\C$ be a category with finite products.

A \defn{$\C$-typed (first-order) language (without equality)} consists merely of a collection of typed \defn{relation symbols}, with as \defn{types} the objects of $\C$.
The \defn{function symbols} are the arrows in $\C$.

\defn{Terms} and \defn{formulas} are defined as usual in (finitary) logic, respecting the modifiers `$\C$-typed', `first-order' and `without equality' appearing in the name for our notion of language.
The \defn{interpretation} of a term is defined in the obvious way as an arrow of $\C$.
Given a term $t: (X_1,\ldots,X_n) \to Y$, we will denote its interpretation $X_1 \times \ldots \times X_n \to Y$ just by $t$.

%; I wish to remark that the product structures of $\C$ enables forming pairs as terms, and point out that terms of type $\Sigma$ are regarded as a kind of formulas.
From now on, suppose that $\C$ is cartesian closed.
Let $P$ be a $\C$-tripos.

Given a $\C$-typed language $\L$, an \defn{interpretation} of $\L$ in $P$ associates to each relation symbol $r$ in $\L$ on a tuple $(X_1,\ldots,X_n)$ of types an element $\itpr{r} \in P(X_1 \times \ldots \times X_n)$.
Depending on an interpretation of $\L$ in $P$, the \defn{interpretation} of a formula $\phi(x_1^{X_1},\ldots,x_n^{X_n})$ is defined as an element $\dbl \phi \dbr \in P$ by induction on $\phi$ in the obvious way, with the following as the clause for the atomic case: given a relation symbol $r$ on some $(Y_1,\ldots,Y_m)$ and $m$ terms $t_1,\ldots,t_m: (X_1,\ldots,X_n) \to Y_1,\ldots,Y_m$, we put $\dbl r(t_1,\ldots,t_m) \dbr = P(t_1,\ldots,t_m)(\dbl r \dbr)$.\footnote{In the last expression, `$(t_1,\ldots,t_m)$' denotes the induced $\C$-morphism $X_1 \times \ldots \times X_n \to Y_1 \times \ldots \times Y_m$.}
The notation $\dbl \cdot \dbr$ may also be written as $\dbl^P \cdot \dbr$ for indicating in which tripos the language is interpreted.

Depending on an interpretation of $\L$ in $P$, given a formula $\phi$ in $\L$, we say that $\phi$ is \defn{true} in $P$, and write $P \models \phi$, if $\itpr{\phi} \in P(1)$ is a top element.
We have soundness: for $\L$ a $\C$-typed language and an interpretation of $\L$ in $P$, given a formula $\phi$ in $\L$, if $\phi$ can be proved in the intuitionistic predicate calculus, then $P \models \phi$.\footnote{Cf. \cite[Theorem 2.1.6]{jvo08}.}

There is a ``universal'' language associated to a $\C$-tripos $P$: its collection of relations symbols are all the elements of the preorders $P(X)$.
This language is called \defn{the language of $P$}, and obviously its relation symbols can be interpreted as themselves.
\end{definition}

The following is our first notion of morphism for triposes.
Later, we shall also consider the notion of `geometric morphisms' between triposes.

\begin{definition}
A \defn{transformation} $P \to Q$ of $\C$-triposes is just a pseudo-natural transformation, with $P,Q$ considered as pseudofunctors $\C \to \Preord$ (instead of $\C \to \Heytpre$).
\end{definition}

\begin{notation}
Let $\Phi: P \to Q$ be a map between triposes.
Given $X \in \Ob(\C)$, we may denote the map $\Phi_X: P(X) \to Q(X)$ just by $\Phi$.
\end{notation}

Finally, we introduce an (important) example: the \emph{Effective Tripos}.

\begin{example}
Let us equip the family $\PN^\ph = (\PN^X \mid \text{$X$ is a set})$ of sets with the structure of tripos.

Let $X$ be a set.
We make the set $\PN^X$ a Heyting prealgebra with specific structure, as follows.
First, define a preordering $\leq$ on $\PN^X$ by
$$r \leq s \iff \bigcap_{x \in X} (r(x) \RI s(x)) \neq \emptyset.$$
Next, define $\wedge,\vee,\RI: \PN^X \times \PN^X \to \PN^X$ to be the functions whose pointwise operations are $\wedge,\vee,\RI: \PN \times \PN \to \PN$ respectively.\footnote{See Preliminary.}
\claim{Then the tuple $(\PN^X,\vee,\wedge,\RI,\PN,\emptyset)$ is a Heyting prealgebra with specific structure.}

Let $f: X \to Y$ be a function.
Clearly, we have the `precomposition with $f$' operation $\PN^f: \PN^Y \to \PN^X$.
Given $r \in \PN^X$, define a function $\exists_f(r): Y \to \PN$ by
$$\exists_f(r)(y) = \exists_{x \in X, f(x) = y} r(x)$$
and a function $\forall_f(r): Y \to \PN$ by
$$\forall_f(r)(y) = \forall_{x \in X, f(x) = y} r(x).$$
\claim{The function $\PN^f: \PN^Y \to \PN^X$ is a morphism of Heyting prealgebras. The functions $\exists_f,\forall_f: \PN^X \to \PN^Y$ are order-preserving, and are respectively left,right adjoint to $\PN^f$. Moreover, the Beck-Chevalley condition holds. Finally, the pair $(\PN,\id_\PN)$ is a generic element.}

Therefore we have a tripos. We define the \defn{Effective Tripos} to be this tripos with the described specific structure. We denote it as $\EffTrip$.
\end{example}

\begin{notation}
Let $\phi$ be an $\ET$-sentence.
Given $n \in \N$, we write $n \realizes \phi$ and say that $n$ \defn{realizes} $\phi$ if $n \in \dbl \phi \dbr$.
We write $\realizes \phi$ if some natural number realizes $\phi$.
\end{notation}

\section{The tripos-to-topos construction}
Throughout this section, let $\C$ be a cartesian closed category.

The \emph{tripos-to-topos construction} associates to each $\C$-tripos $P$ a topos $\C[P]$.
In this section, we will describe this construction and review basics about the topos $\C[P]$.

\subsection{The construction}
\begin{construction}\label{con:TriposToTopos} % Tripos to topos construction.
Let $P$ be a $\C$-tripos.
We proceed to define the associated category $\mathcal{C}[P]$.

A \defn{partial equivalence predicate} on an $X \in \Ob(\C)$ is a predicate $\sim \in P(X \times X)$ such that
\begin{itemize}\setlength{\itemsep}{0mm}
 \item[(symmetry)] $P \models \forall x,x' . x \sim x' \RI x' \sim x$, and
 \item[(transitivity)] $P \models \forall x,x',x'' . x \sim x' \wedge x' \sim x'' \RI x \sim x''$.
\end{itemize}
An \defn{object} of $\C[P]$ is a pair $(X,\sim)$, where $X \in \Ob(\C)$ and $\sim \in P(X \times X)$ is a partial equivalence predicate on $X$.

Given an object $(X,\sim)$ of $\C[P]$ and a term $t\text{:}X$ in the language of $P$, we denote by $\ex(t)$ be the formula $t \sim t$.

Let $(X,\sim)$ and $(Y,\sim)$ be objects of $\C[P]$.
Given $F \in P(X \times Y)$, we say that $F$ is a \defn{functional predicate} from $(X,\sim)$ to $(Y,\sim)$, and write $F: (X,\sim) \to (Y,\sim)$, if\footnote{One might think that the following formulas are appearing somewhat out of thin air. But they are not - we will see soon in Remark \ref{rmk:NoThinAir} where they come from.}
\begin{itemize}\setlength{\itemsep}{0mm}
 \item[(extensionality)] $P \models \forall x,x',y,y' . x \sim x' \wedge y \sim y' \wedge F(x,y) \RI F(x',y')$ % extensionality
 \item[(single-valuedness)] $P \models \forall x,y,y' . x \sim x \wedge y \sim y \wedge y' \sim y' \wedge F(x,y) \wedge F(x,y') \RI y \sim y'$ % sv
 \item[(totality)] $P \models \forall x . x \sim x \RI \exists y . y \sim y \wedge F(x,y)$. % total
\end{itemize}
We write $\fr(F)$ for the $P$-sentence obtained as the meet of the three sentences above.
Given functional predicates $F,G: (X,\sim) \to (Y,\sim)$, we say that $F$ is \defn{equivalent} to $G$ if
\begin{indmath}
P \models \forall x^X,y^Y . x \sim x \wedge y \sim y \RI [F(x,y) \BI G(x,y)].
\end{indmath}
Clearly, this is an equivalence relation.
Given objects $(X,\sim)$ and $(Y,\sim)$ of $\mathcal{C}[P]$, a \defn{map} $(X,\sim) \to (Y,\sim)$ is an equivalence class of functional predicates $(X,\sim) \to (Y,\sim)$.
If $F$ is a functional predicate, we denote as usual by $[F]$ its equivalence class.

Given functional predicates $F: (X,\sim) \to (Y,\sim)$ and $G: (Y,\sim) \to (Z,\sim)$, we define their \defn{composite} $G \circ F$ to be the predicate
$$\dbl x,z . \exists y^Y . \ex(y) \wedge F(x,y) \wedge G(y,z) \dbr  \in P(X \times Z).$$
One easily verifies that $G \circ F$ is a functional predicate $(X,\sim) \to (Z,\sim)$.
We define the \defn{composite} $[G] \circ [F]$ of maps to be $[G \circ F]$; one can check that this is well-defined and associative.

Let $f: X \to Y$ be a morphism in $\C$.
We denote by $\ol{f}$ the predicate
$$\dbl x,y . f(x) \sim y \dbr  \in P(X \times Y).$$
We say that $f$ \defn{represents} a map $[F]: (X,\sim) \to (Y,\sim)$ if the associated predicate $\ol{f} \in P(X \times Y)$ is a functional predicate equivalent to $F$.
One can verify that $f$ represents a map $(X,\sim) \to (Y,\sim)$ if and only if
\begin{equation}\label{tt1}
P \models \forall x,x'^X . x \sim x' \RI f(x) \sim f(x').
\end{equation}
We say that $f$ is \defn{extensional} (w.r.t. the objects $(X,\sim)$ and $(Y,\sim)$) if the condition \eqref{tt1} is satisfied.
Given a term $f: X \to Y$ in the language of $P$, we write $\ext(f)$ for the sentence \eqref{tt1}.
If a morphism $f$ represents a map $(X,\sim) \to (Y,\sim)$ and a morphism $g$ represents a map $(Y,\sim) \to (Z,\sim)$, then the composite morphism $g \circ f$ represents the composite map, as one can verify.

The last definition suggests that some, but not all, maps in $\mathcal{C}[P]$ are represented by a morphism in $\C$;
it is usually nicer\footnote{This will become clear to the reader as we proceed.}, whenever possible, to represent a map by a morphism (rather than by a binary predicate).
Next we meet the first example of a map represented by a morphism.

Given an object $(X,\sim)$, one sees that the morphism $\id: X \to X$ is extensional.
By the \defn{identity} map on $(X,\sim)$ we mean the map represented by the morphism $\id: X \to X$;
this map clearly behaves as identity w.r.t. the composition defined above.

Finally, we define $\mathcal{C}[P]$ to be the category of objects, maps, composition and identity as just defined.
\end{construction}

\begin{remark}
There is a minor difference between the definition of the category $\C[P]$ in \cite{jvo08} and ours.
An objects of $\C[P]$ in l.c. are the same as ours, but a map there is an equivalence class of \emph{strict} functional predicates instead of just functional predicates:
these are functional predicates $F$ that additionally satisfy $P \models F(x,y) \RI \ex(x,y)$.
                                                                                                                                            
However, if $F$ is a functional predicate then $\dbl x,y . \ex(x,y) \wedge F(x,y) \dbr$ is clearly a strict functional predicate equivalent to $F$, and conversely every strict functional predicate is trivially a functional predicate.
This gives an isomorphism between the category $\C[P]$ in our sense and the one in the sense of \cite{jvo08}.

Also later, we will not require a predicate representing a subobject to be strict, unlike the traditional fashion.
The reason I deviate from the traditional presentation at this is that our way leads to simpler and more coherent formulations of results about triposes.
Corollary \ref{LogicPCP} below is a main showcase of this; one would appreciate the formulation there (of reducing\footnote{Those not familiar with this saying, will see soon.} formulas in $\C[P]$ to formulas in $P$) which is only possible because we do not require strictness.
\end{remark}

\begin{notation}
Let $(X,\sim),(Y,\sim)$ be objects of $\C[P]$, and let $f: X \to Y$ be a $\C$-morphism.
If $f$ represents a map $(X,\sim) \to (Y,\sim)$, we may denote that map just by $f$.
\end{notation}
% 
% ``Conversely'', I will show the following behavior.

\begin{convention}
Given that some symbol `$f$' stands for some map $(X,\sim) \to (Y,\sim)$, we will often name some $\C$-morphism $X \to Y$ representing the map $f$ with the same symbol `$f$'.\footnote{This is a deliberate habit serving as a means to prevent overflow of new names.}
\end{convention}

\subsection{Topos structure}\label{ToposStructure}
\begin{lemma}\label{MonoEpi}
Let $f: X \to Y$ be a $\C$-morphism representing a map $(X,\sim) \to (Y,\sim)$.
Then the map $f: (X,\sim) \to (Y,\sim)$ is an epi if and only if\footnote{The reader would recognize that the formulas appearing in this Lemma somewhat resemble the familiar characterizations of surjectivity and injectivity of functions. The precise connection will be unveiled soon in Remark \ref{rmk:NoThinAir}.}
$$P \models \forall y^Y . \ex(y) \RI \exists x^X . \ex(x) \wedge f(x) \sim y.$$
The map $f: (X,\sim) \to (Y,\sim)$ is a mono if and only if
$$P \models \forall x,x'^X . \ex(x,x') \wedge f(x) \sim f(x') \RI x \sim x'.$$
\end{lemma}

\begin{proof}
Left to the reader.
\end{proof}

\begin{construction}[Terminal object]
Consider a terminal object $1$ of $\C$, and the predicate $\top \in P(1 \times 1)$.
\claim{The pair $(1,\top)$ is a terminal object of $\C[P]$.
In fact, for any object $(X,\sim)$, the unique $\C$-morphism $!: X \to 1$ represents the unique map $(X,\sim) \to (1,\top)$.}
We will denote the object $(1,\top)$ by $1$.
\end{construction}

\begin{proof}
Clearly, $\top \in P(1 \times 1)$ is a partial equivalence predicate on $1 \in \Ob(\C)$.
So the pair $(1,\top)$ is an object of $\C[P]$.
Let $(X,\sim)$ be an object of $\C[P]$.
The unique morphism $!: X \to 1$ in $\C$ represents a map $(X,\sim) \to (1,\top)$, as we clearly have $P \models \forall x,x' . x \sim x' \RI !(x) \top !(x')$.
We now show that this is the only map $(X,\sim) \to (1,\top)$.
Let $F$ be a functional predicate $(X,\sim) \to (1,\top)$.
Notice that $P \models \forall y,y'^1 . \ex(y,y') \RI y \top y'$.
Meanwhile, the `totalness' of $F$ says $P \models \forall x^X . \ex(x) \RI [\exists y^1 . \ex(y) \wedge F(x,y)]$.
So the `single-valuedness' of $F$ yields $P \models \forall x^X \forall y^1 . \ex(x,y) \RI F(x,y)$.
In turn, this gives $P \models \forall x^X \forall y^1 . \ex(x,y) \RI [F(x,y) \BI !(x) \top *]$,
i.e., the predicate $F$ and the morphism $!$ represent the same map, as desired.
\end{proof}

\begin{construction}[Products]
Let $(X,\sim),(Y,\sim)$ be objects of $\C[P]$.
Consider the object $(X \times Y,\sim)$, where $(x,y) \sim (x',y') = \dbl x \sim x' \wedge y \sim y' \dbr$.
\begin{itemize}
\item[\claim{(a)}] \claim{This object, together with the maps $(X \times Y,\sim) \to (X,\sim)$ and $(X \times Y,\sim) \to (Y,\sim)$ represented by the projection morphisms, is a product of $(X,\sim)$ and $(Y,\sim)$.}
\end{itemize}
In fact:
Given functional relations $F: (W,\sim) \to (X,\sim)$ and $G: (W,\sim) \to (Y,\sim)$,
let $\la F,G \ra$ be the predicate on $W \times (X \times Y)$ defined by
$$\la F,G \ra(w,(x,y)) = \dbl F(w,x) \wedge G(w,y) \dbr.$$
Then
\begin{itemize}
\item[\claim{(b)}] \claim{$\la F,G \ra$ is a functional predicate $(W,\sim) \to (X \times Y,\sim)$, and represents the induced map $\la [F],[G] \ra: (W,\sim) \to (X \times Y,\sim)$.}
\end{itemize}
Moreover,
\begin{itemize}
\item[\claim{(c)}] \claim{if $f: (W,\sim) \to (X,\sim)$ and $g: (W,\sim) \to (Y,\sim)$ are representative $\C$-morphisms, then the induced morphism $\la f,g \ra: W \to X \times Y$ represents the induced map $\la [f],[g] \ra: (W,\sim) \to (X \times Y,\sim)$.}
\end{itemize}
We will denote the object $(X \times Y, \sim)$ by $(X,\sim) \times (Y,\sim)$.
\end{construction}

\begin{proof}
Straightforward verifications.
\end{proof}

\begin{construction}[Exponentials]
Let $(X,\sim)$ and $(Y,\sim)$ be objects of $\C[P]$.
Consider the object $(\Sigma^{X \times Y},\sim)$, where
\begin{quote}
$F \sim G = \dbl \fr(F) \wedge \forall x,y . x \sim x \wedge y \sim y \RI (F(x,y) \BI G(x,y)) \dbr$.
\end{quote}
Consider the functional predicate $\Ev: (\Sigma^{X \times Y},\sim) \times (X,\sim) \to (Y,\sim)$ defined by
\begin{quote}
$\Ev((F,x),y) = F(x,y)$.
\end{quote}
Then
\begin{itemize}
\item[\claim{(a)}] \claim{the object $(\Sigma^{X \times Y},\sim)$ together with the map $[\Ev]$ is an exponential $(Y,\sim)^{(X,\sim)}$.}
\end{itemize}
In fact,
\begin{itemize}
\item[\claim{(b)}] \claim{if a morphism $F: (W \times X) \times Y \to \Sigma$ is a fuctional predicate $(W,\sim) \times (X,\sim) \to (Y,\sim)$,
then the corresponding morphism $W \to \Sigma^{X \times Y}$ represents the exponential transpose $(W,\sim) \to (\Sigma^{X \times Y},\sim)$.}
\end{itemize}
Moreover,
\begin{itemize}
\item[\claim{(c)}] \claim{if a morphism $f: W \times X \to Y$ represents a map $(W,\sim) \times (X,\sim) \to (Y,\sim)$,
then the functional predicate $(W,\sim) \to (\Sigma^{X \times Y},\sim)$ given by
\begin{quote}
$w,G . \forall x . \ex(x) \RI G(x,f(w,x))$
\end{quote}
represents the exponential transpose $(W,\sim) \to (\Sigma^{X \times Y},\sim)$.}
\end{itemize}
We may denote the object $(\Sigma^{X \times Y},\sim)$ by $(Y,\sim)^{(X,\sim)}$.
\end{construction}

\begin{proof}
Straightforward verifications.
\end{proof}

\begin{construction}[Simply representable exponentials]\label{con:simply}
In the spirit of the phenomenon that some maps in $\C[P]$ can be represented not only by a functional predicate but also by a $\C$-morphism, some exponentials in $\C[P]$ have a simpler presentation (compared to the presentation given in the previous lemma).

For $(X,\sim)$ and $(Y,\sim)$ objects of $\Set[P]$, consider the object $(Y^X,\sim)$ where
\begin{quote}
$f \sim g = \dbl \ex(f) \wedge (\forall x . x \sim x \RI f(x) \sim g(x)) \dbr.$
\end{quote}
Also consider the evaluation function $\ev: Y^X \times X \to Y$, which
\begin{itemize}
 \item[\claim{(a)}] \claim{represents a map $(Y^X,\sim) \times (X,\sim) \to (Y,\sim)$.}
\end{itemize}
Soon we will encouter several cases where this simpler data, i.e. the object $(Y^X,\sim)$ together with the map $\ev: (Y^X,\sim) \times (X,\sim) \to (Y,\sim)$, is the exponential $(Y,\sim)^{(X,\sim)}$.
This is the case precisely if
\begin{equation}\label{BarRepSur}
P \models \forall F \in \Sigma^{X \times Y} . \fr(F) \RI \exists f \in Y^X . \ext(f) \wedge F \sim \ol{f}
\end{equation}
holds; we explain this as follows.

Consider the function $\ol{\ph}: Y^X \to \Sigma^{X \times Y}: f \mapsto \ol{f}$, which
\begin{itemize}
 \item[\claim{(b)}] \claim{represents a mono $(Y^X,\sim) \to (\Sigma^{X \times Y},\sim)$.}
\end{itemize}
Then
\begin{itemize}
 \item[\claim{(c)}] \claim{we have a commuting triangle}
\end{itemize}
\begin{diagram}
(Y^X,\sim) \times (X,\sim)                 &              & \\
\dEmto{\ol{\ph} \times \id}                 & \rdTo{\ev}   & \\
(\Sigma^{X \times Y},\sim) \times (X,\sim) & \rTo{[\Ev]} & (Y,\sim).
\end{diagram}
Therefore $(Y^X,\sim)$ and $\ev: (Y^X,\sim) \times (X,\sim) \to (Y,\sim)$ are the exponential $(Y,\sim)^{(X,\sim)}$ if and only if the map $\ol{\ph}$ is epi, which is\footnote{see Lemma \ref{MonoEpi}} the case precisely if \eqref{BarRepSur} holds.
\begin{itemize}
 \item[\claim{(d)}] \claim{If $F: (W,\sim) \times (X,\sim) \to (Y,\sim)$ is a functional predicate, then the functional predicate $(W,\sim) \to (Y^X,\sim)$
given by
\begin{quote}
%$w,g . \forall x,y . \ex(x,y) \RI [F((w,x),y) \BI g(x) \sim y]$
$w,g . \forall x . \ex(x) \RI F((w,x),g(x))$ 
\end{quote}
represents the transpose map.}
 \item[\claim{(e)}] \claim{If $f: (W,\sim) \times (X,\sim) \to (Y,\sim)$ is an extensional morphism,
then the transpose morphism $f^\sharp: W \to Y^X$ represents the transpose map $(W,\sim) \to (Y^X,\sim)$.}
\end{itemize}
If the object $(Y^X,\sim)$ underlies the exponential, we may denote it by $(Y,\sim)^{(X,\sim)}$.
\end{construction}

\begin{proof}
Routine verifications.
\end{proof}

\begin{construction}[Subobject classifier and power object]\label{SubClC[P]}
%The set $\Sigma$ is a preorder, via the obvious bijections $\Sigma \cong \Set(1,\Sigma) \cong P(1)$.
Consider the object $(\Sigma,\BI)$, where $\BI \in P(\Sigma \times \Sigma)$ is the predicate $\dbl p,q . p \BI q \dbr$.
Consider the morphism $\true: 1 \to \Sigma$, representing $\top \in P(\Sigma)$.
\begin{itemize}
 \item[\claim{(a)}] \claim{The object $(\Sigma,\BI)$ together with the map $\true: 1 \to (\Sigma,\BI)$ is a subobject classifier of $\C[P]$.}
 \item[\claim{(b)}] \claim{Every map $(X,\sim) \to (\Sigma,\BI)$ can be represented by a morphism, and this holds even in the following internal sense.\footnote{Cf. \eqref{BarRepSur}.}
\begin{quote}
$P \models \forall F \in P(X \times \Sigma) . \fr(F) \ri \exists f \in \Sigma^X (\ext(f) \wedge F \sim \bar{f})$.
\end{quote}}
 \item[\claim{(c)}] \claim{If $M: (W,\sim) \to (X,\sim)$ is a functional predicate representing a mono,
then the morphism $X \to \Sigma$ given by
\begin{quote}
$x . \exists w \in W . w \sim w \wedge M(w,x)$
\end{quote}
represents the characteristic arrow of the mono $[M]: (W,\sim) \emto (X,\sim)$.}
 \item[\claim{(d)}] \claim{If $r: (X,\sim) \to (\Sigma,\BI)$ is a extensional morphism, consider the object $(X,\sr{r}{\sim})$ where
\begin{quote}
$x \sr{r}{\sim} x' \eqq r(x) \wedge x \sim x'$.
\end{quote}
Then the function $\id: X \to X$ represents a mono $(X,\sr{r}{\sim}) \to (X,\sim)$, and the diagram
\begin{diagram}
(X,\sr{r}{\sim}) & \rTo{!} & 1 \\
\dTo{\id}        &         & \dTo{\true} \\
(X,\sim)         & \rTo{r} & (\Sigma,\BI)
\end{diagram}
is a pullback.}
\end{itemize}
\end{construction}

\begin{proof}
Straightforward verifications.
\end{proof}

\begin{remark}\label{rmk:OmegaPowerSimple}
The statement \ref{SubClC[P]}(b) tells us that given any object $(X,\sim)$ the exponential $\Omega^{(X,\sim)}$ is simply representable in the sense of Construction \ref{con:simply}.
\end{remark}

\begin{theorem}
The category $\C[P]$ is a topos. \qed
\end{theorem}

\subsection{The internal logic of $\C[P]$}\label{ToposLogic}
We now review how to represent formulas in the language of the topos $\C[P]$ as formulas in the language of the tripos $P$, so that eventually satisfactions of the form $\C[P] \models \cdots$ reduce to satisfactions of the form $P \models \cdots$.

\begin{proposition}\label{EqRep}
Let $(X,\sim)$ be an object of $\C[P]$.
The predicate $\sim \in P(X \times X)$ represents the equality relation $=$ on $(X,\sim) \times (X,\sim)$.
\end{proposition}

\begin{proof}
Since we tautologically have $P \models \forall s,t^{X \times X} . s \sim t \RI s \sim t$, the predicate $\sim \in P(X \times X)$ represents a subobject of $(X \times X,\sim)$.
This subobject contains\footnote{See Construction \ref{SubClC[P]}.} the mono $\id: (X \times X,\dsim) \to (X \times X,\sim)$, where $s \dsim t = s \sim s \wedge s \sim t$.
But this mono is isomorphic to the mono $\la \id_X,\id_X \ra: (X,\sim) \to (X \times X,\sim)$, as the morphism $\la \id_X,\id_X \ra: X \to X \times X$ gives the isomorphism by the following observations.
\begin{itemize}
\item $P \models \forall x,y \in X . x \sim x \wedge y \sim y \RI [\la \id_X,\id_X \ra(x) \dsim \la \id_X,\id_X \ra(y) \RI x \sim y]$.
\item $P \models \forall s \in X \times X . s \dsim s \RI \exists x \in X . \la \id_X,\id_X \ra(x) \dsim s$.
\item $(X \xto{\la \id_X,\id_X \ra} X \times X) = (X \xto{\la \id_X,\id_X \ra} X \times X \xto{\id} X \times X)$.
\end{itemize}
But the latter mono, which is the mono $\la \id_{(X,\sim)},\id_{(X,\sim)} \ra: (X,\sim) \to (X,\sim) \times (X,\sim)$, belongs by topos theory to the relation $=$ on $(X,\sim) \times (X,\sim)$.
The conclusion follows.
\end{proof}

\begin{lemma}
Let $A,B$ be subobjects of an object $(X,\sim)$, and let $a,b \in P(X)$ represent $A,B$ respectively.
We have $A \leq B$ if and only if\\ $P \models \forall x . x \sim x \RI (a(x) \RI b(x))$.\footnote{You will soon see why the latter is a natural, not an ad hoc, formula.}
\end{lemma}

\begin{proof}
We have
\begin{quote}
$A \leq B$
\end{quote}
if and only if (convince yourself!)
\begin{quote}
$\id: X \to X$ represents a map $(X,\sr{a}{\sim}) \to (X,\sr{b}{\sim})$
\end{quote}
if and only if
\begin{quote}
$P \models \forall x,x' . a(x) \wedge x \sim x' \RI b(x) \wedge x \sim x'$
\end{quote}
if and only if
\begin{quote}
$P \models \forall x . x \sim x \RI (a(x) \RI b(x))$ 
\end{quote}
as desired.
\end{proof}

\begin{proposition}
Let $(X,\sim),(Y,\sim)$ be objects of $\C[P]$.
\it\begin{itemize}
 \item[(a)] The Heyting operations on $\Sub(X,\sim)$ are given by the Heyting operations on $P(X)$, in the sense: for exmaple, if predicates $a,b \in P(X)$ represent subobjects of $(X,\sim)$ then $a \wedge b$ represents $\sub(a) \wedge \sub(b)$.
 \item[(b)] The `pullback along a morphism' operation $\Sub(Y,\sim) \to \Sub(X,\sim)$ is given as follows: if a morphism $f: X \to Y$ represents a map $(X,\sim) \to (Y,\sim)$ and some $b \in P(Y)$ represents some $B \in \Sub(X,\sim)$, then $P(f)(b) \in P(X)$ represents $f^*(B) \in \Sub(X,\sim)$.
\end{itemize}\rm
Let us denote by $\pi$ the projection morphism $X \times Y \to Y$.
\it\begin{itemize}
 \item[(c)] If some $a \in P(X \times Y)$ represents some $A \in \Sub((X,\sim) \times (Y,\sim))$, then $\dbl y . \forall x \in X . \ex(x) \RI a(x,y) \dbr \in P(Y)$ represents $\forall_\pi(A)$.
 \item[(d)] If some $a \in P(X \times Y)$ represents some $A \in \Sub((X,\sim) \times (Y,\sim))$, then $\dbl y . \exists x \in X . \ex(x) \wedge a(x,y) \dbr \in P(Y)$ represents $\exists_\pi(A)$.
\end{itemize}\rm
\end{proposition}

\begin{proof}
Let $a,b,c \in P(X)$.

(a)
We have $\sub(\bot) \leq \sub(a)$ iff $P \models \forall x . \ex(x) \RI [\bot \RI a(x)]$, but the latter is trivially true. 
Therefore $\sub(\bot)$ is the bottom element of $\Sub(X,\sim)$.
Similarly, $\sub(\top)$ is the top element of $\Sub(X,\sim)$.

We have: $\sub(a) \leq \sub(b \wedge c)$,
if and only if $P \models \forall x . \ex(x) \RI [a(x) \RI b(x) \wedge c(x)]$,
if and only if $P \models \forall x . \ex(x) \RI [a(x) \RI b(x)]$ and $\forall x . \ex(x) \RI [a(x) \RI c(x)]$,
if and only if $\sub(a) \leq \sub(b)$ and $\sub(a) \leq \sub(c)$.
Therefore $\sub(b \wedge c)$ is the meet of $\sub(b)$ and $\sub(c)$.
Similarly, $\sub(b \vee c)$ is the join of $\sub(b)$ and $\sub(c)$.

We have: $\sub(a) \wedge \sub(b) \leq \sub(c)$,
if and only if $P \models \forall x . \ex(x) \RI [a(x) \wedge b(x) \RI c(x)]$,
if and only if $P \models \forall x . \ex(x) \RI [a(x) \RI (b(x) \RI c(x))]$,
if and only if $\sub(a) \leq \sub(b \RI c)$.
Therefore $\sub(b) \RI \sub(c) = \sub(b \RI c)$.

(b) We know (topos theory) that $\cha(f^*(B)) = \cha(B) \circ f$,
so this map $(X,\sim) \to \Omega$ is represented by the composite morphism $b \circ f: X \to \Sigma$.
So $P(b \circ f)(\sigma) = P(f)(P(b)(\sigma)) = P(f)(b)$ represents $f^*(B)$.

(c)
Let $b \in P(Y)$ represent some $B \in \Sub(Y,\sim)$.
We want: $\pi^*(B) \leq A$ iff $B \leq \dbl y . \forall x \in X . \ex(x) \RI a(x,y) \dbr$.
The LHS holds,
if and only if $P \models \forall (x,y) \in X \times Y . \ex(x,y) \RI [b(\pi(x,y)) \RI a(x,y)]$,
if and only if $P \models \forall y \in Y . \ex(y) \RI [b(y) \RI \forall x^X . \ex(x) \RI a(x,y)]$,
if and only if the RHS holds, as desired.

(d)
Let $b \in P(Y)$ represent some $B \in \Sub(Y,\sim)$.
We want: $\dbl y . \exists x \in X . \ex(x) \wedge a(x,y) \dbr \leq B$ iff $A \leq \pi^*(B)$.
The LHS holds,
if and only if $\forall y^Y . \ex(y) \RI [\exists x^X (\ex(x) \wedge a(x,y)) \RI b(y)]$,
if and only if $\forall (x,y)^{X \times Y} . \ex(x,y) \RI [a(x,y) \RI b(\pi(x,y))]$,
if and only if the RHS holds, as desired.
\end{proof}

\begin{corollary}\label{LogicPCP}
Let $\phi(\vec{x}),\psi(\vec{x})$ be $\C[P]$-formulas (their free variables are $\vec{x}$), and let $\phi'(\vec{x}),\psi'(\vec{x})$ be $P$-formulas such that (the interpretation of) $\phi',\psi'$ represent (the interpretration of) $\phi,\psi$ respectively.
\begin{itemize}
 \item The $\C[P]$-formula $\phi(\vec{x}) \wedge \psi(\vec{x})$ is represented by the $P$-formula $\phi'(\vec{x}) \wedge \psi'(\vec{x})$.
 \item The $\C[P]$-formula $\phi(\vec{x}) \vee \psi(\vec{x})$ is represented by the $P$-formula $\phi'(\vec{x}) \vee \psi'(\vec{x})$.
 \item The $\C[P]$-formula $\phi(\vec{x}) \RI \psi(\vec{x})$ is represented by the $P$-formula $\phi'(\vec{x}) \RI \psi'(\vec{x})$.
 \item The $\C[P]$-formula $\neg \phi(\vec{x})$ is represented by the $P$-formula $\neg \phi'(\vec{x})$.
\end{itemize}
Let $\phi(\vec{x},y)$ be a $\C[P]$-formula, and let $\phi'(\vec{x},y)$ be a $P$-formula that represents $\phi(\vec{x},y)$.
\begin{itemize}
 \item The $\C[P]$-formula $\vec{x} . \forall y^{(Y,\sim)} \phi(\vec{x},y)$ is represented by the $P$-formula $\vec{x} . \forall y^Y . y \sim y \RI \phi'(\vec{x},y)$.
 \item The $\C[P]$-formula $\vec{x} . \exists y^{(Y,\sim)} \phi(\vec{x},y)$ is represented by the $P$-formula $\vec{x} . \exists y^Y . y \sim y \wedge \phi'(\vec{x},y)$.\qed
\end{itemize}
\end{corollary}

\begin{construction}\label{con:CPtoP}
Let $\phi$ be a $\C[P]$-formula.
Given for each relation symbol appearing in $\phi$ a predicate from $P$ representing it, we may build from the $\C[P]$-formula $\phi$ an associated $P$-formula inductively as suggested by Proposition \ref{EqRep} and Corollary \ref{LogicPCP}.
We call the resulting $\C[P]$-formula the \defn{standard translation} of $\phi$, and denote it by $\el^{P/\C[P]} \phi \er$.
Clearly, \claim{the $P$-formula $\el^{P/\C[P]} \phi \er$ represents the $\C[P]$-formula $\phi$}.\qed
\end{construction}

\begin{proposition}
Let $\phi$ be a \emph{sentence} in $\C[P]$, and let $\phi'$ be a sentence in $P$ that represents $\phi$.
Then $\C[P] \models \phi$ iff $P \models \phi'$.
\end{proposition}

\begin{proof}
We have
\begin{quote}
$\C[P] \models \phi$
\end{quote}
iff
\begin{quote}
the characteristic map $1 \to (\Sigma,\BI)$ of $\phi$ equals the map $\true$
\end{quote}
iff (since $\phi'$, regarded as a morphism $1 \to \Sigma$, represents the characteristic map)
\begin{quote}
$P \models \forall * \in 1 . * \top * \RI (\phi' \BI \top)$
\end{quote}
iff
\begin{quote}
$P \models \phi'$
\end{quote}
as desired.
\end{proof}

\begin{remark}\label{rmk:NoThinAir}
Construction \ref{con:TriposToTopos} and Lemma \ref{MonoEpi} involved a number of formulas in the language of $P$ that resembled familiar statements from set theory, but were not exactly the same.
The previous Construction and Proposition explain them now completely: we see that these $P$-formulas are results of applying the construction $\el^{P/\C[P]} \cdot \er$ to the $\C[P]$-formulas that are exactly the relevant statements from set theory.
\end{remark}

\subsection{Derived structures in the topos $\C[P]$}
In \S\ref{ToposStructure} we have given a presentation of the core structures (product, exponential, subobject classifier) that make the category $\C[P]$ a topos.
We know from topos theory that these structures connote many other categorical structures (e.g. limits), and it is our interest now to present such `derived' structures of $\C[P]$.
Usually we define categorical structures in terms of universal properties, but in toposes they can as well be characterized using the internal logic.
The advantage of the logical characterization, in our case of the topos $\C[P]$, is that we immediately obtain a presentation of the structure using the `logic machinary' from \S\ref{ToposLogic}.
So our work below will mainly be some topos theory of giving a logical description of the structure we are interested in.

\subsubsection{Power objects}
From topos theory we know that the power object $\P{X}$ is the exponential $\Omega^X$, in the way that the evaluation map $X \times \Omega^X \to \Omega$ is the map $\in: X \times P{X} \to \Omega$.
So given any object $(X,\sim)$ in $\C[P]$, any presentation of the exponential $\Omega^{(X,\sim)}$ gives a presentation of the power object $\P(X,\sim)$.

\begin{notation}
We shall denote by $\P(X,\sim)$ the object $\Omega^{(X,\sim)}$.
\end{notation}

\subsubsection{Partial map classifiers}
\begin{definition}
In a category, a \defn{partial map} $X \pto Y$ is a pair $(U,f)$, where $U$ is a subobject of an object $X$ and $f$ is a map $U \to Y$.

Let $Y$ be an object of a category.
A \defn{partial map classifier} of $Y$ is an object $\tilde{Y}$ and a monomorphism $\eta: Y \emto \tilde{Y}$ with the following property: for every partial map $X \xemfrom{m} U \xto{f} Y$ there is unique map $\tilde{f}: X \to Y$ such that the square
\begin{diagram}
U           & \rTo{f}           & Y \\  
\dEmto{m}   &                   & \dTo{\eta} \\
X           & \rDato{\tilde{f}} & \tilde{Y}
\end{diagram}
is a pullback.
\end{definition}

In a topos, every object has a partial map classifier.
The following proposition proves this, providing a construction described using the internal logic.

\begin{proposition}
Let $\E$ be a topos, and let $Y$ be an object in $\E$.
Let
$$\tilde{Y} = \{A \in \P(Y) \mid \forall y,y' \in Y . y \in A \wedge y' \in A \RI y = y'\},$$
and let $\eta$ be the map $Y \to \tilde{Y}$ defined internally by
$$\eta(y) = \{y' \in Y \mid y' = y\}.$$
\claim{Then $\eta: Y \to \tilde{Y}$ is a partial map classifier of $Y$.}
In fact, if $X \xemfrom{m} U \xto{f} Y$ is a partial map, then the map $\tilde{f}: X \to \tilde{Y}$ defined by
$$\tilde{f}(x) = \{y \in Y \mid \exists u \in U . m(u) = x \wedge f(u) = y\}.$$
is its characteristic map.
\end{proposition}

\begin{proof}
Elementary topos theory.
\end{proof}

\subsubsection{Equalizers}
In any topos, if $f,g: X \to Y$ are parallel arrows, then the subobject
$$\{x \in X \mid f(x) = g(x)\} \imto X$$
is their equalizer.
So, if $f,g: (X,\sim) \to (Y,\sim)$ are extensional morphisms, then
the extensional morphism $\id: (X,\sr{r}{\sim}) \to (X,\sim)$,
where the notation $\sr{r}{\sim}$ is from Construction \ref{SubClC[P]} and
$$r(x) \eqq f(x) = g(x),$$
represents the equalizer of the maps given by $f,g$.
Similarly, if $F,G: (X,\sim) \to (Y,\sim)$ are functional predicates, then
the extensional morphism $\id: (X,\sr{r}{\sim}) \to (X,\sim)$, where
$$r(x) \eqq \exists y^Y . \ex(y) \wedge F(x,y) \wedge G(x,y),$$
represents the equalizer of the maps given by $F,G$.

\subsubsection{Pullbacks}
In any topos, if $f: X \to Z$ and $g: Y \to Z$ are arrows, then the subobject
$$\{(x,y) \in X \times Y \mid f(x) = g(x)\} \imto X \times Y$$
is their pullback.
So, if $f: (X,\sim) \to (Z,\sim)$ and $g: (Y,\sim) \to (Z,\sim)$ are extensional $\C$-morphisms,
then
\begin{itemize}
 \item the object $(X \times Y,\sr{r}{\sim})$ with $r(x,y) = \dbl f(x) = g(y) \dbr$, together with
 \item the $\C$-morphisms $\pi_X: (X \times Y,\sr{r}{\sim}) \to (X,\sim)$ and $\pi_Y: (X \times Y,\sr{r}{\sim}) \to (Y,\sim)$
\end{itemize}
represent the pullback of $f$ and $g$.
% So...
% 
% \subsubsection{General limits}
% \textbf{General limits.}
% In any topos $\E$, if $D: I \to \E$ is any locally finite diagram\footnote{i.e. a functor with $I$ locally finite}, then the limit of $D$ is the subobject
% $$\{(x_i) \in \prod_{i \in \Ob(I)} \mid \mathop{\&}_{a \in \Hom_I(i,j)} D(a)(x_i) = x_j\}$$
% together with the canonical projections.
% So...

\subsubsection{Colimits}
\begin{construction}
Let $(X,\sim)$ be an object, and let $\phi(x)$ be a $P$-formula on $X$.
The formula
$$\phi^\st(x) := \ex(x) \wedge \phi(x)$$
is called the \defn{strictization} of $\phi$ [w.r.t. $(X,\sim)$].
We claim of this construction the following properties.
\claim{\begin{itemize}
\item[(a)] The formula $\phi^\st$ is \emph{strict w.r.t. $(X,\sim)$}, i.e.,\\ $P \models \forall x^X . \phi^\st(x) \RI \ex(x)$.
\item[(b)] The formulas $\phi$ and $\phi^\st$ represent the same subobject of $(X,\sim)$, i.e.,\\ $P \models \forall x^X . \ex(x) \RI [\phi(x) \BI \phi^\st(x)]$.
\end{itemize}}
\end{construction}

\begin{proof}
Immediate.
\end{proof}

Let us now learn how to present quotients in $\C[P]$.

\begin{proposition}[Presentation of quotients]
Let $(X,\sim)$ be an object of $\C[P]$.
Let $\dsim \in \P(X \times X)$ \emph{strictly} represent an equivalence relation on $(X,\sim)$.
\it\begin{itemize}
\item[(a)] The predicate $\dsim$ is a partial equivalence predicate on $X$.
\item[(b)] The $\C$-morphism $\id_X: X \to X$ represents a map $(X,\sim) \to (X,\dsim)$, and this map is the quotient map $(X,\sim) \to (X,\sim)/\dsim$.
\end{itemize}\rm
\end{proposition}

\begin{proof}
Straightforward to verify.
\end{proof}

\subsection{Closure transformations and local operators}\label{ssec:closure}
In topos theory we have the notion of geometric morphism between toposes, and it is clearly important to us since we are interested in the `inclusions' into the topos $\Eff$ in the sense of geometric morphisms.
There is an analogous (and related) theory of `geometric morphisms' for triposes, and we review here a portion of this theory in so far as relevant to us.\footnote{One can find a more extensive account in \cite[Section 2.5]{jvo08}. It includes essentially everything we discuss here, but be aware that some presentation of ours (notably Construction \ref{SubtriposConstruction}) deviates from l.c. to suit our purpose.}

Briefly the key points of this subsection are as follows.
We will (i) discuss that the `subtriposes' of a tripos $P$ correspond to the subtoposes of the topos $\C[P]$, (ii) particularly describe a construction of obtaining from a representation $j$ of a local operator on $\C[P]$ a subtripos $P_j$ of $P$ such that $\C[P_j]$ is the corresponding subtopos, and (iii) see how the logic of $P_j$ (hence that of $\C[P_j]$) reduces to the logic of $P$.

Let us first recall what geometric morphisms and subtoposes of toposes were, and define to the counterpart notions for triposes.
A \emph{geometric morphism} $\E \to \F$ of toposes is a pair $(f_*,f^*)$, where $f_*: \E \to \F$ and $f^*: \F \to \E$ are functors such that $f^* \leftadjto f_*$ and moreover $f^*$ preserves finite limits; the functor $f_*$ is called the \emph{direct image} part and $f^*$ the \emph{inverse image} part of the morphism.
A \emph{geometric inclusion} is a geometric morphism whose direct image functor is fully faithful.
A \emph{subtopos} of a topos $\E$ is an isomorphism class of geometric inclusions into $\E$.

\begin{definition}
Let $\C$ be a cartesian closed category.
A \defn{geometric morphism} $P \to Q$ of $\C$-triposes is a pair $(\Phi_+,\Phi^+)$, where $\Phi_+: P \to Q$ and $\Phi^+: Q \to P$ are transformations of triposes such that for each $X \in \Ob(C)$ we have $(\Phi^+)_X \leftadjto (\Phi_+)_X$ and the preorder map $(\Phi^+)_X$ preserves finite meets; we call $\Phi_+$ the \defn{direct image} part and $\Phi^+$ the \defn{inverse image} part.
A \defn{geometric inclusion} of triposes is a geometric morphism whose direct image transformation is componentwise fully faithful (i.e. not only preserves but also reflects the preorder).
A \defn{subtripos} of a $\C$-tripos $P$ is an isomorphism class of geometric inclusions into $P$.
\end{definition}

Then we proceed to the notion of `closure transformation' on a tripos, which is an important ingredient in the matter of presenting subtriposes.
It is in a sense -presentationally as well as mathematically- analogous to local operator on a topos, as we will see.

First we consider a definition, which is an abstraction saving us from repeating the `same' axioms in the coming definitions of local operator on a topos and of closure transformation on a tripos.

\begin{definition}
Let $(P,\leq)$ be a Heyting prealgebra.
A function $j: P \to P$ is a \defn{closure operator} on $P$ if one of the following two (constructively) \claim{equivalent} conditions
$$\begin{array}{rlccrl}
(i)   & \top \leq j(\top)                    &&&   (i') & p \leq j(p) \\
(ii)  & j(p \wedge q) \cong j(p) \wedge j(q) &&&  (ii') & j(p \RI q) \leq (j(p) \RI j(q)) \\
(iii) & j(j(p)) \cong j(p)                   &&& (iii') & j(j(p)) \leq j(p).
\end{array}$$
are satisfied.\footnote{The use of stating two characterizations is obvious: when proving that something is a closure operator we can choose one of the two, and when proving something given a closure operator we can use the properties from both characterizations.}
\end{definition}

\begin{proof}[Proof (of the equivalence)]
First, we consider the `monotonicity' condition
\begin{equation}\label{co1}
(p \RI q) \leq j(p) \RI j(q).
\end{equation}
Clearly, this follows from the `meet-preservation' (ii) as well as from (i') and (ii').
So \eqref{co1} is a consequence of the LHS axioms as well as of the RHS ones.

Let us prove `left to right'.
Using (i), \eqref{co1} and the Heyting implication property, we have $p \cong (\top \RI p) \wedge j(\top) \leq j(p)$; this shows (i').
To show (ii') is to show $j(p \RI q) \wedge j(p) \leq j(q)$.
But indeed, by (ii), \eqref{co1} and the fact $(p \RI q) \wedge p \leq q$, we have $j(p \RI q) \wedge j(p) \cong j((p \RI q) \wedge p) \leq j(q)$ as desired.
Finally, (iii') is immediate from (iii).

Let us prove `right to left'.
First, (i) is a special case of (i').
Let us show (ii).
Since $p \wedge q \leq p$, we have by (i') that $j(p \wedge q) \leq j(p)$. Symmetrically, $j(p \wedge q) \leq j(q)$.
This yields `$\leq$'.
To show `$\geq$' is to show $j(p) \leq j(q) \RI j(p \wedge q)$.
Since $p \leq q \RI (p \wedge q)$, we have by \eqref{co1} and by (ii') that $j(p) \leq j(q \RI (p \wedge q)) \leq j(q) \RI j(p \wedge q)$, as desired.
This shows (ii).
Finally, (iii) is immediate from (i') and (iii').
We are done.
\end{proof}

It is very useful to be aware of the following properties.

\begin{lemma}
Let $(P,\leq)$ be a Heyting prealgebra, and let $j: P \to P$ be a closure operator.
\begin{itemize}
\item[(a)] $p \RI q \leq j(p) \RI j(q)$.
\item[(b)] $j(p \RI j(q)) \leq p \RI j(q)$.
\end{itemize}
(a)
We have already observed this in the previous proof.
(b)
By (ii'), we have $j(p \RI j(q)) \leq j(p) \RI jj(q)$.
By (i') and (iii'), we have $j(p) \RI jj(q) \leq p \RI j(q)$.
The conclusion follows.
\end{lemma}

Now, we can characterize a \emph{local operator}, or \emph{(Lawvere-Tierney) topology}, on a topos $\E$ as a map $\Omega \to \Omega$ that is internally a closure operator on the Heyting algebra $(\Omega,\RI)$.
Recall that given a topology $j$ on a topos $\E$, the full subcategory $\Sh_j(\E)$ of $\E$ underlies a geometric inclusion into $\E$, and that every subtopos of $\E$ arises in this way from a unique topology on $\E$.
Let us see an analogy of this on the side of triposes, in the following Definition and Construction.

\begin{definition}
Let $P$ be a tripos over a cartesian closed category $\C$.
A \defn{closure transformation} on $P$ is a transformation $\Phi: P \to P$ such that for each $X \in \Ob(\C)$ the function $\Phi_X: P(X) \to P(X)$ is a closure operator.
\end{definition}

\begin{construction}\label{SubtriposConstruction}
Let $\C$ be a c.c. category, and let $P$ be a structure-specified $\C$-tripos.

Given a closure transformation $\Phi: P \to P$, we are about to construct a geometric inclusion of a structure-specified tripos into $P$.
For each object $X$ of $\C$, consider a relation $\leq_\Phi$ on $P(X)$ given by $p \leq_\Phi q \iff p \leq \Phi_X(q)$.
We let $P_\Phi(X)$ be the tuple $(P(X),\leq_\Phi,\top_\Phi,\ldots)$ where
\begin{eqnarray*}
\top_\Phi & = & \top \\
\bot_\Phi & = & \bot \\
p \wedge_\Phi q & = & p \wedge q \\
p \vee_\Phi q & = & p \vee q \\
p \RI_\Phi q & = & p \RI \Phi_X(q).
\end{eqnarray*}
For each $\C$-morphism $f: X \to Y$, define functions $P_\Phi(f): P(Y) \to P(X)$ and $\exists^\Phi_f,\forall^\Phi_f: P(X) \to P(Y)$ by
\begin{eqnarray*}
P_\Phi(f)      & = & P(f) \\
\exists^\Phi_f & = & \exists_f \\
\forall^\Phi_f & = & \forall_f \circ \Phi_X.
\end{eqnarray*}
Finally, let $(\Sigma_\Phi,\sigma_\Phi) = (\Sigma,\sigma)$.
\claim{Then these data $(P_\Phi,\exists^\Phi,\forall^\Phi,\Sigma_\Phi,\sigma_\Phi)$ form a structure-specified tripos, and the pair $(\Phi,\Id)$ is a geometric inclusion $P_\Phi \imto P$.}
\end{construction}

\begin{proof}
Routine verifications. (This construction is essentially \cite[Remark 5.2(ii)]{pitts81}. A different construction, that results in the equivalent inclusion, is described in \cite[p.96-97]{jvo08}.)
\end{proof}

\begin{construction}\label{con:PjToP}
Let $j: \Sigma \to \Sigma$ be a topology-representing $\C$-morphism.
We associate to a $P_j$-formula $\phi$ a $P$-formula $\el^{P/P_j} \phi \er$, as in the following clauses.\footnote{Compare these to the clauses in the previous Construction.}
\begin{itemize}
\item Given $r \in P(\_)$, we define $\el^{P/P_j} r \er = r$.
\item $\el^{P/P_j} \phi \wedge \psi \er = \el^{P/P_j} \phi \er \wedge \el^{P/P_j} \psi \er$
\item $\el^{P/P_j} \phi \vee \psi \er = \el^{P/P_j} \phi \er \vee \el^{P/P_j} \psi \er$
\item $\el^{P/P_j} \phi \RI \psi \er = \el^{P/P_j} \phi \er \RI j(\el^{P/P_j} \psi \er)$
\item $\el^{P/P_j} \forall x^X . \phi(x,y) \er = \forall x^X . j(\el^{P/P_j} \phi \er)$
\item $\el^{P/P_j} \exists x^X . \phi(x,y) \er = \exists x^X . \el^{P/P_j} \phi \er$
\end{itemize}
Note that if $\phi$ is a sentence, then so is $\el^{P/P_j} \phi \er$.
\claim{We have $P_\Phi \models \phi$ iff $P \models j(\el^{P/P_j} \phi \er)$.}
\end{construction}

\begin{proof}
For clarity, let us use the notation $\dbl^{P_j} \cdots \dbr$ for the interpretations of $P_j$-formulas and $\dbl^P \cdots \dbr$ for the interpretations of $P$-formulas.

Clearly we have $\dbl^{P_j} \phi \dbr = \dbl^P \el^{P/P_j} \phi \er \dbr$ by the previous Construction.
Therefore: $P_j \models \phi$ iff $\top_P = \top_j \leq_j \dbl^{P_j} \phi \dbr$ iff $\top_P \leq_P j \dbl^{P_j} \phi \dbr = j \dbl^P \el^{P/P_j} \phi \er \dbr = \dbl^P j(\el^{P/P_j} \phi \er) \dbr$ iff $P \models j(\el^{P/P_j} \phi \er)$, as desired.
\end{proof}

\begin{proposition}
Consider the mutual mappings
\begin{eqnarray*}
\text{closure transformations on $P$} & \tofrom   & \text{geometric inclusions into $P$} \\
\Phi                                  & \mapsto   & P_\Phi \xto{(\Phi,\Id)} P \\
\Phi_+ \circ \Phi^+                   & \mapsfrom & (\Phi_+,\Phi^+).
\end{eqnarray*}
\emph{These comprise an equivalence of preorders.}
Hence a subtripos of $P$ is precisely an isomorphism class of closure transformations on $P$.
\end{proposition}

\begin{proof}
Straightforward verifications. (Cf. \cite[Theorem 2.5.11(ii) on p.97]{jvo08}.)
\end{proof}

\begin{construction}
Let $\C$ be a c.c. category, and let $P$ be a $\C$-tripos.

Given a $\C$-morphism $j: \Sigma \to \Sigma$ representing a local operator on $\C[P]$, we are about to construct a closure transformation on $P$.
For each $X \in \Ob(C)$, consider the association
\begin{eqnarray*}
P(X) & \to     & P(X) \\
r    & \mapsto & j \circ r.
\end{eqnarray*}
\claim{This constitutes a closure transformation on $P$.}
We will denote this closure transformation just with $j$.
\end{construction}

\begin{proof}
Straightforward verifications. (Cf. \cite[Remark 5.2(i)]{pitts81} or \cite[p.96]{jvo08}.)
\end{proof}

\begin{proposition}\label{LoreClorEqv}
\claim{We have an equivalence of preorders
\begin{eqnarray*}
\text{\begin{tabular}{c} $\C$-morphisms $\Sigma \to \Sigma$ \\ representing a topology on $\C[P]$ \end{tabular}} & \tofrom & \text{closure transformations on $\P$} \\
j                & \mapsto   & j \\
\Phi_\Sigma(\id) & \mapsfrom & \Phi.
\end{eqnarray*}
}Hence, under these associations, a local operator on $\C[P]$ is precisely an isomorphism class of closure transformations on $P$.
\end{proposition}

\begin{proof}
Straightforward verifications.
\end{proof}

\begin{construction}
Let $(\Phi_+,\Phi^+): P \to Q$ be a geometric morphism of $\C$-triposes.
\it\begin{itemize}
\item[(a)] We have a functor
\begin{eqnarray*}
\Phi^*: \C[Q] & \to     & \C[P] \\
     (X,\sim) & \mapsto & (X,\Phi^+(\sim)) \\
      F        & \mapsto & \Phi^+(F).
\end{eqnarray*}

\item[(b)] If a $\C$-morphism $f: X \to Y$ represents a map $[F]: (X,\sim) \to (Y,\sim)$ in $\C[Q]$, then $f$ represents the map $[\Phi^+(F)]: (X,\Phi^+(\sim)) \to (Y,\Phi^+(\sim))$ in $\C[P]$.

\item[(c)] The functor $\Phi^*$ preserves finite limits.

\item[(d)] The functor $\Phi^*$ has a right adjoint.
\end{itemize}\rm
We shall denote an arbitrary\footnote{In \cite{jvo08}, a specific construction of the right adjoint functor is given. In this thesis we do not need it.} right adjoint of $\Phi^*$.
So the pair $(\Phi_*,\Phi^*)$ is a geometric morphism $\C[P] \to \C[Q]$.
\end{construction}

\begin{proof}
(a) See \cite[the Remark below Definition 2.5.5]{jvo08}.
(b) Left to the reader.
(c) This is \cite[Proposition 2.5.6]{jvo08}.
(d) This is \cite[Theorem 2.5.8(i)]{jvo08}.
\end{proof}

\begin{proposition}
Let $\C$ be a c.c. category, let $P$ be a tripos over $\C$ and let $j: \Sigma \to \Sigma$ represent a local operator on $\C[P]$.

\claim{We have an equivalence of categories
\begin{eqnarray*}
\C[P_j]  & \eqv      & \Sh_j(\C[P]) \\
(X,\sim) & \mapsto   & (X,j \circ \sim) \\
F        &           & j \circ F \\
(X,\sim) & \mapsfrom & (X,\sim) \\
F        &           & F.
\end{eqnarray*}
This equivalence underlies that the square
\begin{diagram}[LaTeXeqno]\label{FourPartyCorr}
\text{\begin{tabular}{c} $\C$-morphisms $\Sigma \to \Sigma$ \\ representing a topology on $\C[P]$ \end{tabular}}     & \rTo{j \mapsto \Sh_j(\E)} & \text{geometric inclusions into $\C[P]$} \\
\dTo{j \mapsto j}                     &      & \uTo{(\Phi_+,\Phi^+) \mapsto (\Phi_*,\Phi^*)} \\
\text{closure transformations on $P$} & \rTo{\Phi \mapsto P_\Phi} & \text{geometric inclusions into $P$}
\end{diagram}
commutes up to equivalence.}
\end{proposition}

\begin{proof}
Straightforward verifications.
\end{proof}

\begin{cor}
The notions
\begin{itemize}
 \item subtripos of $P$ (i.e. isomorphism class of geometric inclusions into $P$)
 \item isomorphism class of closure transformations on $P$
 \item local operator on $\C[P]$
 \item subtopos of $\C[P]$ (i.e. isomorphism class of geometric inclusions into $\C[P]$)
\end{itemize}
are the same \emph{under the associations} specified at \eqref{FourPartyCorr}.\footnote{Of course, we knew already since Proposition \ref{LoreClorEqv} that these notions are the same. The point of this corollary is that it is so under the associations \eqref{FourPartyCorr}.}
\end{cor}

\begin{proof}
Clear.
\end{proof}

\section{The effective topos}
\begin{definition}
The \defn{effective topos} is the topos $\Eff := \Set[\EffTrip]$.
\end{definition}

\subsection{The functor $\nabla: \Set \to \Eff$}
\begin{notation}\label{NablaNotation}
Let $U \subseteq X$ be an inclusion of sets.
We denote by $\nabla_U$ the function $X \to \PN$ defined by
$$\nabla_U(x) =
\begin{cases}
\N        & \textif x \in U \\
\emptyset & \otherwise.
\end{cases}$$
We also write $\sr{\nabla}{U}$ instead of $\nabla_U$.
\end{notation}

\begin{definition}
\claim{The associations
\begin{eqnarray*}
\Set        & \to     & \Eff \\
X           & \mapsto & (X,\nabeq) \\
X \xto{f} Y & \mapsto & (X,\nabeq) \xto{f} (Y,\nabeq).  
\end{eqnarray*}
comprise a functor.}
We will denote this functor by $\nabla$.
\end{definition}

\begin{proof}
Given a set $X$, the function $\nabeq \in \PN^{X \times X}$ is clearly a partial equivalence predicate on $X$.
Any function $f: X \to Y$ clearly represents a map $(X,\nabeq) \to (Y,\nabeq)$.
Given a set $X$, we know that the identity function $\id_X$ on $X$ represents the identity map on $(X,\sim)$.
Finally, given functions $f: X \to Y$ and $g: Y \to Z$, we know that the function composite $X \xto{f} Y \xto{g} Z$ represents the map composite $(X,\nabeq) \xto{f} (Y,\nabeq) \xto{g} (Z,\nabeq)$.
This verifies that we have a functor.
\end{proof}

\begin{proposition}
We consider the functor $\nabla: \Set \to \Eff$, and the global sections functor $\Gamma: \Eff \to \Set$.
The pair $(\nabla,\Gamma)$ is a geometric inclusion $\Set \emto \Eff$, and this subtopos corresponds to the topology $\negneg$.
\end{proposition}

\begin{proof}
Left to the reader. (Cf. \cite{jvo08}.)
\end{proof}

\begin{notation}
By the previous Proposition, we have canonical bijections
\begin{equation}\label{PointNota}
X \cong \Set(1,X) \cong \Set(\Gamma(1),X) \cong \Eff(1,\nabla(X)).
\end{equation}
Let $X$ be a set, and let $x \in X$ an element.
We will denote the point of the object $\nabla(X)$ corresponding $x$ under the bijection \eqref{PointNota} just by $x$.
\end{notation}
% 
% \begin{proposition}
% Let $\negneg: \PN \to \PN$ be the function $\nabla_{\PsN \subseteq \PN}$, i.e. the function given by
% $$\negneg(A) =
% \begin{cases}
% \N        & \textif A \neq \emptyset \\
% \emptyset & \textif A = \emptyset.
% \end{cases}$$
% \claim{The function $\negneg$ represents the map $\negneg$ in $\Eff$.}
% \end{proposition}
% 
% \begin{proof}
% 
% \end{proof}

\subsection{The natural numbers object}
\begin{proposition-notation}
Let $(X,E)$ be a preassembly.
\claim{Then the function $\sim: X \to \PN$ defined by
$$x \sim x' = \begin{cases}
E(x)      & \textif x = x' \\
\emptyset & \otherwise
\end{cases}$$
is a partial equivalence predicate.}
We may view the preassembly $(X,E)$ as the object $(X,\sim)$ in $\Eff$. \qed
\end{proposition-notation}

\begin{proposition-notation}
\claim{The assembly $\N$, the element $0 \in \N$ and the successor function $\N \to \N$ give a natural numbers object in $\Eff$.}
We denote by $\NNO$ the object in $\Eff$ associated to the assembly $\N$. \qed
\end{proposition-notation}

\subsection{The toposes $\Eff_j$}
\begin{notation}
Let $j: \PN \to \PN$ be a function representing a topology on $\Eff$.
We denote by $\Eff_j$ the topos $\Set[\EffTrip_j]$.
\end{notation}

\begin{proposition}
Let $j: \PN \to \PN$ be a function representing a \emph{non-degenerate} topology on $\Eff$.

\it\begin{itemize}
\item[(i)]
Let $\phi$ be an $\EffTrip$-\emph{sentence}.
Then $\EffTrip \models j(\phi)$ if and only if $\EffTrip \models \phi$.

\item[(ii)]
Let $\phi$ be an $\Eff$-\emph{sentence}.
Then $\Eff \models j(\phi)$ if and only if $\Eff \models \phi$.
\end{itemize}\rm
\end{proposition}

\begin{proof}
(i)
Since $j$ represents a non-degenerate topology, we have for $p \in \PN$ that
\begin{center}
$j(p) = \emptyset$ if and only if $p = \emptyset$.
\end{center}
Therefore: $\EffTrip \models j(\phi)$ iff $\dbl j(\phi) \dbr \neq \emptyset$ iff $\dbl \phi \dbr \neq \emptyset$ iff $\EffTrip \models \phi$, as desired.

(ii)
Indeed: $\Eff \models j(\phi)$ iff $\EffTrip \models \dbl_\Eff j(\phi) \dbr$ iff $\EffTrip \models j \dbl_\Eff \phi \dbr$ iff\footnote{by (i)} $\EffTrip \models \dbl_\Eff \phi \dbr$ iff $\Eff \models \phi$, as desired.
\end{proof}

From this we obtain a simplication of Corollary \ref{con:PjToP} in the case of \emph{sentences} and \emph{non-degenerate} subtoposes of $\Eff$, as follows.
Recall the notation `$\el^{P/P_j} \phi \er$' from that Corollary.

\begin{corollary}\label{cor:NonDegSen}
Let $j: \Sigma \to \Sigma$ be a function representing a \emph{non-degenerate} topology on $\Eff$.
Let $\phi$ be a $\EffTrip_j$-\emph{sentence}.
Then we have $\EffTrip_j \models \phi$ iff $\EffTrip \models \el^{P/P_j} \phi \er$.
\end{corollary}

\begin{proof}
By Corollary \ref{con:PjToP}, we have $\EffTrip_j \models \phi$ if and only if $\EffTrip \models j(\el^{P/P_j} \phi \er)$.
By (i) of the previous Proposition, we have $\EffTrip \models j(\el^{P/P_j} \phi \er)$ if and only if $\EffTrip \models \el^{P/P_j} \phi \er$.
The conclusion follows.
\end{proof}

Let us derive a simple presentation of the natural numbers object in the topos $\Eff_j$.

\begin{proposition-notation}
Let $j: \PN \to \PN$ be a function representing a local operator on $\Eff$.
\claim{The pair $(\N,\{\cdot\})$, the element $0 \in \N$ and the successor function $\N \to \N$ give a natural numbers object in $\Eff_j$.}
We denote by $\NNO$ the object $(\N,\{\cdot\})$ in $\Eff_j$.
\end{proposition-notation}

\begin{proof}
We learned from Subsection \ref{ssec:closure} that the functor
\begin{eqnarray*}
\Id^*: \Eff & \to     & \Eff_j   \\
   (X,\sim) & \mapsto & (X,\sim) \\
          F & \mapsto & F        \\
		  f & \mapsto & f
\end{eqnarray*}
is the inverse image part of a geometric inclusion $\Eff_j \emto \Eff$ corresponding to the topology $j$.
We know that the inverse image part of a geometric inclusion between toposes preserve natural numbers objects.
Applying the functor $\Id^*$ to our natural numbers object $N$ in $\Eff$, the result follows.
\end{proof}

\chapter{General constructions of local operators}\label{chap:lop}
There are several constructions of local operator, which one can apply for any topos.
% (Here I do not include the order-theoretic operations in the Heyting algebra of local operators. Or shall I?)
In this section we shall recall them and instantiate in the case of the effective topos.

Given a topos $\E$, We write $\Lop(\E)$ for the class of local operators on $\E$.
%Given $D \in \Sub(\Omega)$, we let $\Xi(D)$ be the class of monomorphisms in $\E$ whose characteristic maps factor through $D \imto \Omega$.
Internally in a topos, if $A \subseteq B$ is an inclusion of sets, we may denote by $\chi_{A \subseteq B}$ or just by $\chi_A$ the characteristic function $B \to \Omega: x \mapsto (x \in A)$.

% 
% 
% \section{On the lattice of local operators}

\section{The open and closed topologies}
We have the following two constructions of topology arising from a subobject of $1$.

\begin{construction}
Let $p_0$ be a subobject of $1$ in a topos $\mathcal{E}$.
The \defn{open} topology of $q$ is the map
$$\Omega \to \Omega: p \mapsto p_0 \RI p,$$
and the \defn{closed} topology of $U$ is the map
$$\Omega \to \Omega: p \mapsto p_0 \vee p.$$
One easily proves by logic that these maps are indeed topologies.
\end{construction}

\begin{example}
Clearly,
\begin{itemize}
 \item the open topology of $\true$ is the map $\id: p \mapsto p$,
 \item the open topology of $\false$ is the map $\true\circ!: p \mapsto \top$,
 \item the closed topology of $\true$ is the map $\true\circ!: p \mapsto \top$,
 \item the closed topology of $\false$ is the map $\id: p \mapsto p$.
\end{itemize}
So when $\mathcal{E}$ is two-valued, just like in the case of the effective topos,
these constructions do not give a interesting topology; the topology $\id$ is the
least topology giving $\mathcal{E}$ itself as the corresponding subtopos,
and $\true\circ!$ is the largest topology giving the degenerate topos.
\end{example}

\section{Joyal's construction}\label{sec:Joyal}
We now discuss a construction of the least local operator that makes some subobject dense.
First, let us introduce some affable `internal' termonologies.

\begin{definition}
Internally in a topos $\E$, we may call elements of $\Omega$ \defn{truth values}.
Given a local operator $j$, we call a truth value $p$
\begin{itemize}
 \item \defn{$j$-closed} if $j(p)$ implies $p$, and
 \item \defn{$j$-dense} if $j(p)$ holds.
\end{itemize}
\end{definition}

\begin{notation}
Let us recall a number of standard notations in topos theory.
Let $\E$ be a topos, and let $j$ be a local operator on $\E$.
We write $J$ for the subobject of $\Omega$ characterized by $j: \Omega \to \Omega$,
and we write $\Omega_j$ for the subobject of $\Omega$ given as the equalizer of $j,\id: \Omega \to \Omega$.
So saying internally, $J = \{p \in \Omega \mid j(p)\}$ is the set of $j$-dense truth values,
and $\Omega_j = \{p \in \Omega \mid j(p) = p\}$ is the set of $j$-closed truth values.
\end{notation}

\begin{proposition}
Internally in a topos, if $D$ is a subset of $X$ and $\chi_D$ is its characteristic function $X \to \Omega$, then
\begin{itemize}
 \item[(a)] $D$ is $j$-dense if and only if $\chi_D$ factors via $J$, and
 \item[(b)] $D$ is $j$-closed if and only if $\chi_D$ factors via $\Omega_j$.
\end{itemize}
We also have the following.
\begin{itemize}
 \item[(c)] The object $\P_j(X) = \{A \in \P(X) \mid \forall x \in X . j(x \in A) \RI x \in A\}$ of (saying internally) $j$-closed subsets of $X$ is the object $\Omega_j^X$, in the natural sense that the canonical isomorphism $\P(X) \cong \Omega^X$ restricts to an isomorphism $\P_j(X) \cong \Omega_j^X$ (clearly, the inclusion $\Omega_j \imto \Omega$ induces an inclusion $\Omega_j^X \imto \Omega^X$; in this way we view $\Omega_j^X$ as a subobject of $\Omega^X$).
 \item[(d)] The object $\Pdense{j}(X) = \{A \in \P(X) \mid \forall x \in X . j(x \in A)\}$ of the $j$-dense subsets of $X$ is the object $J^X$, in the natural sense that the canonical isomorphism $\P(X) \cong \Omega^X$ restricts to an isomorphism $\Pdense{j}(X) \cong J^X$.
\end{itemize}
\end{proposition}

\begin{proof}
Left to the reader.
\end{proof}

\begin{definition}
Internally in a topos, let $D$ be a subset of $\Omega$.
We say that $D$ is \defn{upwards closed} if for all $p \in D$ and $q \in \Omega$ if $p \leq q$ then $q \in D$, i.e. if the characteristic function $\chi_D: \Omega \to \Omega$ of $D$ is monotone.
\end{definition}

\begin{construction}
Internally in a topos, let $D$ be a subset of $\Omega$.
Let
$$D^\mo = \{q \in \Omega \mid \exists p \in \Omega . p \in D \wedge (p \RI q)\},$$
which is clearly (think internally!) the least upwards closed subset of $\Omega$ that includes $D$.
We will employ this notation $\ph^\mo$ also for functions $\Omega \to \Omega$ in the obvious manner, namely via the canonical correspondence $\P(\Omega) \cong \Omega^\Omega$.
\end{construction}

The main ingredient of our goal construction is the following result by A. Joyal.

\begin{proposition}
Internally, let $D$ be a subset of $\Omega$.
Let
\begin{align*}
D^\r
&= \{q^\Omega \mid \forall p^\Omega . p \in D \RI [(p \RI q) \RI q]\} \\
&= \{q^\Omega \mid \forall p^\Omega . p \in D \wedge (p \RI q) \RI q\} \\
&= \{q^\Omega \mid [\exists p^\Omega . p \in D \wedge (p \RI q)] \RI q\} \\
&= \{q^\Omega \mid q \in D^\mo \RI q\} \\
&\in \Sub(\Omega),
\end{align*}
and let
\begin{align*}
D^\l
&= \{p^\Omega \mid \forall q^\Omega . q \in D \RI [(p \RI q) \RI q]\} \\
&= \{p^\Omega \mid \forall q^\Omega . q \in D \wedge (p \RI q) \RI q\} \\
&\in \Sub(\Omega).
\end{align*}
Then
\begin{itemize}
 \item[(i)] The mappings $D \mapsto D^\r$ and $D \mapsto D^\l$ form a Galois connection from $\P(\Omega)$ to itself,
in the sense that they are (1) order-reversing functions $\P(\Omega) \to \P(\Omega)$ and (2) adjoint to each other on the right (i.e. $D \leq D^\rl$ and $D \leq D^\lr$ for all $D \in \P(\Omega)$).
 \item[(ii)] If $j: \Omega \to \Omega$ is a closure operator, then $J^\r = \Omega_j$ and $(\Omega_j)^\l = J$.
 \item[(iii)] A subobject $D \in \P(\Omega)$ satisfies $D^\rl = D$ iff its characteristic map is a closure operator.
\end{itemize}
Notice that the external variant of these statements immediately follow:
\begin{itemize}
 \item[(i')] The mappings $D \mapsto D^\r$ and $D \mapsto D^\l$ form a Galois connection from $\Sub(\Omega)$ to itself,
in the sense that they are (1) order-reversing functions $\Sub(\Omega) \to \Sub(\Omega)$ and (2) adjoint to each other on the right (i.e. $D \leq D^\rl$ and $D \leq D^\lr$ for all $D \in \Sub(\Omega)$).
 \item[(ii')] If $j$ is a local operator on $\E$, then $J^\r = \Omega_j$ and $(\Omega_j)^\l = J$.
 \item[(iii')] A subobject $D \in \Sub(\Omega)$ satisfies $D^\rl = D$ iff its characteristic map is a local operator on $\E$.
\end{itemize}
\end{proposition}

\begin{proof}
Left to the reader.
The external variant is precisely \cite[Theorem 3.57]{joh77} (or \cite[Proposition 4.5.12]{ele12}).
\end{proof}

\begin{corollary}\label{GeneralLo}
Internally, let $D \subseteq \Omega$.
We write $D^\loar = D^\rl$, and write
$$D^\lomo = \{p \mid \forall q . (q \in D^\mo \RI q) \wedge (p \RI q) \RI q\}.$$
Then:
\begin{itemize}
 \item $\ph^\loar = \ph^\lomo \circ \ph^\mo$.
 \item $D^\loar$ corresponds to the least closure operator $j$ with $D \leq J$.
 \item If $D$ is monotone, then $D^\lomo$ is the least closure operator $\geq D$. \qed
\end{itemize}
\end{corollary}

An immediate consequence of this corollary is that $\Lop(\E)$ has binary joins:

\begin{proposition}
If $j_1,j_2$ are local operators on $\E$, then the least local operator $j$ with $J_1 \cup J_2 \leq J$ is the join $j_1 \vee j_2$.%\footnote{I believe that [Elephant, Example A4.5.14(a)], which is about the same thing as the present point, is misleading: it describes the join $j_1 \vee j_2$ as the least local operator $j$ for which all monos in $\Xi(J_1 \cup J_2)$ are $j$-dense. With the assumption (which derives from the proof of [Elephant, Corollary 4.5.13(i)]) that such a $j$ is the characteristic map of $(J_1 \cup J_2)^\rl$, this is indeed true. But (1) clearly then $j_1 \vee j_2$ shouldn't be described like that, and (2) I am actually not convinced of the truth of this assumption itself (META: help me?).}
\end{proposition}

\begin{proof}
As $J_1 \leq J_1 \cup J_2 \leq J$, we have $j_1 \leq j_1 \vee j_2$ and symmetrically $j_2 \leq j_1 \vee j_2$.
If $k$ is a local operator with $j_1,j_2 \leq k$, i.e. $J_1 \cup J_2 \leq K$, then the leastness of $j$ yields $j \leq k$.
This proves that $j$ is the join of $j_1$ and $j_2$.
\end{proof}

\begin{lemma}
Let $j$ be a local operator on a topos $\E$, and let $D$ be a subobject of $\Omega$.
The following are equivalent.
\begin{itemize}
 \item $D$ is $j$-dense.
 \item $\Im(\chi_D) \leq J$.
\end{itemize}
\end{lemma}

\begin{proof}
Elementary topos theory.
\end{proof}

\begin{construction}\label{LeastDense}
Internally in a topos, if $A \subseteq B$ is an inclusion of sets, then there is the least local operator $j$ for which the inclusion $A \subseteq B$ is $j$-dense: namely, the least local operator $j$ with $\Im(\chi_{A \subseteq B}) \subseteq J$. (We use Corollary \ref{GeneralLo}.)

For a logical description of this local operator, we write
\begin{align*}
\Im(\chi_{A \subseteq B}) &= \{p \in \Omega \mid \exists x \in B . p = \chi_{A \subseteq B}(x)\} \\
&= \{p \in \Omega \mid \exists x \in B . p \BI x \in A\},
\end{align*}
so that
\begin{align*}
\Im(\chi_{A \subseteq B})^\mo &= \{q \in \Omega \mid \exists p \in \Omega . p \in \Im(\chi_{A \subseteq B}) \wedge (p \RI q)\} \\
&= \{q \in \Omega \mid \exists p \in \Omega . \exists x \in B (p \BI x \in A) \wedge (p \RI q)\} \\
&= \{q \in \Omega \mid \exists x \in B . x \in A \RI q\}.
\end{align*}
\end{construction}
% 
% \begin{proof}
% (a) If $k$ is a local operator for which $U$ is $k$-dense, then $\Im(\chi_U) \leq K$, so $j \leq k$ by the leastness of $j$.
% \end{proof}

% 
% 
% \section{Forcing an object to be a sheaf}
% Proposition A4.5.15 of [Johnstone, Elephants] proves, providing a construction, that for any object $X$ of a topos there is the largest topology $j$ for which $X$ is a $j$-sheaf.
% As he mentions there, this doesn't seem to follow from the general construction of a largest topology which makes a class of monomorphisms $j$-closed.\footnote{FIX: At the momoent of writing this sentence, I have not written about that general construction yet. So take care.}
% The next point is the result of extracting a logical description of this $j$ from the construction in l.c.
% 
% \begin{point}
% Let $X$ be an object in a topos.
% The map $j: \Omega \to \Omega$ that assigns
% $$p \mapsto \forall s \in \tilde{A} . (\exists a \in A (a \in s) \bi p) \ri \forall a,a' \in A (a \in s \bi a' \in s)$$
% is the largest topology for which $X$ is a $j$-sheaf. (META: I'm still working on this, even the formula above might not be correct.)
% \end{point}

\begin{proposition}\label{EveryTVset}
Internally in a topos, every subset of $\Omega$ is the image of a characteristic function.
\end{proposition}

\begin{proof}
Let $D$ be a subset of $\Omega$.
Let $c: D \to \Omega$ be the characteristic function of the inclusion $D \cap \{p \in \Omega \mid p\} \subseteq D$.
Then
\begin{align*}
\Im(c) &= \{p \in \Omega \mid \exists d \in D . p = c(d)\} \\
&= \{p \in \Omega \mid \exists d \in D . p = (d \in D \cap \{p \in \Omega \mid p\})\} \\
&= \{p \in \Omega \mid \exists d \in D . p = d\} \\
&= \{p \in \Omega \mid p \in D\} = D,
\end{align*}
as desired.
\end{proof}

\begin{corollary}\label{cor:EveryLoc}
Internally in a topos, every local operator is the least one for which some inclusion of sets (in fact an inclusion into $\Omega$) is dense. \qed
\end{corollary}

\chapter{Subtoposes of $\Eff$}
\section{Representation by sequences of collections of number sets}
In our treatment, we will present examples of subtoposes of $\Eff$ as infinite sequences of collections of natural number sets (i.e. as functions $\N \to \PPN$).
This section is devoted to explaining how we can represent subtoposes as such sequences, discussing some properties of the representation and describing relevant (pre)orderings on the representatives.

\subsection{Objects of internal functions $\Omega \to \Omega$}
% In order now to proceed the story internally,
We present here three objects of internal functions $\Omega \to \Omega$ which we will deal with.
Recall the notation `$(X,\sr{r}{\sim})$' from Construction \ref{SubClC[P]}(d).

\begin{proposition}
The object $(\PN^\PN,\sr{\ext}{\BI})$, together with the obvious evaluation map, is the exponential $\Omega^\Omega$.
\end{proposition}

\begin{proof}
By Remark \ref{rmk:OmegaPowerSimple}, the object $(\PN^\PN,\sim)$ and the obvious evaluation map, where
$$h \sim h' = \ext(h) \wedge \forall p [(p \BI p) \RI h(p) \BI h'(p)],$$
are the exponential $\Omega^\Omega$.
But clearly, the predicate $\sr{\ext}{\BI}: \PN^\PN \times \PN^\PN \to \PN$ is isomorphic to $\sim$.
The claim follows.
\end{proof}

\begin{definition}
Let $\C$ be a cartesian closed category, and let $P$ be a $\C$-tripos.
Given a term $h: \Sigma^\Sigma$ in the language of $P$, we define a $P$-formula
$$\mon(h) := \forall p,q^\Sigma . (p \RI q) \RI (h(p) \RI h(q)).$$
Let $h: \Sigma \to \Sigma$ be a morphism in $\C$.
We say that $h$ is \defn{monotone} if $P \models \mon(h)$.
Note that $h$ is monotone if and only if it represents a monotone map $\Omega \to \Omega$.
\end{definition}

\begin{notation-proposition}
Write $\Mo = (\PN^\PN,\sr{\mon}{\BI})$.
The function $\id: \PN^\PN \to \PN^\PN$ represents a mono $\Mo \to \Omega^\Omega$, and this mono gives the subobject of internal monotone functions $\Omega \to \Omega$.
\end{notation-proposition}

\begin{proof}
By Construction \ref{SubClC[P]}(d), the mono $\id: (\PN^\PN,\sr{\ext,\mon}{\BI}) \to (\PN^\PN,\sr{\ext}{\BI})$ gives the desired subobject.\footnote{Here, the notation `$\sr{\ext,\mon}{\BI}$' should be read as putting `$\mon$' above `$\sr{\ext}{\BI}$'.}
But clearly, the predicate $\sr{\mon}{\BI}$ is isomorphic to $\sr{\ext,\mon}{\BI}$.
The claim follows.
\end{proof}

\begin{notation-proposition}
Write $\Lo = (\PN^\PN,\sr{\lop}{\BI})$.
The function $\id: \PN^\PN \to \PN^\PN$ represents a mono $\Lo \to \Mo$, and this mono gives the subobject of internal local operators $\Omega \to \Omega$.
\end{notation-proposition}

\begin{proof}
By Construction \ref{SubClC[P]}(d), the mono $\id: (\PN^\PN,\sr{\mon,\lop}{\BI}) \to (\PN^\PN,\sr{\mon}{\BI})$ gives the desired subobject.\footnote{Here, $\lop: \PN^\PN \to \PN$ is a predicate representing the subobject $\Lo \subseteq \Omega^\Omega$.}
But clearly, the predicate $\sr{\lop}{\BI}$ is isomorphic to $\sr{\mon,\lop}{\BI}$.
The claim follows.
\end{proof}

\begin{remark}
Recall from Section \ref{sec:Joyal} that we have a map $\ph^\mo:\Omega^\Omega \onto \Mo$, which is left adjoint to the inclusion $\Mo \imto \Omega^\Omega$.
Likewise we have a map $\ph^\lomo: \Mo \onto \Lo$, and it is left adjoint to the inclusion $\Lo \imto \Mo$.
Finally we have a map $\ph^\loar: \Omega^\Omega \onto \Lo$, which is the composite $\ph^\lomo \circ \ph^\mo: \Omega^\Omega \onto \Lo$.
\end{remark}

\subsection{NNO-indexed joins in $\Omega^\Omega$ and $\Mo$}
We will define a map $\gr_\ph: \nabla(\PPN) \to \Omega^\Omega$, and show that internally every function $\Omega \to \Omega$ is an $\NNO$-indexed join of functions in the image of this map.

\begin{proposition}
Given $\A \in \PPN$, define a function $\gr_\A: \PN \to \PN$ by
$$\gr_\A(p) = \exists_{A \in \A} (A \BI p).$$
\claim{The function $\gr_\ph: \PPN \to \PN^\PN$ represents a map $\nabla(\PPN) \to \Omega^\Omega$.}
\end{proposition}

\begin{proof}
We have to verify
$$\realizes \forall \A,\B \in \PPN . \A \nabeq \B \RI \gr_\A \sr{\ext}{\BI} \gr_\B,$$
equivalently that
$$\realizes \forall \A \in \PPN . \ext(\gr_\A),$$
i.e. that
$$\realizes \forall \A \in \PPN \forall p,q \in \PN . [p \BI q] \RI [\gr_\A(p) \BI \gr_\A(q)],$$
i.e. that
$$\realizes \forall \A \in \PPN \forall p,q \in \PN . [p \BI q] \RI [\exists_{A \in \A} (A \BI p) \BI \exists_{A \in \A} (A \BI q)].$$
But the last one is evident.
\end{proof}

\begin{proposition}\label{prop:AllAreJoin}
Internally in $\Eff$, every function $\Omega \to \Omega$ is an $\NNO$-indexed join of functions
in the image of the map $\gr_\ph: \nabla(\PPN) \to \Omega^\Omega$.
\end{proposition}

\begin{proof}
That is, we prove
\begin{eqnarray*}
\Eff & \models & \forall h \in \Omega^\Omega \exists \theta \in \nabla(\PPN)^N . h = \bigvee_{n \in N} \gr_{\theta(n)}   \\
     & \eqq    & \forall h \exists \theta \forall p \in \Omega . h(p) = (\bigvee_{n \in N} \gr_{\theta(n)})(p)           \\
     & \eqq    & \forall h \exists \theta \forall p . h(p) \BI \exists n \in N [\gr_{\theta(n)}(p)],
\end{eqnarray*}
i.e. that
\begin{equation}\label{aaj1}
\realizes \forall h \in \PN^\PN . \ext(h) \RI \exists \theta \forall p . h(p) \BI \exists n \in \N [\{n\} \wedge \exists_{A \in \theta(n)} (A \BI p)].
\end{equation}
Given $h \in \PN^\PN$, define a function $h^*: \N \to \PPN$ by
$$h^*(n) = \{A \in \PN \mid n \in \dbl \forall q . (A \BI q) \RI h(q) \dbr\}.$$
Then we have the realization
\begin{align*}
& \lambda e . \lambda a . \la \lambda x . e(x)_1(a) , \la \id,\id \ra \ra \\
& \realizes \forall h \in \PN^\PN . \ext(h) \RI \forall p . h(p) \RI \exists n \in \N [\{n\} \wedge \exists_{A \in \PN, n \in \dbl \forall q . (A \BI q) \RI h(q) \dbr} (A \BI p)]
\\
& = \forall h \in \PN^\PN . \ext(h) \RI \forall p . h(p) \RI \exists n \in \N [\{n\} \wedge \exists_{A \in h^*(n)} (A \BI p)]
\end{align*}
and `conversely' the realization
\begin{align*}
& \lambda e . \lambda \la n,b \ra . n(b) \\
& \realizes \forall h \in \PN^\PN . \ext(h) \RI \forall p . h(p) \LI \exists n \in \N [\{n\} \wedge \exists_{A \in \PN, n \in \dbl \forall q . (A \BI q) \RI h(q) \dbr} (A \BI p)]
\\
& = \forall h \in \PN^\PN . \ext(h) \RI \forall p . h(p) \LI \exists n \in \N [\{n\} \wedge \exists_{A \in h^*(n)} (A \BI p)].
\end{align*}
Clearly \eqref{aaj1} follows from these realizations.
\end{proof}

In passing, we show a remarkable order-theoretic property regarding $\NNO$-indexed joins that the image of (a slight restriction of) the map $\gr_\ph$ enjoys in the internal poset $(\Omega^\Omega,\leq)$.

\begin{definition}
Let $I$ be a set.
Let $P$ be a poset with $I$-indexed joins.
We say that an element $p \in P$ is \defn{inaccessible by $I$-indexed joins} if
$$\forall p_\ph \in P^I . p \leq \bigvee_{i \in I} p_i \RI \exists i \in I . p \leq p_i.$$
\end{definition}

\begin{proposition}\label{prop:Irred}
Internally in $\Eff$, every element in the image of the map $\gr_\ph: \nabla(\PsPN) \to \Omega^\Omega$ (notice the modified domain!\footnote{The author does not know whether the result still holds for $\nabla(\PPN)$. The reader should note that the image of the map $\gr_\ph: \nabla(\PPN) \to \Omega^\Omega$ differs from that of $\gr_\ph: \nabla(\PsPN) \to \Omega^\Omega$, since $\Eff \not\models \forall \A \in \nabla(\PPN) \exists \B \in \nabla(\PsPN) . \gr_\A = \gr_\B$ (easy to verify).}) is inaccessible by $I$-indexed joins.
\end{proposition}

\begin{proof}
That is, we show the satisfaction
\begin{eqnarray*}
\Eff & \models & \forall \A \in \nabla(\PsPN) \forall h_\ph \in (\Omega^\Omega)^\NNO . \gr_\A \leq \bigvee_{n \in N} h_n \RI \exists n \in \NNO . \gr_\A \leq h_n \\
     & \eqq    & \forall \A \forall h_\ph . [\forall p^\Omega . \gr_\A(p) \RI \exists n^\NNO h_n(p)] \RI [\exists n \forall p . \gr_\A(p) \RI h_n(p)],
\end{eqnarray*}
i.e. that
\begin{equation*}
\begin{array}{l}
\realizes \forall \A \in \PsPN \forall h_\ph \in (\PN^\PN)^\N . \ext(h_\ph) \RI \\
([\forall p . \exists_{A \in \A} (A \BI p) \RI \exists n . \{n\} \wedge h_n(p)] \RI [\exists n . \{n\} \wedge \forall p . \exists_{A \in \A} (A \BI p) \RI h_n(p)]).
\end{array}
\end{equation*}
But $\lambda e . \lambda f . \LET \la n,x \ra = f(\la\id,\id\ra) \IN \la n , \lambda b . e(n)(b)_1(x) \ra$ realizes this.\footnote{Here, `$\LET \ldots \IN \ldots$' is some functional programming notation for introducing aliases. In the usual lambda notation, we mean $\lambda e . \lambda f . \la f(\la\id,\id\ra)_1 , \lambda b . e(f(\la\id,\id\ra)_1)(b)_1(f(\la\id,\id\ra)_2) \ra$.}
We are done.
\end{proof}

The author learnt the previous two propositions from Jaap van Oosten - but for a different setting, to which we now turn our attention.
Given $\A \in \PPN$, define a function $G_\A: \PN \to \PN$ by
$$G_\A(p) = \exists_{A \in \A} (A \RI p).$$
Then the function $G_\ph: \PPN \to \PN^\PN$ represents the composite map
$$\nabla(\PPN) \xto{\gr_\ph} \Omega^\Omega \xto{\ph^\mo} \Mo,$$
and the statements and proofs of Propositions \ref{prop:AllAreJoin} and \ref{prop:Irred} carry easily over to the setting with the function $G_\ph$ and the internal poset $(\Mo,\leq)$ replacing $\gr_\ph$ and $(\Omega^\Omega,\leq)$.
However it is perhaps more neat to deduce these variants from the original statements, as follows.

First, the following is the variant on Proposition \ref{prop:AllAreJoin}.

\begin{proposition}
Internally in $\Eff$, every monotone function $\Omega \to \Omega$ is an $\NNO$-indexed join of functions
in the image of the map $\nabla(\PPN) \xto{\gr_\ph} \Omega^\Omega \xto{\ph^\mo} \Mo$.
\end{proposition}

\begin{proof}
Immediate from Proposition \ref{prop:AllAreJoin}, as the reflection $\ph^\mo: \Omega^\Omega \to \Mo$ is surjective and and preserves joins.
\end{proof}

For the variant on Proposition \ref{prop:Irred}, we use the following general argument.

\begin{proposition}
Let $I$ be a set, and let $(P,\leq)$ be a poset with $I$-indexed joins.
Let $Q \subseteq P$ be a reflective subposet, \emph{closed under $I$-indexed joins in $P$}.
Denote by $l: P \onto Q$ the reflection.
Let $B \subseteq P$.
Suppose that every element of $B$ is inaccessible by $I$-indexed joins in $P$.
Then every element of $l(B)$ is inaccessible by $I$-indexed joins in $Q$.
\end{proposition}

\begin{proof}
To distinguish the join in $P$ and the join in $Q$, we will use the notations $\bigvee^P$ and $\bigvee^Q$.

Let $(l(b_i) \mid i \in I)$ be a family in $l(B)$ and let $l(b) \in l(B)$ such that $l(b) \leq \bigvee^Q_{i \in I} l(b_i)$.
The subset $Q$ being closed under $I$-indexed joins in $P$ is equivalent to that $I$-indexed joins in $Q$ are given by $I$-indexed joins in $P$.
So $\bigvee^Q_{i \in I} l(b_i) = \bigvee^P_{i \in I} l(b_i)$.
It follows that there is $i$ such that $l(b) \leq l(b_i)$.
We are done.
\end{proof}

\begin{proposition}
Let $(P,\leq)$ be a poset, and let $I$ be a set such that $(P,\leq)$ has $I$-indexed joins.
A pointwise join of an $I$-indexed family of monotone functions $P \to P$ is again monotone.
\end{proposition}

\begin{proof}
Let $(h_i \mid i \in I)$ be a family of monotone functions $P \to P$.
We show that the pointwise join $\bigvee_{i \in I} h_i$ is monotonic.
Let $p,q \in P$ with $p \leq q$.
We want $(\bigvee_{i \in I} h_i)(p) \leq (\bigvee_{i \in I} h_i)(q)$,
i.e. $\bigvee_{i \in I} h_i(p) \leq \bigvee_{i \in I} h_i(q)$,
i.e. that for all $i \in I$ one has $h_i(p) \leq \bigvee_{i \in I} h_i(q)$.
But indeed, we have $h_i(p) \leq h_i(q) \leq \bigvee_{i \in I} h_i(q)$.
We conclude that the pointwise join $\bigvee_{i \in I} h_i$ is monotonic.
\end{proof}

\begin{corollary}
Internally in $\Eff$, every element in the image of the map
$$\nabla(\PsPN) \xto{\gr_\ph} \Omega^\Omega \xto{\ph^\mo} \Mo$$
is inaccessible by $I$-indexed joins. \qed
\end{corollary}

% \begin{remark}
% Note, however, that $\nabla(\PPN)$ need not freely generate $\Lo$ under NNO-indexed joins. META: expand.
% \end{remark}

% \begin{remark}
% $\nabla(\PPN) \cong \P_\negneg(\nabla\PN) \subseteq \P(\nabla\PN)$
% \end{remark}

%Proposition \ref{FreeStructure} connotes that internally, every element $\in \Omega^\Omega$ is an $N$-indexed join of values of the map $\nabla_\ph: \nabla(\PPN) \to \Omega^\Omega$.
%Therefore $\Omega^\Omega$, as well as $\Mo$ and $\Lo$, is covered 
%To name this representation, we consider the following diagram.

\subsection{The representation}
We finally define our representation of local operators by functions $\N \to \PPN$.
We do this internally, in the sense that we single out a map $\nabla(\PPN^\N) \to \Lo$.
Of course this also give the external representation, as a function $\N \to \PPN$ is precisely a point of $\nabla(\PPN^\N)$ and a point of $\Lo$ is precisely a local operator.

\begin{notation-proposition}
Given $\A \in \PPN$, we denote by $\Delta_\A: \N \to \PPN$ the function with action
$$\Delta_\A(n) = \A.$$
\emph{The function $\Delta_\ph: \PPN \to \PPN^\N$ represents a mono $\nabla(\PPN) \to \nabla(\PPN^\N)$.} \qed
\end{notation-proposition}

\begin{notation}\label{OmitDelta}
Internally in $\Eff$, given $\A \in \nabla(\PPN)$, we may denote $\Delta_\A \in \nabla(\PPN^\N)$ just by $\A$.%\footnote{`Given $\A \in \PPN$, we may denote $\delta_\A \in \PPN^\N$ just by $\A$.' 하고 외부적 정의를 하는 것과는 엄밀하게 차이가 있음을 염두에 두라.}
\end{notation}

\begin{proposition-notation}
\emph{The function $\id: \PPN^\N \to \PPN^\N$ represents both directions of an isomorphism $\nabla(\PPN^\N) \cong \nabla(\PPN)^N$.}
Under this isomorphism, we shall not distinguish notationally between an (internal) element of $\nabla(\PPN^\N)$ and that of $\nabla(\PPN)^N$. \qed
\end{proposition-notation}

\begin{definition}
Consider the following diagram of maps in $\Eff$, where the unnamed maps are obvious canonical maps.
Clearly everything commutes.
\begin{diagram}
                &                 & \nabla(\PPN)   & \rTo{\gr_\ph} & \Omega^\Omega \\
  & \ldImto{\Delta_\ph}           & \dImto         &                  & \dImto & \rdEq    \\
\nabla(\PPN^\N) & \sr{\id}{\cong} & \nabla(\PPN)^N & \rDoto & (\Omega^\Omega)^N & \rTo{\vee} & \Omega^\Omega & \rOnto{\ph^\mo} & \Mo & \rOnto{\ph^\lomo} & \Lo
\end{diagram}
Note that by Proposition \ref{prop:AllAreJoin}, the composite map $\nabla(\PPN)^N \to \Omega^\Omega$ is an epi.
We shall refer to the composite maps
$$\nabla(\PPN^\N) \onto \Omega^\Omega,\Mo,\Lo$$
as $\ar_\ph,\mo_\ph,\lo_\ph$ respectively.
Internally in $\Eff$, given $\theta \in \nabla(\PPN^\N)$, we call $\lo_\theta$ the local operator \defn{associated} to $\theta$.
\end{definition}

\begin{remark}
Combining the previous Notation with Notation \ref{OmitDelta}, we can denote the maps
$$\nabla(\PPN) \to \Omega^\Omega,\Mo,\Lo$$
in the diagram above also as $\ar_\ph,\mo_\ph,\lo_\ph$.

Note (from the commutativity of the diagram) that as maps $\nabla(\PPN) \to \Omega^\Omega$ we have $\gr_\ph = \ar_\ph$.
\end{remark}

\begin{definition}\label{def:basic}
Internally in $\Eff$, let $j \in \Lo$.
We say that $j$ is \defn{basic} if $j = \lo_\A$ for some $A \in \nabla(\PPN)$.
\end{definition}

Let us describe a representation of the map $\ar_\ph: \nabla(\PPN^\N) \to \Omega^\Omega$.

\begin{proposition}
Given $\theta \in \PPN^\N$, define a function $\ar_\theta: \PN \to \PN$ by
$$\ar_\theta(p) = \exists n^\N . \{n\} \wedge \gr_{\theta(n)}(p) = \exists n^\N . \{n\} \wedge \exists_{A \in \theta(n)} (A \BI p).$$
\claim{The function $\ar_\ph: \PPN^\N \to \PN^\PN$ represents the map $\ar_\ph: \nabla(\PPN)^\NNO \onto \Omega^\Omega$.}
\end{proposition}

\begin{proof}
The map $\ar_\ph: \nabla(\PPN)^N \to \Omega^\Omega$ is by definition the composite
$$\nabla(\PPN)^N \xto{(\gr_\ph)^N} (\Omega^\Omega)^N \xto{\vee} \Omega^\Omega.$$
So internally in $\Eff$, for $\theta \in \nabla(\PPN)^N$, we have
$$\ar_\theta(p) = \vee[(\gr_\ph)^\NNO(\theta)](p) = \bigvee_{n \in N} \gr_{\theta(n)}(p) = \exists n \in N . \gr_{\theta(n)}(p).$$
Therefore the function $\ar_\ph: \PPN^\N \to \PN^\PN$ clearly represents the map $\ar_\ph: \nabla(\PPN)^\NNO \onto \Omega^\Omega$.
\end{proof}

\begin{remark}
Now we can cleanly explain the connection between the construction $(\theta \in \PPN^\N) \mapsto (\lo_\theta \in \Lop(\Eff))$ and the construction of the least local operator making a given subobject dense.

Let us see that the latter construction is an instance of the former, as follows.
Let $(X,\sim)$ be an object of $\Eff$, and let $R: X \to \PN$ be a predicate on $(X,\sim)$.
Define a function $\delta[(X,\sim),R]: \N \to \PPN$ by
$$\delta[(X,\sim),R](n) = \{R(x) \mid x \in X, n \in \ex(x)\}.$$
Then it is immediate that the function $\ar_{\delta[(X,\sim),R]}: \PN \to \PN$ represents the subobject
$$\Im(\chi_{R \subseteq (X,\sim)}: \Omega \to \Omega) \in \Sub(\Omega),$$
hence $\lo_{\delta[(X,\sim),R]} = (\ar_{\delta[(X,\sim),R]})^\loar$ is\footnote{see Construction \ref{LeastDense}} the least local operator making the subobject $R \subseteq (X,\sim)$ dense.

The converse, that every local operator of the form $\lo_\theta$ is the least one making a subobject dense, follows of course from the general fact\footnote{see Corollary \ref{cor:EveryLoc}} that any local operator is such a one.
But let us track this concretely.
Given $\theta \in \PPN^\N$, define a function $\theta^*: \PN \to \PN$ by
$$\theta^*(p) = \{m \in \N \mid p \in \theta(m)\}.$$
Then we have a preassembly $(\PN,\theta^*)$ and a predicate $\id_\PN: \PN \to \PN$ on it.
Note that
\begin{align*}
\theta(n) &= \{p \mid p \in \PN \en p \in \theta(n)\} \\
          &= \{p \mid p \in \PN \en n \in \{m \in \N \mid p \in \theta(m)\}\} \\
          &= \{\id_\PN(p) \mid p \in \PN, n \in \theta^*(p)\} \\
          &= \delta[(\PN,\theta^*),\id_\PN](n),
\end{align*}
so that $\lo_\theta$ is the least local operator making the subobject $\id_\PN \subseteq (\PN,\theta^*)$ dense.
\end{remark}

\begin{remark}
Building upon the previous remark, we can now see that a local operator $j$ is basic if and only if it is the least one making a subobject of a $\negneg$-sheaf dense.

`if': Let $X$ be a set and $R: X \to \PN$ be a function such that $j$ is the least local operator making $R \subseteq \nabla(X)$ dense.
So $$\delta[\nabla(X),R](n) = \{R(x) \mid x \in X \en n \in \ex(x)\} = \{R(x) \mid x \in X\}.$$
Hence $j = \lo_{\{R(x) \mid x \in X\}}$ is basic.

`only if': If $j$ is basic, then $j = \lo_\A$ for some $\A \in \PPN$, so $j$ is the least local operator making the inclusion $\id_\PN \subseteq (\PN,(\Delta_\A)^*)$ dense. But
$$(\Delta_\A)^*(p) = \{m \in \N \mid p \in \Delta_\A(m) = \A\} = \begin{cases} \N & \textif p \in \A \\ \emptyset & \otherwise. \end{cases}$$ So $(\PN,(\Delta_\A)^*)$ is a $\negneg$-sheaf, as desired.
\end{remark}

Next we describe a representation of the map $\mo_\ph: \nabla(\PPN^\N) \to \Mo$.

\begin{proposition}\label{MoDesc}
Given $\theta \in \PPN^\N$, define $\mo_\theta \in \PN^\PN$ to be the predicate given by
$$\mo_\theta(p) = \dbl \exists n^\N \exists_{A \in \theta(n)} . \{n\} \wedge (A \RI p) \dbr.$$
\claim{The function $\mo_\ph: \PPN^\N \to \PN^\PN$ represents the map $\mo_\ph: \nabla(\PPN)^N \onto \Mo$.}
\end{proposition}

\begin{proof}
The map $\mo_\ph: \nabla(\PPN)^N \to \Omega^\Omega$ is the composite
$$\nabla(\PPN)^N \xto{\ar_\ph} \Omega^\Omega \xto{\ph^\mo} \Mo.$$
So internally in $\Eff$, for $\theta \in \nabla(\PPN)^N$ and $p \in \Omega$, we have
$$\mo_\theta(p) = (\ar_\theta)^\mo(p) = \exists q^\Omega . \ar_\theta(q) \wedge (q \RI p).$$
Therefore the map $\mo_\ph\ph: \nabla(\PPN)^\NNO \times \Omega \to \Omega$ is represented by the formula
\begin{eqnarray*}
&      & \theta \in \PPN^\N, p \in \PN . \exists q^\PN . \exists n^\N [\{n\} \wedge \exists_{A \in \theta(n)} (A \BI q)] \wedge (q \RI p) \\
& \eqq & \theta \in \PPN^\N, p \in \PN . \exists q^\PN \exists n^\N \exists_{A \in \theta(n)} . \{n\} \wedge (A \BI q) \wedge (q \RI p) \\
& \eqq & \theta \in \PPN^\N, p \in \PN . \exists n^\N \exists_{A \in \theta(n)} . \{n\} \wedge (A \RI p).
\end{eqnarray*}
It follows that the function $\mo_\ph$ represents the map $\mo_\ph$.
\end{proof}

Having described representations of the maps $\ar_\ph$ and $\mo_\ph$, the map $\lo_\ph$ remains.
Describing a good representation of $\lo_\ph$ is a substantial issue, and will be considered in the next section (\S\ref{sec:sights}).
Before going into it, we first discuss about relevant (pre)orders on the set $\PPN^\N$.

\subsection{The preorderings $\leq_\mo$ and $\leq_\lo$}
Internally speaking, the `representation maps' $\nabla(\PPN^\N) \onto \Mo,\Lo$ facillitate studying monotonics and topologies in terms of elements of $\nabla(\PPN^\N)$.
In particular, we may study their inequalities so.
In this spirit, we pull back the orderings of monotonics and of topologies to $\nabla(\PPN^\N)$, as in the following definition.

\begin{definition}
We define $(\nabla(\PPN^\N),\leq_\mo)$ and $(\nabla(\PPN^\N),\leq_\lo)$ to be the preorders induced by the maps
$$\mo_\ph,\lo_\ph: \nabla(\PPN^\N) \to \Mo,\Lo$$
respectively.
\end{definition}

\begin{remark}
With Notation \ref{OmitDelta}, this also gives us preorders $(\nabla(\PPN),\leq_\mo)$ and $(\nabla(\PPN),\leq_\lo)$.
\end{remark}

\begin{lemma}\label{SmallestMoFn}
We have
\begin{quote}
$\realizes \forall \A \in \PPN \forall h \in \PN^\PN . \mon(h) \RI [(\mo_\A \RI h) \BI \bigcap_{A \in \A} h(A)]$.
\end{quote}
In particular, given $\A \in \PPN$, the function $\mo_\A$ is the smallest monotone function $h: \PN \to \PN$ with $\bigcap_{A \in \A} h(A) \neq \emptyset$.
\end{lemma}

\begin{proof}
Recall that
\begin{indentation}
$\mon(h) = \forall p,q^\PN . (p \RI q) \RI (h(p) \RI h(q))$,
\end{indentation}
that
\begin{align*}
\mo_\A \RI h
&\eqq \forall p^\PN . [\exists n \in \N \exists_{A' \in \Delta_\A(n)} . \{n\} \wedge (A' \RI p)] \RI h(p) \\
&\eqq \forall p^\PN . [\exists A'^\PN . A' \nabin \A \wedge (A' \RI p)] \RI h(p),
\end{align*}
and that
$$\forall_{A \in \A} h(A) \eqq \forall A^\PN . A \nabin \A \RI h(A).$$
Now observe that the claim holds.
\end{proof}

\begin{proposition}\label{LeqMoDesc}
Define a predicate $\leq_\mo: \PPN^\N \times \PPN^\N \to \PN$ by
$$\eta \leq_\mo \theta = \stil \forall n^\N . \{n\} \RI \forall_{A \in \eta(n)} \exists m^\N . \{m\} \wedge \exists_{B \in \theta(m)} (B \RI A) \stir.$$
\claim{The predicate $\leq_\mo$ represents the relation $(\nabla(\PPN^\N),\leq_\mo)$.}
\end{proposition}

\begin{proof}
For now, to avoid ambiguity, let us denote the \emph{predicate} $\leq_\mo$ by $\leq_\mo'$.
The predicate $\leq_\mo'$ automatically represents a binary relation on $\nabla(\PPN^\N)$.
We have to prove the satisfaction
\begin{eqnarray*}
\Eff & \models & \forall \eta,\theta \in \nabla(\PPN^\N) . \eta \leq_\mo' \theta \BI \eta \leq_\mo \theta \\
     & \eqq    & \forall \eta,\theta . \eta \leq_\mo' \theta \BI \mo_\eta \leq \mo_\theta \\
     & \eqq    & \forall \eta,\theta . \eta \leq_\mo' \theta \BI \bigvee_{n \in N} \mo_{\eta(n)} \leq \mo_\theta \\
     & \eqq    & \forall \eta,\theta . \eta \leq_\mo' \theta \BI \forall n^N . \mo_{\eta(n)} \leq \mo_\theta,
\end{eqnarray*}
i.e. that
\begin{equation}\label{lmd1}
\realizes \forall \eta,\theta \in \PPN^\N . \eta \leq_\mo' \theta \BI [\forall n^\N . \{n\} \RI \mo_{\eta(n)} \leq \mo_\theta].
\end{equation}
But since\footnote{by Proposition \ref{MoDesc}} $\realizes \forall \A . \mon(\mo_\A)$,
we have by Lemma \ref{SmallestMoFn} that
$$\realizes \forall \eta,\theta \forall n . \mo_{\eta(n)} \leq \mo_\theta \BI \forall_{A \in \eta(n)} \mo_\theta(A).$$
So \eqref{lmd1} is equivalent to
\begin{align*}
& \realizes \forall \eta,\theta \in \PPN^\N . \eta \leq_\mo' \theta \BI
[\forall n^\N . \{n\} \RI \forall_{A \in \eta(n)} \mo_\theta(A)] \\
& \eqq      \forall \eta,\theta \in \PPN^\N . \eta \leq_\mo' \theta \BI
[\forall n^\N . \{n\} \RI \forall_{A \in \eta(n)} \exists m^\N . \{m\} \wedge \exists_{B \in \theta(m)} (B \RI A)],
\end{align*}
but the last formula holds trivially by definition of $\leq_\mo'$.
Done.
\end{proof}

\begin{corollary}\label{BasicLeqMoDesc}
Define a predicate $\leq_\mo: \PPN \times \PPN \to \PN$ by
$$\A \leq_\mo \B = \stil \forall_{A \in \A} \exists_{B \in \B} . B \RI A \stir.$$
\claim{This predicate $\leq_\mo$ represents the relation $(\nabla(\PPN),\leq_\mo)$.}
\end{corollary}

\begin{proof}
For now, again to avoid ambiguity, let us denote the \emph{predicate} $\leq_\mo$ by $\leq_\mo'$.
We want the satisfaction
\begin{eqnarray*}
\Eff & \models & \forall \A,\B^{\nabla(\PPN)} . \A \leq_\mo' \B \BI \A \leq_\mo \B \\
     & \eqq    & \forall \A,\B^{\nabla(\PPN)} . \A \leq_\mo' \B \BI \Delta_\A \leq_\mo \Delta_\B,
\end{eqnarray*}
i.e. (by the last Proposition) that
\begin{equation}\label{blmd1}
\realizes \forall \A,\B^\PPN . \A \leq_\mo' \B \BI \forall n^\N . \{n\} \RI \forall_{A \in \Delta\A(n)} \exists m^\N . \{m\} \wedge \exists_{B \in \Delta\B(m)} (B \RI A).
\end{equation}
But since $\realizes \forall \A^\PPN . \A \nabeq \Delta\A$ and $\realizes \exists n^\N . \{n\}$ and $\realizes \forall x^X . x \nabin X$, we have by logic
\begin{eqnarray*}
&& \forall n^\N . \{n\} \RI \forall_{A \in \Delta\A(n)} \exists m^\N . \{m\} \wedge \exists_{B \in \Delta\B(m)} (B \RI A) \\
& \eqq & \forall n^\N . \{n\} \RI \forall A^\PN . A \nabin \Delta\A(n) \RI \exists m^\N . \{m\} \wedge \exists B^\PN . B \nabin \Delta\B(m) \wedge (B \RI A) \\
& \eqq & \forall n^\N . \{n\} \RI \forall A^\PN . A \nabin \A \RI \exists m^\N . \{m\} \wedge \exists B^\PN . B \nabin \B \wedge (B \RI A) \\
& \eqq & \forall A^\PN . A \nabin \A \RI \exists B^\PN . B \nabin \B \wedge (B \RI A) \\
& \eqq & \forall_{A \in \A} \exists_{B \in \B} (B \RI A).
\end{eqnarray*}
So \eqref{blmd1} is equivalent to
$$\realizes \forall \A,\B^\PPN . \A \leq_\mo' \B \BI \forall_{A \in \A} \exists_{B \in \B} (B \RI A),$$
which holds trivially by definition of $\leq_\mo'$.
Done.
\end{proof}

In the following we construct a meet operation for the internal preorder $(\nabla(\PPN^\N),\leq_\mo)$.
Jaap van Oosten provided the core part (the operation on $\PPN$) of this construction.

\begin{proposition}
Given $\A,\B \in \PPN$, let
$$\A \ovee \B = \{A \vee B \mid A \in \A \en B \in \B\} \in \PPN.$$
Given $\eta,\theta \in \PPN^\N$, define a function $\eta \ovee \theta: \N \to \PPN$ by
$$(\eta \ovee \theta)(\la n,m \ra) = \eta(n) \ovee \theta(m).$$
\claim{The function $\ovee: \PPN^\N \times \PPN^\N \to \PPN^\N$ represents a \emph{meet} operation for the preorder $(\nabla(\PPN^\N),\leq_\mo)$.}
\end{proposition}

\begin{proof}
Clearly, the function $\ovee$ represents a map $\nabla(\PPN^\N) \times \nabla(\PPN^\N) \to \nabla(\PPN^\N)$.
First we verify that
$$\Eff \models \forall \zeta,\xi \in \nabla(\PPN^\N) . \zeta \ovee \xi \leq_\mo \zeta,$$
i.e.\footnote{See Proposition \ref{LeqMoDesc}.} that
\begin{equation}\label{eq:ov1}
\realizes \forall n \in \N . \{n\} \ri \forall_{A \vee B \in (\zeta \ovee \xi)(n)} \exists m \in \N [\{m\} \wedge \exists_{A' \in \zeta(m)} (A' \ri A \vee B)].
\end{equation}
Let $p_1: \N \to \N$ be the function determined by $p_1(\la m,n \ra) = m$.
As this function is computable, we can clearly realize
$$\realizes \forall n \in \N . \{n\} \ri \forall_{A \vee B \in (\zeta \ovee \xi)(n)} [\{p_1(n)\} \wedge A \nabin \zeta(p_1(n)) \wedge (A \ri A \vee B)].$$
This yields \eqref{eq:ov1} by logic, as desired.
Symmetrically, $\Eff \models \forall \zeta,\xi . \zeta \ovee \xi \leq_\mo \xi$.

Next, we need to verify that
$$\Eff \models \forall \theta,\zeta,\xi \in \nabla(\PPN^\N) . \theta \leq_\mo \zeta \wedge \theta \leq_\mo \xi \RI \theta \leq_\mo \zeta \ovee \xi,$$
i.e. that
\begin{align*}
& \realizes \forall \theta,\zeta,\xi \in \PPN^\N . \\
& \qquad [\forall n \in \N . \{n\} \ri \forall_{C \in \theta(n)} \exists m \in \N . \{m\} \wedge \exists_{A \in \zeta(m)} (A \RI C)] \\
& \qquad \wedge \\
& \qquad [\forall n \in \N . \{n\} \ri \forall_{C \in \theta(n)} \exists k \in \N . \{k\} \wedge \exists_{B \in \xi(m)} (B \RI C)] \\
& \qquad \RI \\
& \qquad [\forall n \in \N . \{n\} \ri \forall_{C \in \theta(n)} \exists \la m,k \ra \in \N . \{\la m,k \ra\} \wedge \exists_{A \vee B \in (\zeta \ovee \xi)(\la m,k \ra)} (A \vee B \RI C)].
\end{align*}
But, since the mapping $\N \times \N \to \N: (n,m) \mapsto \la n,m \ra$ is computable, this is clearly realizable.
Done.
\end{proof}

\begin{corollary}
\claim{The function $\ovee: \PPN \times \PPN \to \PPN$ represents a meet operation for the preorder $(\nabla(\PPN),\leq_\mo)$.} \qed
\end{corollary}

Let us now turn our attention to $(\nabla(\PPN^\N),\leq_\lo)$.
Given a representation of the map $\lo_\ph: \nabla(\PPN^\N) \to \Lo$, we can obtain a representation of the preordering $\leq_\lo$ on $\nabla(\PPN^\N)$, as follows.

\begin{proposition}
Let $l_\ph: \PPN^\N \to \PN^\PN$ be a function representing the map $\lo_\ph: \nabla(\PPN)^N \to \Lo$.
Then the predicate $\leq_l: \PPN^\N \times \PPN^\N \to \PN$ defined by
$$\theta \leq_l \zeta = \dbl \forall n^\N . \{n\} \RI \forall_{A \in \theta(n)} l_\zeta(A) \dbr$$
represents the preordering $\leq_\lo$ on $\nabla(\PPN)^N$.
\end{proposition}

\begin{proof}
Analogous to the proof of Proposition \ref{LeqMoDesc}.
\end{proof}

Of course, this proposition will gain its force only after we discuss representations of the map $\lo_\ph$ in the next section.

\section{The method of sights}\label{sec:sights}
In this section, we aim to establish a couple of representations of the map $\lo_\ph: \nabla(\PPN^\N) \to \Lo$ and understand these representations.

\subsection{A representation of $\lo_\ph$, and sights}
The following Proposition establishes the first of the two representations we shall consider of the map $\lo_\ph: \nabla(\PPN^\N) \to \Lo$.
The first characterization of this (first) representation, which we now meet as definition, is (essentially) due to A.M. Pitts; cf. \cite[Proposition 5.6]{pitts81}.

\begin{proposition}\label{PittsLemma}
Given $\theta \in \PPN^\N$ and $p \in \PN$, define
$$\loF_\theta(p) = \bigcap \{q \subseteq \N \mid \{0\} \wedge p \subseteq q \en \{1\} \wedge \mo_\theta(q) \subseteq q\}.$$
The function $\loF_\ph: \PPN^\N \to \PN^\PN$ represents the map $\nabla(\PPN^\N) \to \Lo$.
\end{proposition}

\begin{proof}
It suffices to show that the functions $\loF_\ph,(\mo_\ph)^\lomo: \PPN^\N \to \PN^\PN$ are isomorphic as functions $\nabla(\PPN^\N) \to \Lo$, i.e.
$$\realizes \forall \theta \in \PPN^\N . \loF_\theta \sr{\lo}{\BI} (\mo_\theta)^\lomo,$$
i.e. [using the fact that $\realizes \forall \theta . \loF((\mo_\theta)^\lomo)$]
\begin{equation}\label{eq:pl1}
\realizes \forall \theta \in \PPN^\N . \loF_\theta \BI (\mo_\theta)^\lomo.
\end{equation}
Recall, for $\theta \in \PPN^\N$, that
$$\mo_\theta(p) = \exists n \in \N . \{n\} \wedge \exists A \in \theta(n) (A \RI p)$$
and that
$$(\mo_\theta)^\lomo(p) = \forall q \in \PN . (\mo_\theta(q) \RI q) \wedge (p \RI q) \RI q.$$

First we show `$\LI$' from \eqref{eq:pl1}, i.e. that
\begin{equation}\label{lomo<=star}
\realizes \theta,p . \forall q [(\mo_\theta(q) \RI q) \wedge (p \RI q) \RI q] \RI \loF_\theta(p).
\end{equation}
But we have $\realizes \theta,p . \mo_\theta(\loF_\theta(p)) \RI \loF_\theta(p)$ and $\realizes \theta,p . p \RI \loF_\theta(p)$.
Hence \eqref{lomo<=star} follows, as desired.

Next we show `$\RI$'.
Take $a \realizes \forall \theta \forall p . p \RI (\mo_\theta)^\lomo(p)$.
Also take
\begin{align*}
b &\realizes \forall \theta \forall p . [\exists n \in \N . \{n\} \wedge \exists A \in \theta(n) (A \RI \mo_\theta^\lomo(p))] \RI \mo_\theta^\lomo(p) \\
  &=         \forall \theta \forall p . \mo_\theta(\mo_\theta^\lomo(p)) \RI \mo_\theta^\lomo(p).
\end{align*}
Take an index $c$ by recursion theorem such that
$$
c(x) =
\begin{cases}
a(y)                   & \text{if $x = \la 0,y \ra$ for some $y$}, \\
b(n,\lambda m.c[e(m)]) & \text{if $x = \la 1,\la n,e \ra \ra$}.
\end{cases}
$$
Given $\theta \in \PPN^\N$ and $p \in \PN$, consider the set $S_\theta(p) = \{x \in \N \mid c(x) \in \mo_\theta^\lomo(p)\}$.
We see that
\begin{itemize}
 \item $\{0\} \wedge p \subseteq S_\theta(p)$
 \item $\{1\} \wedge \mo_\theta(S_\theta(p)) \subseteq S_\theta(p)$,
\end{itemize}
so we can deduce that $\loF_\theta(p) \subseteq S_\theta(p)$.
Thus clearly $c \realizes \forall \theta \forall p . \loF_\theta(p) \RI \mo_\theta^\lomo(p)$, and this completes the proof.
\end{proof}

\begin{proposition}
Let $\theta \in \PPN^\N$ and $p \in \PN$.
Define
\begin{quote}
$\mo_\theta^0(p) = \{0\} \wedge p$ \\
$\mo_\theta^{\alpha+1}(p) = \mo_\theta^\alpha(p) \cup \{1\} \wedge \mo_\theta(\mo_\theta^\alpha(p))$ \\
$\mo_\theta^\lambda(p) = \bigcup_{\alpha < \lambda} \mo_\theta^\alpha(p)$.
\end{quote}
Then $\mo_\theta^{\omega_1}(p) = \loF_\theta(p)$.
\end{proposition}

\begin{proof}
For brevity, we abbreviate $h = \mo_\theta$.

Let us first see that $h: \PN \to \PN$ preserves inclusions.
If $p,q \in \PN$ with $p \subseteq q$, then clearly
$$
h(p)
=
\bbl \exists n \in \N \exists_{A \in \theta(n)} . \{n\} \wedge (A \RI p) \bbr
\subseteq
\bbl \exists n \in \N \exists_{A \in \theta(n)} . \{n\} \wedge (A \RI q) \bbr
=
h(q),
$$
as desired.

We prove the statement
\begin{quote}
For all $\alpha$ and for all $q \in \PN$ with $\{0\} \wedge p \subseteq q$ and $\{1\} \wedge h(q) \subseteq q$, we have $h^\alpha(p) \subseteq q$.
\end{quote}
by induction on $\alpha$.

The ordinal 0 case. Trivially, $h^\alpha(p) = \{0\} \wedge p \subseteq q$.

The case of a successor ordinal $\alpha+1$.
We want to show that $h^\alpha(p) \cup \{1\} \wedge h(h^\alpha(p)) \subseteq q$.
The IH says $h^\alpha(p) \subseteq q$, so it remains to show $\{1\} \wedge h(h^\alpha(p)) \subseteq q$.
Since $h$ preserves inclusion, we have $h(h^\alpha(p)) \subseteq h(q)$.
It follows that $\{1\} \wedge h(h^\alpha(p)) \subseteq \{1\} \wedge h(q) \subseteq q$, as desired.

The case of a limit ordinal $\lambda$.
The IH tells us that each $\alpha < \lambda$ satisfies $h^\alpha(p) \subseteq q$.
It follows that $h^\lambda(p) = \bigcup_{\alpha < \lambda} h^\alpha(p) \subseteq q$, as desired.

This completes the induction.
It follows that for all $\alpha$ (in particular for $\omega_1$) we have $h^\alpha(p) \subseteq \{q \in \PN \mid \{0\} \wedge p \subseteq q \en \{1\} \wedge h(q) \subseteq q\} = \loF_\theta(p)$.

To show $\loF_\theta(p) \subseteq h^{\omega_1}(p)$, it clearly suffices\footnote{See the definition of $\loF_\theta(p)$.} to show that $\{0\} \wedge p \subseteq h^{\omega_1}(p)$ and $\{1\} \wedge h(h^{\omega_1}(p)) \subseteq h^{\omega_1}(p)$.
The former is clear, so let us show the latter inclusion.
Let $\la n,e \ra \in h(h^{\omega_1}(p)) = h(\bigcup_{\alpha < \omega_1} h^\alpha(p))$.
Then we can choose\footnote{See the definition of $h$.} $A \in \theta(n)$ for which $e \realizes A \RI \bigcup_{\alpha < \omega_1} h^\alpha(p)$.
So for each $x \in A$, we can choose an ordinal $o(x) < \omega_1$ with $e(x) \in h^{o(x)}(p)$.
Let $\pi = \bigcup_{x \in A} o(x)$, which is (being a countable union of countable ordinals) a countable ordinal.
Then for each $x \in A$ we have $e(x) \in h^{o(x)}(p) \subseteq h^\pi(p)$, so $e \realizes A \RI h^\pi(p)$.
Therefore $\la n,e \ra \in h(h^\pi(p))$, and so $\la 1,\la n,e \ra \ra \in \{1\} \wedge h(h^\pi(p)) \subseteq h^{\pi+1}(p) \subseteq h^{\omega_1}(p)$, as desired.
This proves $\loF_\theta(p) \subseteq h^{\omega_1}(p)$.
\end{proof}

%\begin{remark}
%The reason I replaced the `fixing an index for the empty function' formalism to what is now, is because I'm quite convinced that that the former is psychologically not very pleasant (in particular related to the edge cases where the empty set is is contained in $\mathcal{A}$ or $\theta(n)$).
%\end{remark}

\begin{definition}
A \defn{sight} is, inductively,
\begin{itemize}
 \item either: something called $\nil$,
 \item or: a pair $(A,\sigma)$ where $A \in \PN$ and $\sigma(a)$ is a sight for each $a \in A$.
\end{itemize}
Given $\theta \in \PPN^\N$, $z \in \N$ and $p \in \PN$, a sight $S$ is \defn{$(z,\theta,p)$-dedicated} if, by induction on $S$,
\begin{itemize}
 \item in case $S = \nil$: $z = \la 0,y \ra$ with $y \in p$,
 \item in case $S = (A,\sigma)$: (i) $z = \la 1, \la n,e \ra \ra$ for some $n,e \in \N$, (ii) $A \in \theta(n)$, (iii) $e$ is defined on $A$, and (iv) for each $a \in A$ the sight $\sigma(a)$ is $(e(a),\theta,p)$-dedicated.
\end{itemize}
We say that $S$ is \defn{$(z,\theta)$-dedicated} if $S$ is $(z,\theta,\N)$-dedicated.
\end{definition}

\begin{proposition}\label{prop:loF}
Let $\theta \in \PPN^\N$.
For all $z \in \N$ and $p \in \PN$, we have
\begin{center}
$z \in \loF_\theta(p)$ if and only if there is a $(z,\theta,p)$-dedicated sight.
\end{center}
\end{proposition}

\begin{proof}
To deduce the `only if' part of the proposition, we show
\begin{center}
For all $\alpha$, for all $z \in \N$ and $p \in \PN$,\\
if $z \in \mo_\theta^\alpha(p)$ then there is a $(z,\theta,p)$-dedicated sight.
\end{center}
by transfinite induction on $\alpha$.

The ordinal $0$ case.
That $z \in \mo_\theta^0(p)$ means that $z = \la 0,y \ra$ with $y \in p$.
Therefore the sight $\nil$ is a $(z,\theta,p)$-dedicated sight.

The case of a successor ordinal $\alpha+1$.
Remind that $\mo_\theta^{\alpha+1}(p) = \mo_\theta^\alpha(p) \cup \mo_\theta(\mo_\theta^\alpha(p))$.
If $z \in \mo_\theta^\alpha(p)$, we are clearly done by the IH; so suppose $z \in \mo_\theta(\mo_\theta^\alpha(p))$.
Write $z = \la 1, \la n,e \ra \ra$.
There is $A \in \theta(n)$ such that $e$ is defined on $A$ and $e(a) \in \mo_\theta^\alpha(p)$ for each $a \in A$.
By the IH, choose for each $a \in A$ an $(e(a),\theta,p)$-dedicated sight $\sigma(a)$.
Then $(A,\sigma)$ is clearly a $(z,\theta,p)$-dedicated sight.

The case of a limit ordinal $\lambda$.
As $z \in \mo_\theta^\lambda(p) = \bigcup_{\alpha < \lambda} \mo_\theta^\alpha(p)$, there is $\alpha < \lambda$ such that $x \in \mo_\theta^\alpha(p)$.
Hence by the IH, there is a $(z,\theta,p)$-dedicated sight.
Induction done.

To deduce the `if' part of the proposition, we show
\begin{center}
For every sight $S$, for all $z \in \N$ and $p \in \PN$,\\
if $S$ is $(z,\theta,p)$-dedicated, then $z \in \mo_\theta^{\omega_1}(p)$.
\end{center}
by induction on $S$.

The case $S = \nil$.
Then for some $y \in p$ we have
$$z = \la 0,y \ra \in \{0\} \wedge p = \mo_\theta^0(p) \subseteq \mo_\theta^{\omega_1}(p)$$
as desired.

The case $S = (A,\sigma)$.
Then $z = \la 1, \la n,e \ra \ra$ for some $n,e$, and $A \in \theta(n)$.
For each $a \in A$, the sight $\sigma(a)$ is $(e(a),\theta,p)$-dedicated;
so by the IH, we can choose an ordinal $o(a) < \omega_1$ such that $e(a) \in \mo_\theta^{o(a)}(p)$.
Consider the ordinal $\pi = \bigcup_{a \in A} o(a)$.
Since $o(a) \leq \pi$ for each $a \in A$, we have $e(a) \in \mo_\theta^{o(a)}(p) \subseteq \mo_\theta^\pi(p)$.
Therefore $z \in \mo_\theta(\mo_\theta^{\pi}(p)) \subseteq \mo_\theta^{\pi+1}(p)$.
We are done, since $\pi+1$ is still countable.
\end{proof}
% 
% \begin{corollary}
% Let $\eta,\theta \in \PPN^\N$, and let $f \in \N$.
% We have,
% \begin{quote}
% $f \realizes \eta \leq_\lo \theta$
% \end{quote}
% if and only if
% \begin{quote}
% for each $n$ and $A \in \eta(n)$ there is a $(f(n),\theta,A)$-dedicated sight.
% \end{quote}
% \end{corollary}
% 
% \begin{proof}
% To be written.
% \end{proof}
% 
% \begin{corollary}
% Let $\A,\B \in \PPN$, and let $z \in \N$.
% We have,
% \begin{quote}
% $z \realizes \A \leq_\lo \B$
% \end{quote}
% if and only if
% \begin{quote}
% for each $A \in \A$ there is a $(z,\B,A)$-dedicated sight.
% \end{quote}
% \end{corollary}
% 
% \begin{proof}
% To be written.
% \end{proof}

\subsection{Basic notions around sight}
\begin{definition}
Let $S$ be a sight.

A finite sequence $(x_1,\ldots,x_n)$ in $\N$ is a \defn{node} of $S$ if, by induction on $n$,
\begin{itemize}
 \item if $n = 0$, always,
 \item if $n > 0$, then $S = (A,\sigma)$, $x_1 \in A$ and $(x_2,\ldots,x_n)$ is a node of $\sigma(x_1)$.
\end{itemize}
We denote by $\Nds(S)$ the set of nodes of $S$.

Given a node $(x_1,\ldots,x_n)$ of $S$, the \defn{subsight} of $S$ \defn{at} $(x_1,\ldots,x_n)$, to be denoted $\Subsight_S(x_1,\ldots,x_n)$, is, by induction on $n$,
\begin{itemize}
 \item if $n = 0$, then it is the sight $S$,
 \item if $n > 0$, then it is the subsight of $\sigma(x_1)$ at $(x_2,\ldots,x_n)$.
\end{itemize}
A node $s$ of $S$ is a \defn{leaf} if $\Subsight_S(s) = \nil$.

Given a node $s$ of $S$, we write $\Out_S(s) = \{x \in \N \mid \text{$s \cons x$ is a node of $S$}\}$.
Note that if $\Subsight_S(s) = (B,\_)$, then $\Out_S(s) = B$.
\end{definition}

\begin{definition}
A node $s$ of a sight is \defn{degenerate} if the subsight at $s$ is $(\emptyset,\emptyset)$.
A sight is \defn{degenerate} if it has a degenerate node, i.e. it has a subsight equal to $(\emptyset,\emptyset)$.
\end{definition}

\begin{definition}
Given $\A \in \PPN$, a sight $S$ is said to be \defn{on $\A$} if each subsight of $S$ of the form $(A',\sigma')$ satisfies $A' \in \A$.
\end{definition}

\begin{proposition}
If a sight $S$ is $(z,\theta)$-dedicated for some $z,\theta$, then $S$ is on $\bigcup_{n \in \N} \theta(n)$. \qed
\end{proposition}

\begin{definition}
Let $z \in \N$, and let $s = (x_1,\ldots,x_k) \in \N^*$.
The \defn{r-value of $s$ under $z$}, denoted $z[s]$, is sometimes defined, by induction on $k$, as follows.
\begin{itemize}
 \item (case $k = 0$) For it to exist, we must have $z = \la 0,y \ra$. If so, it is $y$.
 \item (case $k > 0$) For it to exist, we must have that (i) $z = \la 1,\la\_,e\ra \ra$, (ii) $e$ is defined on $x_1$, and (iii) the r-value of $(x_2,\ldots,x_k)$ under $e(x_1)$ exists. If so, it is the r-value of $(x_2,\ldots,x_k)$ under $e(x_1)$.
\end{itemize}
We say that $z$ is \defn{r-defined} on $s$ if the r-value of $s$ under $z$ exists.
\end{definition}

\begin{definition}
Let $z \in \N$, and let $S$ be a sight.
We say that $z$ is \defn{r-defined} on $S$ if $z$ is defined on every leaf of $S$.
If $z$ is r-defined on $S$, we define the \defn{image} of $S$ under $z$ to be the set
$$z[S] := \{z[s] \mid \text{$s$ is a leaf of $S$}\}.$$
A \defn{$z$-value of $S$} is an element of $z[S]$.
\end{definition}

\begin{proposition}
If a sight $S$ is $(z,\theta)$-dedicated for some $z,\theta$, then $z$ is defined on $S$.
A $(z,\theta)$-dedicated sight $S$ is $(z,\theta,p)$-dedicated if and only if $z[S] \subseteq p$. \qed
\end{proposition}

Next we prove some properties of the function $\loF_\theta: \PN \to \PN$ around the condition on $\theta \in \PPN^\N$ that $\emptyset \in \bigcup_{n \in \N} \theta(n)$.

\begin{proposition}\label{prop:degenerate}
Let $z \in \N$, $\theta \in \PPN^\N$ and $p \in \PN$.
\it\begin{itemize}
\item[(a)] If a degenerate sight is $(z,\theta,p)$-dedicated, then $\emptyset \in \bigcup_{n \in \N} \theta(n)$.
\item[(b)] If $\emptyset \in \bigcup_{n \in \N} \theta(n)$, then the sight $(\emptyset,\emptyset)$ is $(z,\theta,p)$-dedicated.
\end{itemize}\rm
\end{proposition}

\begin{proof}
(a)
We show the statement
\begin{center}
For every sight $S$, if $S$ is degenerate and is $(z,\theta,p)$-dedicated,\\
then $\emptyset \in \bigcup_{n \in \N} \theta(n)$.
\end{center}
by induction on $S$.

The case $S = \Nil$ is trivially okay.
Let us consider the case $S = (A,\sigma)$.

Suppose that $S = (\emptyset,\emptyset)$.
Then, since $S$ is $(z,\theta,p)$-dedicated, we have that $z = \la 1, \la n,\_ \ra \ra$ for some $n \in \N$ and that $\emptyset \in \theta(n)$.
Therefore $\emptyset \in \bigcup_{n \in \N} \theta(n)$.

Suppose that $S = (A,\sigma)$ with $A \neq \emptyset$.
Since $S$ is degenerate, there must be some $a_0 \in A$ such that the sight $\sigma(a_0)$ is degenerate.
By the IH, we have $\emptyset \in \bigcup_{n \in \N} \theta(n)$.

Induction complete. The conclusion (a) follows.

(b)
Immediate by definition.
\end{proof}

\begin{corollary}\label{cor:DegLoFp}
Let $\theta \in \PPN^\N$ and $p \in \PN$.
If $\emptyset \in \bigcup_{n \in \N} \theta(n)$, then $\loF_\theta(p) = \N$.
\end{corollary}

\begin{proof}
By Proposition \ref{prop:loF}, we have
$$\loF_\theta(p) = \{z \in \N \mid \text{there is a $(z,\theta,p)$-dedicated sight}\}.$$
So if $\emptyset \in \bigcup_{n \in \N}$,
then for each $z \in \N$ the sight $(\emptyset,\emptyset)$ is $(z,\theta,p)$-dedicated\footnote{This is Proposition \ref{prop:degenerate}(b)},
so $\loF_\theta(p) = \N$, as desired.
\end{proof}

\subsection{Well-founded trees, and another representation of $\lo_\ph$}
It is not entirely easy to write down a concrete construction of sight.
We now discuss an axiomatization of the set of nodes of a sight, which will facillitate defining a sight by specifying what its nodes are - a task rather considerable.

\begin{definition}
As usual, we denote by $\N^*$ the set of finite sequences in $\N$.
Given $s,t \in \N^*$, we write $s \ini t$ if $s$ is an initial segment of $t$.

A \defn{tree} (for us) is a \emph{non-empty} subset $T$ of $\N^*$ that is closed under initial segments, i.e. if $s \in T$ and $s' \in \N^*$ with $s' \ini s$ then $s' \in T$.
A tree is \defn{well-founded} if it admits no infinite chain w.r.t. the ordering $\ini$.
We may abbreviate `well-founded tree' to \defn{wf-tree}.

Let $T$ be a tree.
Given $t \in T$, we write $\Out_T(t) = \{x \in \N \mid t \cons x \in T\}$.

An element $s \in T$ may be called a \defn{node} of $T$.
A node $s$ of $T$ is a \defn{leaf} if $s$ is a maximal in $T$ w.r.t. $\ini$.
We denote by $\Lvs(T)$ the set of leaves in $T$.

Let $t \in T$.
We write $\subtree_T(t) = \{s \in \N^* \mid t \conc s \in T\} \subseteq \N^*$.
It is easy to see that $\subtree_T(t)$ is a tree.
We say that $\subtree_T(t)$ is the \defn{subtree} of $T$ at $t$.
\end{definition}

\begin{proposition}
A well-founded tree has a leaf.
\end{proposition}

\begin{proof}
If a tree $T$ has no leaf, then any node of $T$ can be extended to a longer node, which clearly implies that $T$ is not well-founded.
The Proposition is the contrapositive of this observation.
\end{proof}

\begin{definition}
Let $T$ be a well-founded tree.
The \defn{foundation number} of $T$, to be denoted $\fn(T)$, is the least length of a leaf of $T$.
\end{definition}

\begin{proposition}
The set of nodes of a non-degenerate sight is a well-founded tree.
Moreover, every well-founded tree is the node set of a unique non-degenerate sight.
\end{proposition}

\begin{proof}
We show that the association
$$\Nds: \{\text{Non-degenerate sights}\} \to \{\text{Well-founded trees}\}$$
is bijective.

Injectivity.
Suppose, for contradiction, that there are non-degenerate sights $S \neq S'$ with $\Nds(S) = \Nds(S')$.
We show the inconsistency of mathematics by induction on $S$.

The case $S = \Nil$.
Since $S \neq S'$, we have $S' = (A,\sigma)$ with (since $S'$ is nondegenerate) $A \neq \emptyset$.
Choose $a_0 \in A$.
Then the length 1 sequence $(a_0)$ belongs to $\Nds(S') = \Nds(S) = \{()\}$, so mathematics is inconsistent.

The case $S = (A,\sigma)$.
Choose $a_0 \in A$.
Applying the IH to the sight $\sigma(a_0)$, mathematics is inconsistent.
Induction complete.
This proves the injectivity.

Surjectivity.
We prove that for every well-founded tree $T$ there is a non-degenerate sight $S$ with $\Nds(S) = T$, by induction on the foundation number of $T$.

The case $\fn(T) = 0$.
Then $T = \{()\} = \Nds(\Nil)$, so okay.

The case $\fn(T) > 0$.
Then $\Out_T() \neq \emptyset$, and for each $a \in \Out_T()$ the foundation number of $\subtree_T(a)$ is $\fn(T)-1$.
So by the IH we can choose for each $a \in \Out_T()$ a non-degenerate sight $\sigma(a)$ with $\Nds(\sigma(a)) = \subtree_T(a)$.
Now clearly the pair $(\Out_T(),\sigma)$ is a non-degenerate sight and we have $\Nds(\Out_T(),\sigma) = T$, as desired.
Induction complete.
This proves the surjectivity.
Proof complete.
\end{proof}

\begin{remark}
Let $S$ be a \emph{non-degenerate} sight, and let $s \in \N^*$.
Observe that
\begin{itemize}
\item $s$ is a node of $S$ if and only if $s$ is a node of $\Nds(S)$,
\item $s$ is a leaf of $S$ if and only if $s$ is a leaf of $\Nds(S)$,
\item $\Nds(\Subsight_S(s)) = \Subtree_{\Nds(S)}(s)$,
\item $\Out_S(s) = \Out_{\Nds(S)}(s)$.
\end{itemize}
\end{remark}

\begin{notation}
Given a non-degenerate sight $S$, we may denote the node set of $S$ just by $S$.
\end{notation}

The point of view the last Proposition provides, namely to view a sight as a set of sequences, has stimulated the author to reconsider the way a natural number acts on sights, as in the following Definition and Proposition.

\begin{definition}
Let $w: \N^* \pto \N$ be a partial function, let $\theta \in \PPN^\N$ and let $p \in \PN$.
We say that a sight $S$ is \defn{$(w,\theta,p)$-supporting} if
\begin{itemize}
 \item for each leaf $s \in \Nds(S)$, we have $w(s) = \la 0,y \ra$ with $y \in p$, and
 \item for each non-leaf $s \in \Nds(S)$, we have $w(s) = \la 1,n \ra$ with $\Out_S(s) \in \theta(n)$.
\end{itemize}
We say that $S$ is \defn{$(w,\theta)$-supporting} if $S$ is $(w,\theta,\N)$-supporting.
\end{definition}

\begin{proposition}
Given $\theta \in \PPN^\N$, define a function $\loS_\theta: \PN \to \PN$ by
$$\loS_\theta(p) = \{w \in \N \mid \text{there is a $(w,\theta,p)$-supporting sight}\}$$
The function $\loS_\ph: \PPN^\N \to \PN^\PN$ represents the map $\lo_\ph: \nabla(\PPN^\N) \to \Lo$.
\end{proposition}

\begin{proof}
We proceed in a number of steps.

\textsc{Step 1.}
We define a partial computable function $\snd_\ph\ph: \N \times \N^* \pto \N$ satisfying
$$\snd_z(c_1,\ldots,c_p) \eqv
\begin{cases}
\la 0,y \ra                   & \text{if $z = \la 0, y \ra$ and $p = 0$} \\
\la 1,n \ra                   & \text{if $z = \la 1, \la n,e \ra \ra$ and $p = 0$} \\
\snd_{e(c_1)}(c_2,\ldots,c_p) & \text{if $z = \la 1, \la n,e \ra \ra$ and $p > 0$}.
\end{cases}$$
We show that
\begin{center}
if a sight $S$ is $(z,\theta,p)$-dedicated, then it is $(\snd_z,\theta,p)$-supporting.
\end{center}
This clearly follows from the statement
\begin{center}
For each $(x_1,\ldots,x_k) \in \N^*$, each $z \in \N$ and each sight $S$, \\
given that $(x_1,\ldots,x_k) \in \Nds(S)$ and that $S$ is $(z,\theta,p)$-dedicated, we have: \\
if $(x_1,\ldots,x_k)$ is a leaf of $S$ then $\snd_z(x_1,\ldots,x_k) = \la 0,y \ra$ with $y \in p$, and \\
if $(x_1,\ldots,x_k)$ is a non-leaf of $S$ then $\snd_z(x_1,\ldots,x_k) = \la 1,n \ra$ with $\Out_S(x_1,\ldots,x_k) \in \theta(n)$.
\end{center}
which prove by induction on $k$.

The case $p = 0$, i.e. $(x_1,\ldots,x_p) = ()$.
Suppose that $()$ is a leaf of $S$.
Then $S = \nil$.
Since $S$ is $(z,\theta,p)$-dedicated, we have $z = \la 0,y \ra$ with $y \in p$.
By definition of $\snd$, we have $\snd_z() = \la 0,y \ra$, as desired.
Suppose that $()$ is a non-leaf of $S$.
Then $S = (A,\sigma)$ for some $A,\sigma$.
Since $S$ is $(z,\theta,p)$-dedicated, we have $z = \la 1,\la n,e \ra \ra$ for some $n,e$, and have $\Out_S() \in \theta(n)$.
By definition of $\snd$, we have $\snd_z() = \la 1,n \ra$.
The case done.

The case $p > 0$.
Then $S = (A,\sigma)$ for some $A,\sigma$.
Since $(A,\sigma)$ is $(z,\theta,p)$-dedicated, we have that $z = \la 1, \la \_,e \ra \ra$ for some $e$, and that the sight $\sigma(x_1)$ is $\la e(x_1),\theta,p \ra$-dedicated.
If $(x_1,\ldots,x_k)$ is a leaf of $S$, then $(x_2,\ldots,x_k)$ is a leaf of $\sigma(x_1)$,
so by the IH we have
$$\snd_z(x_1,\ldots,x_k) = \snd_{e(x_1)}(x_2,\ldots,x_p) = \la 0,y \ra$$
with $y \in p$, as desired.
If $(x_1,\ldots,x_k)$ is a non-leaf of $S$, then $(x_1,\ldots,x_p)$ is a non-leaf of $\sigma(x_1)$,
so by the IH we have
$$\snd_z(x_1,\ldots,x_k) = \snd_{e(x_1)}(x_2,\ldots,x_p) = \la 1,n \ra$$
with $\Out_S(x_1,\ldots,x_k) = \Out_{\sigma(x_1)}(x_2,\ldots,x_k) \in \theta(n)$, as desired.
Induction complete.

\textsc{Step 2.}
Given a partial function $w: \N^* \pto \N$, define a partial function $\fst'_w: \N^* \pto \N$ by
$$\fst'_w(x_1,\ldots,x_k) \eqv
\begin{cases}
\la 0,y \ra                                             & \textif w(x_1,\ldots,x_k) = \la 0,y \ra \\
\la 1, \la n, \lambda x . \fst'_w(x_1,\ldots,x_k,x) \ra & \textif w(x_1,\ldots,x_k) = \la 1,n \ra.
\end{cases}$$
Note that if $w: \N^* \pto \N$ is effective, then so is $\fst'_w: \N^* \pto \N$.

We define a possibly-undefined element $\fst_w \in \N$ by
$$\fst_w \eqv \fst'_w().$$
Clearly, the partial function $\fwd_\ph: \Ptl(\N^*,\N) \pto \N$ is effective.
%Moreover, any tracking $\ulcorner \fst_\ph \urcorner$ of it must be a total index,
%because every natural number is a code of the partial function $\emptyset: \N^* \pto \N$ and we have $\fst_\emptyset \down$.

\textsc{Step 3.}
Given a partial computable function $w: \N^* \pto \N$ and $a \in \N$,
define a partial computable function $w@a: \N^* \pto \N$ by
$$w@a(x_1,\ldots,x_k) \eqv w(a,x_1,\ldots,x_k).$$
We prove the statement
\begin{equation}\label{lr2}
\text{\begin{tabular}{c}
For every sight $S$, for $k \in \N$ and $x_1,\ldots,x_k \in \N$ such \\
that the sight $S$ is $(\fst'_{w@a}(x_1,\ldots,x_k)\cvg,\theta,p)$-dedicated, \\
the sight $S$ is $(\fst'_w(a,x_1,\ldots,x_k)\cvg,\theta,p)$-dedicated.
\end{tabular}}
\end{equation}
by induction on $S$.

The case $S = \nil$.
Then $\fst'_{w@a}(x_1,\ldots,x_k) = \la 0,y \ra$ with $y \in p$.
By definition of $\fst'_{w@a}(x_1,\ldots,x_k)$, we have $w@a(x_1,\ldots,x_k) = \la 0,y \ra$.
Hence $w(x_1,\ldots,x_k) = \la 0,y \ra$, and so $\fst'_w(a,x_1,\ldots,x_k) = \la 0,y \ra$.
Thus $\nil$ is $(\fst'_w(a,x_1,\ldots,x_k),\theta,p)$-dedicated.

The case $S = (B,\tau)$.
Then $\fst'_{w@a}(x_1,\ldots,x_k) = \la 1, \la n, \lambda x. \fst'_{w@a}(x_1,\ldots,x_k,x) \ra \ra$ for some $n$, and we have $B \in \theta(n)$.
By definition of $\fst'_{w@a}(x_1,\ldots,x_k)$, it follows that $w@a(x_1,\ldots,x_k) = \la 1,n \ra$.
So $w(a,x_1,\ldots,x_k) = \la 1,n \ra$.
Hence $\fst'_w(a,x_1,\ldots,x_k) = \la 1, \la n, \lambda x. \fst'_w(a,x_1,\ldots,x_k,x) \ra \ra$.
Let $b \in B$.
We have $\fst'_{w@a}(x_1,\ldots,x_k,b)\cvg$, and $\tau(b)$ is $(\fst'_{w@a}(x_1,\ldots,x_k,b),\theta,p)$-dedicated.
By the IH, we have $\fst'_w(a,x_1,\ldots,x_k,b)\cvg$, and $\tau(b)$ is $(\fst'_w(a,x_1,\ldots,x_k,b),\theta,p)$-dedicated.
Hence $(B,\tau)$ is $(\fst'_w(a,x_1,\ldots,x_k),\theta,p)$-dedicated, as desired.
Induction complete.
This proves \eqref{lr2}.

\textsc{Step 4.}
We show that
\begin{center}
if a sight $S$ is $(w,\theta,p)$-supporting, then it is $(\fst_w,\theta,p)$-dedicated.
\end{center}
by induction on $S$.

The case $S = \nil$.
Then $()$ is the only node of $S$.
Since $S$ is $(w,\theta,p)$-dedicated and $()$ is a leaf of $S$, we have $w() = \la 0,y \ra$ with $y \in p$.
So $\fst_w \eqv \fst'_w() = \la 0,y \ra$.
Therefore $\fst_w\cvg$, and the sight $S$ is $(\fst_w,\theta,p)$-dedicated.

The case $S = (A,\sigma)$.
We want: (i) $\fst_w = \la 1,\la n,e \ra \ra$ for some $n,e$, (ii) $A \in \theta(n)$, and (iii) for each $a \in A$ one has $e(a)\cvg$ and the sight $\sigma(a)$ is $(e(a),\theta,p)$-dedicated.
Since $()$ is a non-leaf of $S$, we have $w() = \la 1,n \ra$ for some $n$ and have $\Out_S() \in \theta(n)$.
As $\fst_w = \fst'_w() = \la 1, \la n, \lambda x . \fst'_w(x) \ra \ra$, we have (i).
As $A = \Out_S()$, we have (ii).
Let us verify (iii).
Let $a \in A$.
Since $(A,\sigma)$ is $(w,\theta,p)$-supporting, the sight $\sigma(a)$ is $(w@a,\theta,p)$-supporting.
By the IH, as $\fst_{w@a} \eqv \fst'_{w@a}()$, we have $\fst'_{w@a}()\cvg$ and the sight $\sigma(a)$ is $(\fst'_{w@a}(),\theta,p)$-dedicated.
By \eqref{lr2}, we have $\fst'_w(a)\cvg$ and the sight $\sigma(a)$ is $(\fst'_w(a),\theta,p)$-dedicated.
As $\fst'_w(a) \eqv (\lambda x . \fst'_w(x))(a)$, this proves (iii).
Induction complete.

\textsc{Step 5.}
We now see
$$\la \ulcorner \snd_\ph \urcorner , \ulcorner \fst_\ph \urcorner \ra \realizes \forall \theta \in \PPN^\N \forall p \in \PN . \loF_\theta(p) \BI \loS_\theta(p).$$
So the functions $\loF_\ph,\loS_\ph: \PPN^\N \to \PN^\PN$ represent the same map $\nabla(\PPN^\N) \to \Lo$, which is the map $\lo_\ph$.
We are done.
\end{proof}

\subsection{Join of the preordering $\leq_\lo$}
As the first application of well-founded trees, we present a construction for joins in the preorder $(\PPN^\N,\leq_\lo)$.

% 
% \begin{lemma}
% Let $\lo_\ph: \PPN^\N \to \PN^\PN$ be any function representing the map $\lo_\ph$.
% We have 
% $\realizes \forall \theta . \lo_\theta \sr{}{\BI} \lo_{\theta^*}$. 
% \end{lemma}
% 

\begin{construction}[Concatenation of sights]
Given non-degenerate sights $S$ and $T$, define
$$S \conc T = \{s \conc t \mid s \in \Lvs(S) \en t \in T\}.$$
\claim{Then the set $S \conc T$ is a well-founded tree, hence a non-degenerate sight.}
\end{construction}

\begin{proof}
Easy.
\end{proof}

\begin{lemma}\label{JoinAction}
Given partial functions $a,b: \N^* \pto \N$, choose a partial function $a \conc b: \N^* \pto \N$ subject to the following, clearly effective action.

\begin{itemize}
\item[]
Let $(x_1,\ldots,x_l) \in \N^*$.

\item[]
If for all $0 \leq i \leq l$ one has $a(x_1,\ldots,x_i) = \la 1,\_ \ra$, \\
then $(a \conc b)(x_1,\ldots,x_l) = a(x_1,\ldots,x_l)$.

\item[]
If there is $0 \leq k \leq l$ such that
  \begin{itemize}
  \item for all $0 \leq i < k$ one has $a(x_1,\ldots,x_i) = \la 1,\_ \ra$,
  \item $a(x_1,\ldots,x_k) = \la 0,\_ \ra$,
  \item for all $k \leq i \leq l$ one has $b(x_{k+1},\ldots,x_i) = \la 1,\_ \ra$,
  \end{itemize}
then $(a \conc b)(x_1,\ldots,x_l) = b(x_{k+1},\ldots,x_l)$.

\item[]
If there is $0 \leq k \leq l$ such that
  \begin{itemize}
  \item for all $0 \leq i < k$ one has  $a(x_1,\ldots,x_i) = \la 1,\_ \ra$,
  \item $a(x_1,\ldots,x_k) = \la 0,y \ra$,
  \item for all $k \leq i < l$ one has $b(x_{k+1},\ldots,x_i) = \la 1,\_ \ra$,
  \item $b(x_{k+1},\ldots,x_l) = \la 0,z \ra$,
  \end{itemize}
then $(a \conc b)(x_1,\ldots,x_l) = \la 0,\la y,z \ra\ra$.
\end{itemize}

Let $\theta \in \PPN^\N$, $p,q \in \PN$ and $a,b \in \N$.
Let $S,T$ be non-degenerate sights such that $S$ is $(a,\theta,p)$-supporting and $T$ is $(b,\theta,q)$-supporting.
\claim{Then $S \conc T$ is $(a \conc b,\theta,p \wedge q)$-supporting.}
\end{lemma}

\begin{proof}
The construction of $a \conc b$ is tailored to satisfy this property - one can best verify this in mind by staring at the definition of $a \conc b$ above.
\end{proof}

\begin{proposition}
Given $\A,\B \in \PsPN$, we define
$$\A \owedge \B = \{A \wedge B \mid A \in \A \en B \in \B\} \in \PsPN.$$
Let $\eta,\theta \in \PsPN^\N$.
We define a function $(\eta \owedge \theta): \N \to \PsPN$ by
$$(\eta \owedge \theta)(n) = \eta(n) \owedge \theta(n).$$
\claim{Then the function $\eta \owedge \theta$ is the join of $\eta$ and $\theta$ in $(\PsPN^\N,\leq_\lo)$.}
\end{proposition}

\begin{proof}
First we show that $\eta \leq_\lo \eta \owedge \theta$.
It suffices to have $\eta \leq_\mo \eta \owedge \theta$, i.e. that $\realizes \forall n . \{n\} \RI \forall_{A \in \eta(n)} \exists m \exists_{B \in (\eta \owedge \theta)(n)} . \{m\} \wedge (B \RI A)$, i.e. that $\realizes \forall n . \{n\} \RI \forall_{A \in \eta(n)} \exists m \exists_{A' \in \eta(n)} \exists_{B' \in \theta(n)} \{m\} \wedge (A' \wedge B' \RI A)$.
But clearly $\lambda n . \la n, \pi^2_1 \ra$ realizes the last sentence, as desired.
Analogously we have $\theta \leq_\lo \eta \owedge \theta$.

Then we show that for all $\zeta \in \PPN^\N$ if $\eta \leq_\lo \zeta$ and $\theta \leq_\lo \zeta$ then $\eta \owedge \theta \leq_\lo \zeta$.
The inequality $\eta \leq_\lo \zeta$ means that some $\alpha \realizes \forall n . \{n\} \RI \forall_{A \in \eta(n)} \loS_\zeta(A)$, and the inequality $\theta \leq_\lo \zeta$ means that some $\beta \realizes \forall n . \{n\} \RI \forall_{A \in \theta(n)} \loS_\zeta(A)$.
The inequality $\eta \owedge \theta \leq_\lo \zeta$ means
\begin{equation}\label{ow1}
\realizes \forall n . \{n\} \RI \forall_{A \wedge B \in (\eta \owedge \theta)(n)} \loS_\zeta(A \wedge B).
\end{equation}
By the previous Lemma, we have that given $n \in \N$, $A \in \eta(n)$ and $B \in \theta(n)$ if $a \in \loS_\zeta(A)$ and $b \in \loS_\zeta(B)$ then $a \conc b \in \loS_\zeta(A \wedge B)$.
Therefore $\lambda n . \alpha(n) * \beta(n)$ realizes \eqref{ow1}, as desired.
This proves the proposition.
\end{proof}

% 
% \begin{corollary}
% The function $\owedge: \PPN \times \PPN \to \PPN$ represents a join operation on $(\nabla(\PPN),\leq_\lo)$.
% \end{corollary}
% 
% \begin{proof}
% 
% \end{proof}

\subsection{$\theta$-Realizability}
Using the language of sights, we can provide a realizability-style semantics for the arithmetic in subtoposes of $\Eff$.
Let $\theta \in \PPN^\N$ such that $\lo_\theta$ is non-degenerate.

\begin{definition}
We define inductively a relation `$\theta$-realizes' between the natural numbers and the arithmetic \emph{sentences}.
Let $n \in \N$.
\begin{itemize}
 \item Let $\sigma,\tau$ be closed terms in the language of arithmetic. We say that $n$ \defn{$\theta$-realizes} the sentence $\sigma = \tau$ if $\dbl^\N \sigma \dbr = n = \dbl^\N \tau \dbr$, where the latter denote the interpretation of the terms $\sigma,\tau$ in the standard model $\N$.
\end{itemize}
Let $\phi,\psi$ be arithmetic sentences.
\begin{itemize}
 \item We say that $n$ \defn{$\theta$-realizes} $\phi \wedge \psi$ if, writing $n = \la n_1,n_2 \ra$, $n_1$ $\theta$-realizes $\phi$ and $n_2$ $\theta$-realizes $\psi$.
 \item We say that $n$ \defn{$\theta$-realizes} $\phi \vee \psi$ if, writing $n = \la g,m \ra$, either, $g = 0$ and $m$ $\theta$-realizes $\phi$, or, $g = 1$ and $m$ $\theta$-realizes $\psi$.
 \item We say that $n$ \defn{$\theta$-realizes} $\phi \RI \psi$ if for every $\theta$-realizer $m$ of $\phi$ there is a $(n(m),\theta)$-dedicated sight $S$ such that every $n(m)$-value of $S$ $\theta$-realizes $\psi$.
 \item We say that $n$ \defn{$\theta$-realizes} $\forall x . \phi(x)$ if for every $x \in \N$ there is a $(n,\theta)$-dedicated sight $S$ such that for every $n$-value $m$ of $S$ there is a $(m(x),\theta)$-dedicated sight $T$ whose $m(x)$-values $\theta$-realize $\phi(x)$.
 \item We say that $n$ \defn{$\theta$-realizes} $\exists x . \phi(x)$ if there are $x,m \in \N$ such that $n = \la x,m \ra$ and $m$ $\theta$-realizes $\phi(x)$.
\end{itemize}
\end{definition}

\begin{proposition}
An arithmetic sentence $\phi$ is true in $\Eff_\theta$ if and only if some number $\theta$-realizes $\phi$.
\end{proposition}

\begin{proof}
The inductive clauses in the definition of `$n$ $\theta$-realizes $\phi$' are designed to satisfy that
\begin{center}
$n$ $\theta$-realizes $\phi$ if and only if\footnote{Recall the definition of $\el^{\ET_\theta/\Eff_\theta} \ldots \er$ from Construction \ref{con:CPtoP}, and the definition of $\el^{\ET/\ET_\theta} \ldots \er$ from Corollary \ref{cor:NonDegSen}.} $n \in \dbl \el^{\ET/\ET_\theta} \el^{\ET_\theta/\Eff_\theta} \phi \er \er \dbr$.
\end{center}
On the other hand, we have: $\Eff_\theta \models \phi$ iff $\ET_\theta \models \el^{\ET_\theta/\Eff_\theta} \phi \er$ iff $\ET \models \el^{\ET/\ET_\theta} \el^{\ET_\theta/\Eff_\theta} \phi \er \er$ iff $\dbl \el^{\ET/\ET_\theta} \el^{\ET_\theta/\Eff_\theta} \phi \er \er \dbr \neq \emptyset$.
The conclusion follows.
\end{proof}

\begin{remark}
Note that the definition of $\theta$-realizability does not involve any category-theoretic termininology, and is written purely in elementary terms using the notion of sights.
This means that the last Proposition enables one to study the arithmetic in a subtopos of $\Eff$ without reference to category theory.
\end{remark}

\section{Examples of basic topologies}\label{sec:examples}
\subsection{Extreme examples}
First we characterize which collections $\in \PPN$ give the least local operator $\id$.

\begin{calculation}
Let $\A \in \PPN$.
We have, $\lo_\A = \id$ as maps $\Omega \to \Omega$ if and only if $\bigcap \A \neq \emptyset$.\footnote{Remind that $\bigcap \{\} = \N$.}
\end{calculation}

\begin{proof}
Let us show `if'.
Choose $a_0 \in \bigcap \A$.
Recall that
$$\gr_\A(p) = \exists_{A \in \A} (A \BI p).$$
Thus $\lambda x . \la (\lambda \_ . x), (\lambda \_ . a_0) \ra \realizes \forall p . \id(p) \RI \gr_\A(p)$,
and $\lambda \la e,\_ \ra . e(a_0) \realizes \gr_\A(p) \RI \id(p)$.
Hence $\gr_\A = \id$ as maps $\Omega \to \Omega$.
Therefore $\gr_\A$ is a local operator, and it follows that $\lo_\A = \gr_\A = \id$ as maps $\Omega \to \Omega$.

Now we show `only if'.
Since $\lo_\A = \id$ is the least topology, we have $\A \leq_\lo \{\}$.
In other words, there is $z \in \N$ such that for each $A \in \A$ there is a $(z,\{\},A)$-dedicated sight $S_A$ into $A$.
It follows that each $S_A = \nil$.
But the empty sequence $()$ is the only leaf on the sight $\nil$.
Therefore the $z$-value of $()$ belongs to each $A \in \A$, so $\bigcap \A \neq \emptyset$ as desired.
\end{proof}

Next we characterize which sequences $\in \PPN^\N$ give the largest local operator $\top$.

\begin{calculation}
Let $\theta \in \PPN^\N$.
The following are equivalent.
\begin{itemize}
 \item[(i)] $\theta$ is a top element of $(\PPN^\N,\leq_\mo)$
 \item[(ii)] $\theta$ is a top element of $(\PPN^\N,\leq_\lo)$
 \item[(iii)] $\emptyset \in \bigcup_{n \in \N} \theta(n)$
\end{itemize}
It follows that the local operator $\top$ is basic: for instance, the collection $\{\emptyset\} \in \A$ gives it.
\end{calculation}

\begin{proof}
(i) \implies (ii):
Clear.

(iii) \implies (i): 
Assume (iii).
Let $\zeta \in \PPN^\N$.
To conclude (i), we show that $\zeta \leq_\mo \theta$, i.e. that
\begin{equation}\label{eq:tq1}
\realizes \forall n . \{n\} \RI \forall_{A \in \zeta(n)} \exists m . \{m\} \wedge \exists_{B \in \theta(m)} (B \RI A).
\end{equation}
Choose $m_0 \in \N$ such that $\emptyset \in \theta(m_0)$.
Then $\lambda \_ . \la m_0 , \_ \ra$ clearly realizes \eqref{eq:tq1}, as desired.

(ii) \implies (iii):
By (ii), we have $\{\emptyset\} \leq_\lo \theta$.
It follows that there is a number $z \in \N$ and a $(z,\theta,\emptyset)$-dedicated sight $S$.
Thus $S$ cannot have a leaf, and so\footnote{Every non-degenerate sight has a leaf!} $S$ is a degenerate sight.
By Proposition \ref{prop:degenerate}(a), we have $\emptyset \in \bigcup_{n \in \N} \theta(n)$, as desired.
\end{proof}

Next, let us see that the double negation topology is also basic.
First we recall a lemma from \cite{hyl82}.

\begin{lemma}
Let $j: \PN \to \PN$ be a function representing a topology on $\Eff$.
We have $\negneg \leq j$ if and only if $\bigcap \{j(p) \mid p \in \PsN\}$ is non-empty.
\end{lemma}

\begin{proof}
This is \cite[Lemma 16.2]{hyl82}.
\end{proof}

Now we are ready to single out a class of collections $\in \PPN$ that give the topology $\negneg$.

\begin{proposition}
Let $\A \in \PPN$, such that $\A$ contains two, recursively separable\footnote{Sets $A,B \subseteq \N$ are \defn{recursively separable} if there is a recursive set $C$ with $A \subseteq C$ and $B \subseteq \N \bs C$.} sets.
Then $\negneg \leq \lo_\A$ as maps $\Omega \to \Omega$.
\end{proposition}

\begin{proof}
By the previous Lemma, we may show that $\bigcap \{\loS_\A(p) \mid p \in \PsN\}$ contains a number.
To this end, it suffices to construct
\begin{itemize}
\item a partial effective function $w: \N^* \to \N$, and
\item for each $y \in \N$ a $(w,\A,\{y\})$-supporting wf-tree,
\end{itemize}
for then (any index of) $w$ belongs to $\bigcap \{\loS_\A(p) \mid p \in \PsN\}$.
Let us construct these.

Choose sets $A,B \in \A$ that are recursively separable.
Choose a partial computable function $w: \N^* \to \N$ subject to following conditions.
\begin{quote}
We have $w() = \la 1,0 \ra$.

Let $(x_0,\ldots,x_k) \in \N^*$ with $k \geq 0$.

If $x_0,\ldots,x_k \in A$, then $w(x_0,\ldots,x_k) = \la 1,0 \ra$.

If $x_0,\ldots,x_{k-1} \in A$ and $x_k \in B$, then $w(x_0,\ldots,x_k) = \la 0,k \ra$.
\end{quote}
For $y \in \N$, let
\begin{align*}
S_y &= \{()\} \\
    &\cup \{(x_0,\ldots,x_k) \mid 0 \leq k < y \en x_0,\ldots,x_k \in A\} \\
	&\cup \{(x_0,\ldots,x_y) \mid x_0,\ldots,x_{y-1} \in A \en x_y \in B\}.
\end{align*}
Now clearly, $S_y$ is a well-founded tree, and is $(w,\A,\{y\})$-supporting.
Done.
\end{proof}

\begin{corollary}\label{cor:RecSepNN}
Let $\A \in \PPsN$, and suppose that $\A$ contains two, recursively separable sets.
Then $\negneg = \lo_\A$ as maps $\Omega \to \Omega$.
\end{corollary}

\begin{proof}
Since $\emptyset \notin \A$, the local operator $\lo_\A$ is non-degenerate.
By the previous Proposition, it must then be $\negneg$.
\end{proof}

%\begin{calculation}
%Let $\A \in \PPsN$, such that $\A$ contains two, recursively separable\footnote{Sets $A,B \subseteq \N$ are \defn{recursively separable} if there is a recursive set $C$ with $A \subseteq C$ and $B \subseteq \N \bs C$.} sets.
%Then $\negneg \geq \lo_\A$ as maps $\Omega \to \Omega$.
%\end{calculation}
%
%\begin{proof}
%Left to the reader.\footnote{Proof to be supplied.}
%\end{proof}

%\begin{exercise}
%Give an example of $\A \in \PPN$ that gives the topology $\negneg$, while not satisfying the condition we just considered.
%Try to give a more general condition (even necessary?) on $\A \in \PPN$ in order that $\lo_\A$ be the topology $\negneg$.
%\end{exercise}

%\begin{calculation}
%$\negneg \leq_\lo \theta$ if and only if there is $x$ s.t. for every $p \in \P^*\N$ there is $(x,\theta)$-dedicated sight into $p$.
%\end{calculation}

The discussion so far has shown that the `extreme' local operators $\id,\top,\negneg$ are basic.
Clearly, we can easily come up with a collection $\A \in \PPN$ that does not give $\id$ or $\top$, and it seems not difficult to avoid $\negneg$.
At this point, one might decide to challenge his imagination and produce such examples of collections $\in \PPN$.
The next subsection follows this line of thoughts, and introduces examples of $\A \in \PPN$ that will give a `non-trivial' topology.

\subsection{Further examples}
We introduce now some subcollections of $\PN$. (But it does not mean that we will use all of them.)

\begin{definition}
Let $\A \in \PPN$.
Define
$$\mathcal{A}^* = \{(\bigcup \mathcal{A}) \bs A \mid A \in \mathcal{A}\}.$$
We say that $\mathcal{A}^*$ is the \defn{dual} collection of $\A$.
Note that $|\mathcal{A}^*| = |\mathcal{A}|$.
\end{definition}

\paragraph{Countable simple graphs.}
\begin{itemize}
 \item For $m \geq 2$, let $\mathcal{L}_m = \{\{i,i+1\} \mid i+1 < m\}$. This is the ``line of $m$ vertices''.
 \item For $m \geq 3$, let $\mathcal{C}_m = (L_m \cup \{\{0,m-1\}\})$. This is the ``circle of $m$ vertices''.
 \item For $m \geq 2$, let $\mathcal{K}_m = \{\{i,j\} \mid 0 \leq i < j < m\}$. This is the ``complete graph of $m$ vertices''.
\end{itemize}
Note that by Corollary \ref{cor:RecSepNN}, these examples all give the topology $\negneg$.
The trick, for obtaining a more interesting example, is to take their duals.

\paragraph{The co-$m$-tons.}
\begin{itemize}
%  \item For $k \geq 1$, for $k \leq m \leq \omega$, let $\mathcal{T}^k_m = \{\text{the $k$-tons $\subseteq m$}\}$.
%  \item For example, $\mathcal{T}^k_\omega = \{\text{the $k$-tons in $\N$}\}$.
%  \item Note that $\mathcal{T}^2_m = \mathcal{K}_m$.
 \item Let $1 < 2m < \alpha \leq \omega$. We write $\mathcal{O}^\alpha_m = \{\text{the co-$m$-tons $\subseteq \alpha$}\}$. We write $\O^\omega = \O_1^\omega$.
\end{itemize}

\paragraph{The finites, the cofinites.}
\begin{itemize}
 \item Let $\F = \{\text{the finite subsets of $\N$}\}$. So $\F^* = \{\text{the cofinite subsets of $\N$}\}$.
 \item Let $\ups \N = \{\ups n \mid n \in \N\}$, where $\ups n := \{m \in \N \mid n \leq m\}$.\footnote{This is the example of Pitts \cite[Example 5.8]{pitts81} mentioned in the Introduction.}
\end{itemize}
% 
% \paragraph{Interesting assemblies.}
% \begin{itemize}
%  \item ...
% \end{itemize}
% 
% \paragraph{Recursion-theoretic examples.}
% \begin{itemize}
%  \item collections of index sets, ...
% \end{itemize}

% \subsection{Elementary things about $\leq_\mo$}
% \begin{fact}
% Let $\mathcal{A},\mathcal{B} \subseteq \PN$.
% We have $\mathcal{A} \leq_\mo \mathcal{B}$ if and only if $\realizes \forall_{A \in \mathcal{A}} \exists_{B \in \mathcal{B}} (B \ri A)$.
% \end{fact}
% 
% \begin{fact}
% If $\A \subseteq \B$, then $\A \leq_\mo \B$. \qed
% \end{fact}

\subsection{Elementary $\leq_\mo$ calculations}
\begin{calculation}\label{incl=>leqmo}
$\F^* \cong_\mo \ups \N$.
\end{calculation}

\begin{proof}
Notice that
\begin{itemize}
 \item `$\leq_\mo$' means $\realizes \forall_\text{cofinite} A \exists n (\ups{n} \ri A)$,
 \item `$\geq_\mo$' means $\realizes \forall n \exists_\text{cofinite} B (B \ri \ups{n})$.
\end{itemize}
Now, observe that $\id$ is a realizer for both cases.
\end{proof}

\begin{proposition}
Let $m \geq 2$.
$\K_m <_\mo \K_{m+1}$.
\end{proposition}

\begin{proof}
Since $\K_m \subseteq \K_{m+1}$, we have `$\leq_\mo$'.
Let us show `$\not\geq_\mo$'.
Suppose, for contradiction, that some
$$\gamma \realizes \forall A \in \K_{m+1} \exists B \in \K_m (B \ri A).$$

Case 1: $\gamma$ is not injective on $\{0,\ldots,m-1\}$, i.e. $|\gamma\{0,\ldots,m-1\}| < m$.
Then choose distinct numbers $0 \leq a,b \leq m$, and observe that $\gamma \not\realizes \exists B \in K_m (B \ri \{a,b\})$.
Contradiction.

Case 2: $\gamma$ is injective on $\{0,\ldots,m-1\}$.
Let $x$ be the unique element of $\{0,\ldots,m\} \bs \gamma(\{0,\ldots,m-1\})$.
Choose $a \in \gamma\{0,\ldots,m-1\}$.
There exists $B \in K_m$ such that $\gamma \realizes B \ri \{a,x\}$.
This is impossible, as $\gamma(B) = \{b \neq c\}$ with $b,c \in \gamma\{0,\ldots,m-1\}$.
Done.
\end{proof}

\section{Turing degree topologies}\label{sec:turing}
\subsection{Introduction}
\begin{definition}
For $D \subseteq \N$, we write
$$\rho_D(n) =
\begin{cases}
\{\{0\}\} & \textif n \in D, \\
\{\{1\}\} & \otherwise.
\end{cases}$$
We might say that $\rho_D \in \PPN^\N$ gives the \softdefn{Turing degree topology} of $D$.
\end{definition}

These topologies were first studied in \cite{hyl82} and \cite{phoa89}.
We will establish below (Remark \ref{HylandTuring}) that our definition indeed gives the topologies as defined in \cite{hyl82}.
For now, assuming this correlation, we will remind ourselves of some results known about these topologies.

First of all, the following proposition from \cite{hyl82} is fundamental, justifying the name `Turing degree topology'.

\begin{proposition}
The association $D \mapsto \rho_D$ induces an order-embedding of Turing degrees in the lattice of topologies on $\Eff$.
In other words, given $D,E \subseteq \N$, we have $D \leq_\T E$ if and only if $\rho_D \leq_\lo \rho_E$.
\end{proposition}

\begin{proof}
This is \cite[Theorem 17.2]{hyl82}.
\end{proof}

%The following observation from \cite{phoa89} requires acquaintance with the notion of pca, which we have not dealt with in our treatment.
%Nevertheless it is relevant from the viewpoint of studying subtoposes of realizability toposes, so we state it.
%One can learn about pcas in the first chapter of \cite{jvo08}.
%
%\begin{proposition}
%Let $D \subseteq \N$.
%The subtopos of $\Eff$ given by $\rho_D$ is (equivalent to) the realizability topos $\RT(\K_D)$.
%\end{proposition}
%
%\begin{proof}
%This is \cite[Proposition 1]{phoa89}.
%\end{proof}

A result from \cite{phoa89} states that $\negneg$ is the least topology greater than all the $\rho_D$, as follows.

\begin{proposition}\label{Phoa3}
Let $\theta \in \PPN$.
If $\rho_D \leq_\lo \theta$ for all $D \subseteq \N$, then $\negneg \leq \lo_\theta$.
\end{proposition}

\begin{proof}
This is \cite[Proposition 3]{phoa89}.
\end{proof}

\subsection{Comparison to basics} 
Then we present an observation of our own, that a basic topology below some Turing degree topology is necessarily the bottom topology $\id$.

\begin{proposition}\label{prop:TuringBasic}
Let $\A \in \PPN$ and $D \subseteq \N$.
If $\A \leq_\lo \rho_D$, then $\A \cong_\lo \{\}$.
\end{proposition}

\begin{proof}
First we prove the statement
\begin{equation}\label{eq:tb1}
\text{Given sights $S,T$ and $z \in \N$, if $S,T$ are $(z,\rho_D)$-dedicated, then $S = T$.}
\end{equation}
by induction on $S$.

The case $S = \nil$.
Then $z = \la 0,\_ \ra$.
It follows that $T = \nil$ as well, as desired.

The case $S = (U,\sigma)$.
Then (i) $z = \la 1,\la n,e \ra \ra$ for some $n,e \in \N$, (ii) $U = \{\chi_D(n)\}$, (iii) $e$ is defined on $\chi_D(n)$ and (iv) the sight $\sigma(\chi_D(n))$ is $(e(\chi_D(n)),\rho_D)$-dedicated.
As $T$ is also $(z,\rho_D)$-dedicated, it follows that $T = (\{\chi_D(n)\},\tau)$ for some $\tau$ and the sight $\tau(\chi_D(n))$ is $(e(\chi_D(n)),\rho_D)$-dedicated.
By the IH, we have $\sigma(\chi_D(n)) = \tau(\chi_D(n))$.
It follows that $S = T$, as desired.

Let us now establish the proposition.
Take $z \realizes \A \leq_\lo \rho_D$.
Choose, for each $A \in \A$, a $(z,\rho_D,A)$-dedicated sight $S_A$.
By \eqref{eq:tb1}, all the $S_A$ are the same sight, say $S$.
Since $S$ is on $\rho_D$, it is non-degenerate.
So we can choose a leaf $d$ of $S$.
Then $z(d) \in \bigcap \A$.
It follows that $\A \cong_\lo \{\}$, as desired.
\end{proof}

% It follows that the Turing degree topologies, with the only exception being the topology of decidable sets, are not basic.

\begin{corollary}
If $D \subseteq \N$ is undecidable, then $\rho_D$ is not basic. \qed
\end{corollary}

\subsection{Notion of a set $\subseteq \N$ being effective in a subtopos of $\Eff$}

%But before getting to that point, I wish to put some perspective on what it may mean to compare topologies (e.g. the basics) with the Turing degree ones, through the next subsection.

First, let us free ourselves from the presentational distinction between partial maps and partial functional relations.

\begin{remark}
We pass through the coincidence of partial maps and partial functional relations in a topos, and its compatibility with sheafification.

In a topos $\E$, a partial map $X \pto Y$ is precisely a partial functional relation $X \pto Y$, as follows.
If $(U,f)$ is a partial map, then the relation
$$x^X,y^Y . x \in U \wedge f(x) = y$$
is a partial functional relation.
If $G \in \Sub(X \times Y)$ is a partial functional relation,
let $U = \{x:X \mid \exists y \in Y . (x,y) \in G\}$,
and let $g$ be the map $U \to Y$ induced by the fact
$$\E \models \forall u \in U \exists! y \in Y. (\incl_{U \subseteq X}(u),y) \in G,$$
then $(U,g)$ is a partial map $X \pto Y$.
The two directions we have described are mutual inverses.

Let $j$ be a local operator on $\E$, and let $\a: \E \onto \Sh_j(\E)$ be the sheafification.
If $(U,f)$ is a partial map $X \pto Y$, then $(\a(U),\a(f))$ is clearly a partial map $\a(X) \pto \a(Y)$.
If $G \in \Sub_\E(X \times Y)$ is a partial functional relation, i.e. satisfies
$$\forall x \forall y,y' . (x,y) \in G \wedge (x,y') \in G \RI y = y'$$
an entailment of two geometric formulas,
then as $\a$ preserves this we have that $\a(G) \in \Sub_{\Sh_j(E)}(\a(X) \times \a(Y))$ is a partial functional relation $\a(X) \pto \a(Y)$.

The final remark is that the bijection between partial maps and partial functional relations commutes with sheafification, in the now obvious sense.
\end{remark}

\begin{proposition}\label{a_and_j}
In a topos $\E$, let $X$ be an object and $U,V$ its subobjects.
Let $j$ be a local operator on $\E$, and denote by $\a: \E \to \Sh_j(\E)$ the sheafification.
We denote (as usual) by $\a(U)$ the obvious subobject of $\a(X)$, and by $j(U)$ the $j$-closure of the subobject $U$.
\claim{We have
\begin{center}
$\a(U) \leq \a(U)$ in $\Sub_{\Sh_j(\E)}(\a{X})$ if and only if $j(U) \leq j(V)$ in $\Sub_\E(X)$.
\end{center}}
\end{proposition}

\begin{proof}
We know from topos theory that \\
(i) the assignment $\ClSub_\E(X) \xto{\a} \Sub_{\Sh_j(\E)}(\a{X})$ is an order isomorphism, and
(ii) the assignment $\Sub_\E(X) \xto{\a} \Sub_{\Sh_j(\E)}(\a{X})$ factors as $\Sub_\E(X) \xto{\text{take $j$-closure}} \ClSub(X) \xto{\a} \Sub_{\Sh_j(\E)}(\a{X})$.
The Proposition follows.
\end{proof}

\begin{definition}
In a topos $\E$, let $(D,g): X \pto Y$ be a partial map, and let $U \in \Sub(X)$.
We say that the partial map $(D,g)$ is \defn{defined on $U$} if $U \leq D$ in $\Sub(X)$.
Let $j$ be a local operator on $\E$.
We say that the partial map $(D,g)$ is \defn{defined on $U$ in $j$} if the partial map $(\a(D),\a(g))$ is defined on $\a(U)$ in $\Sh_j(\E)$.
\end{definition}

\begin{corollary}
Let $j,j'$ be local operators on a topos $\E$, with $j \leq j'$.
Let $(D,g)$ be a partial map $X \pto Y$ in $\E$, and let $U \in \Sub(X)$.
If $(D,g)$ is defined on $U$ in $j$, then $(D,g)$ is defined on $U$ in $j'$.
\end{corollary}

\begin{proof}
Note that, $(D,g)$ is defined on $U$ in $j$ if and only if $\a_j(U) \leq \a_j(D)$ if and only if (by Proposition \ref{a_and_j}) $U \leq j(D)$.
On the other hand, $(D,g)$ is defined on $j'$ if and only if $U \leq j'(D)$.
But $j(D) \leq j'(D)$, hence the result follows.
\end{proof}

\begin{construction}
Let $k \in \N$, and let $D \subseteq \N^k$.
We describe two constructions of a subobject $D_\Eff \in \Sub(\NNO^k)$ in $\Eff$.

On one hand, we have the function $\nab{D}: \N^k \to \PN$.
This is automatically an extensional predicate on the object $\NNO^k$ in $\Eff$, hence determines a subobject.

On the other hand, we may consider the assembly $(D,\singt{\cdot})$.
The function $\incl: D \emto \N^k$ clearly represents a mono $(D,\singt{\cdot}) \emto \NNO^k$, hence determines a subobject.

The claim is that the two subobjects are the same.
We denote by $D_\Eff$ this subobject of $\NNO^k$ in $\Eff$.
Moreover if $j$ is a local operator on $\Eff$, we denote by $D_j$ the subobject $\a_j(D_\Eff)$ of $\a_j(N^k) \cong \a_j(N)^k$.
\end{construction}

\begin{proof}
The second subobject is clearly represented by the formula in $\EffTrip$
$$\phi(\vec{n}) := \exists \vec{m} \in D . \singt{\vec{m}} \wedge \incl(\vec{m}) \sim \vec{n}.$$
It suffices that $\phi$ and $\nab{D}$ are equivalent as predicates on $N$, i.e. that
$$\realizes \forall \vec{n} \in \N^k . \singt{\vec{n}} \RI [\phi(\vec{n}) \BI \nab{D}(\vec{n})].$$
Now $\lambda n . \la (\lambda x . n),(\lambda x . n) \ra$ realizes this.
\end{proof}

\begin{proposition}
Let $g$ be a partial function $\N \pto \N$.
The subobject $g_\Eff$ of $\NNO \times \NNO$ in $\Eff$ is a partial functional relation.
\end{proposition}

\begin{proof}
We have to verify the single-valuedness requirement, that
\begin{itemize}
 \item[] $\realizes \forall x,y,y' \in \N . \{x\} \wedge \{y\} \wedge \{y'\} \RI [\nab{g}(x,y) \wedge \nab{g}(x,y') \RI y \sim y']$.
\end{itemize}
But its realizer is easily found.
\end{proof}

\begin{definition}
Let $g: \N \pto \N$ be a partial function, let $R \subseteq \N$, and let $j$ be a local operator on $\Eff$.
We say that $g$ is \defn{defined on $R$ in $j$} if $g_j$ is defined on $R_j$.
We say that $g$ is \defn{effective in $j$} if $g$ is defined on $\Dom(g)$ in $j$.
\end{definition}

\begin{proposition}
Let $g,j,R$ be as in the previous Definition.
Suppose that $j \neq \top$.
If $R \bs \Dom(g) \neq \emptyset$, then $g$ is not defined on $R$ in $j$.
\end{proposition}

\begin{proof}
Assume $R \bs \Dom(g) \neq \emptyset$.
If $g$ would be defined on $R$ in $j$, then, as $j \leq \negneg$, $g$ is defined on $R$ in $\negneg$, i.e., the partial function $g: \N \pto \N$ is defined on $R$ in the ordinary sense, contradicting the assumption.
\end{proof}

\begin{proposition}
Let $D \subseteq \N$, and let $j: \PN \to \PN$ be a function representing a local operator on $\Eff$.
The following are equivalent.
\begin{itemize}
 \item[(i)] The characteristic function $\N \to \N$ of $D$ is effective in $j$.
 \item[(ii)] $D_j$ is decidable.\footnote{In a topos, a subobject $U$ of an object $X$ is \defn{decidable} if $X \leq U \vee \neg U$ in $\Sub(X)$.}
 \item[(iii)] $\rho_D \leq j$ as local operators.
 \item[(iv)] $\ET \models \forall n^\N . \{n\} \RI j(\{\chi_D(n)\})$.
\end{itemize}
We say that $D$ is \defn{effective in} the subtopos $j$ if these conditions are satisfied.
\end{proposition}

\begin{proof}
We will show that each of (i),(ii),(iii) is equivalent to (iv).

`(i)$\BI$(iv)'.
We have: (i) holds, if and only if the partial functional relation $(\chi_D)_j \in \Sub(\NNO \times \NNO)$ in $\Eff_j$ is total, if and only if $\Eff_j \models \forall x^N \exists y^N . (x,y) \in (\chi_D)_j$, if and only if
\begin{equation}\label{sf1}
\ET_j \models \forall x^\N \{x\} \RI \exists y^\N . \{y\} \wedge \Delta_{\chi_D}(x,y).
\end{equation}
Note that
$$\dbl^{\ET_j} \exists y^\N . \{y\} \wedge \Delta_{\chi_D}(x,y) \dbr(x) = \{\chi_D(x)\} \wedge \N,$$
so that
$$[x . \exists y^\N . \{y\} \wedge \Delta_{\chi_D}(x,y)] \eqq_j [x . \{\chi_D(x)\} \wedge \N] \eqq_j [x . \{\chi_D(x)\}].$$
Thus \eqref{sf1} is equivalent to
$$\ET_j \models \forall n^\N \{n\} \RI \{\chi_D(n)\},$$
which in turn is equivalent to (iv).

`(ii)$\BI$(iv)'.
We have: (ii) holds, if and only if
$$\Eff_j \models \forall n \in \NNO . n \in (D_j \vee \neg D_j),$$
if and only if
$$\ET_j \models \forall n^\N . \{n\} \RI (\nabla_D \vee \neg \circ \nabla_D)(n),$$
if and only if
\begin{equation}\label{sf2}
\ET \models \forall n^\N . \{n\} \RI j[(\nabla_D \vee \neg \circ \nabla_D)(n)].
\end{equation}
We have
$$(\nabla_D \vee \neg \circ \nabla_D)(n) = \begin{cases}
\{0\} \wedge \N \cup \{1\} \wedge \emptyset & \textif n \in D \\
\{0\} \wedge \emptyset \cup \{1\} \wedge \N & \textif n \notin D
\end{cases} = \{\chi_D(n)\} \wedge \N.$$
So \eqref{sf2} is equivalent to (iv).

`(iii)$\BI$(iv)'.
Choose $\theta \in \PPN^\N$ such that $j \cong \lo_\theta$.
We have: (iii) holds, if and only if $\rho_D \leq_\lo \theta$, if and only if
\begin{eqnarray*}
\ET & \models & \forall n^\N . \{n\} \RI \forall_{A \in \rho_D(n)} \lo_\theta(A) \\
    & \eqq    & \forall n^\N . \{n\} \RI \forall_{A \in \{\{\chi_D(n)\}\}} \lo_\theta(A) \\
    & \eqq    & \forall n^\N . \{n\} \RI \lo_\theta(\{\chi_D(n)\}) \\
    & \eqq    & \forall n^\N . \{n\} \RI j(\{\chi_D(n)\}),
\end{eqnarray*}
as desired.
\end{proof}

\begin{remark}\label{HylandTuring}
Let $D \subseteq \N$.
The definition in \cite{hyl82} of the topology $k_D$ associated to the Turing degree $D$ comes down to that it is the least topology $j$ for which $D_j$ is decidable.
With our terminology: $k_D$ is (by the previous Proposition) thus the least topology in which $D$ is effective.
Since (also by the previous Proposition) our topology $\lo_{\rho_D}$ has this property, we must have $k_D = \lo_{\rho_D}$ as topologies on $\Eff$.
\end{remark}

\begin{application}
The least topology in which every set $\subseteq \N$ is effective, is $\neg\neg$.
\end{application}

\begin{proof}
By the last Proposition, this is just a paraphrase of Proposition \ref{Phoa3}.
\end{proof}

\section{The examples $\O_m^\alpha$}
\subsection{More on sights: intersection properties}

\begin{lemma}\label{joint-intersection}
Let $n \geq 1$.
Let $\mathcal{A}_1,\ldots,\mathcal{A}_n$ be collections such that the \defn{joint intersection property} is satisfied:
\begin{center}
for any $A_1 \in \mathcal{A}_1, \ldots, A_n \in \mathcal{A}_n$ we have $A_1 \cap \ldots \cap A_n \neq \emptyset$.
\end{center}
Given for each $1 \leq i \leq n$ a sight $S_i$ on $\A_i$, there is a finite sequence $d$ such that
\begin{itemize}
 \item $d$ is a node of each $S_i$, and
 \item $d$ is a leaf of one of the $S_i$.
\end{itemize}
\end{lemma}

\begin{proof}
We prove this by induction on $S_1$.

The case $S_1 = \nil$. Then $()$ is a leaf of $S_1$, and at the same time it is a node of $S_2,\ldots,S_n$. So fine.

The case $S_1 = (A_1,\sigma_1)$.
If $S_{i_0} = \nil$ for some $2 \leq i_0 \leq n$, then $()$ is a leaf of $S_{i_0}$ and a node of each $S_i$, so we can forget this case and assume that each $s_i$ is some $(A_i,\sigma_i)$.
By the joint intersection property, we can choose $a \in A_1 \cap \ldots \cap A_n$.
By the IH, there is a finite sequence $d'$ such that
\begin{itemize}
 \item $d'$ is a node of each $\alpha_i(a)$, and
 \item $d'$ is a leaf of $\alpha_{i_0}(a)$ for some $1 \leq i_0 \leq n$.
\end{itemize}
Clearly, the sequence $a \cons d'$ is a node of each $S_i = (A_i,\sigma_i)$, and is a leaf of $S_{i_0} = (A_{i_0},\sigma_{i_0})$.
Done.
\end{proof}

\begin{corollary}\label{n-intersection}
Let $n \geq 1$.
Let $\A$ be a collection that has the \defn{$n$-intersection property}:
\begin{center}
for any $A_1,\ldots,A_n \in \A$, the intersection $A_1 \cap \ldots \cap A_n$ is nonempty.
\end{center}
Then, given sights $S_1,\ldots,S_n$ on $\A$, there is a finite sequence $d$ such that
\begin{itemize}
 \item $d$ is a node of each $S_i$, and
 \item $d$ is a leaf of some $S_i$. \qed
\end{itemize}
\end{corollary}

\begin{lemma}\label{all-complete}
Let $S$ be a sight, and let $d = (x_1,\ldots,x_n)$ be a sequence that is a node of both $S$ and $T$.
Suppose that some $z$ is r-defined on both $S$ and $T$.
If $d$ is a leaf of, say, $S$, then it is a leaf of $T$.
\end{lemma}

\begin{proof}
By induction on $n$.
The case $n = 0$.
Then $S = \nil$, so $z = \la0,\_\ra$, so $T = \nil$ as well, so $d$ is a leaf of $b$.
The case $n > 0$.
Then $S$ is some $(A,\sigma)$ and $T$ is some $(B,\tau)$.
The length $n-1$ sequence $(x_2,\ldots,x_n)$ is a leaf of $\sigma(x_1)$ and a node of $\tau(x_1)$,
so by the IH it is a leaf of $\tau(x_1)$.
It follows that $(x_1,\ldots,x_n)$ is a leaf of $(B,\tau)$.
Done.
\end{proof}

%\begin{corollary}
%Suppose that $\mathcal{B}$ has finite intersection property.
%Then any finite number of $\mathcal{B}$-sights for $\la w,e \ra$ have a common descent. \qed
%\end{corollary}

\subsection{Comparisons between the examples $\O_m^\alpha$.}\label{subs:between}
We wish to establish strict inequalities between the examples $\O_m^\alpha$.
First we observe the following.

\begin{calculation}
Let $1 < 2m < \alpha \leq \omega$.
We have $\O_m^\alpha <_\mo \O_{m+1}^\alpha$.
\end{calculation}

\begin{proof}
In the following, a `co-$n$-ton' means a co-$n$-ton in $\alpha$.

To have $\leq_\mo$, we need $\realizes \forall_\text{co-$m$-ton} A \exists_\text{co-$m+1$-ton} B (B \ri A)$.
In fact, $\id$ realizes this; if $\alpha\bs{A'}$ is a co-$m$-ton, choose any $m+1$-ton $B' \subseteq \alpha$ with $A' \subseteq B'$, then $\alpha\bs{B'}$ is a co-$m+1$-ton and $\id \realizes \alpha\bs{B'} \ri \alpha\bs{A'}$.

We now show $\not\geq_\mo$.
Suppose, for contradiction, that some
$$\gamma \realizes \forall_\text{co-$m+1$-ton} A \exists_\text{co-$m$-ton} B (B \ri A).$$

Then $\Im(\gamma) \cap \alpha$ must contains at least $m+1$ values; for if not, say $\Im(\gamma) \cap \alpha = \{v_1,\ldots,v_k\}$ with $k \leq m$, choose any $m+1$-ton $A' \subseteq \alpha$ with $\{v_1,\ldots,v_k\} \subseteq A'$, then $\gamma \realizes B \ri \alpha\bs{A'}$ for some co-$m$-ton $B$, so $\gamma: B \ri \alpha\bs{A'} \subseteq \alpha\bs\{v_1,\ldots,v_k\}$ has in $\Im(\gamma) \cap \alpha$ more than $\{v_1,\ldots,v_k\}$, a constradiction.

Choose $m+1$ distinct elements $v_1,\ldots,v_{m+1}$ in $\Im(\gamma) \cap \alpha$.
Then there is an $m$-ton $B' \subseteq \alpha$ such that $\gamma: \alpha\bs{B'} \to \alpha\bs\{v_1,\ldots,v_{m+1}\}$.
But then $\gamma(B')$ must contain $v_1,\ldots,v_{m+1}$, which is impossible as $|\gamma(B')| \leq |B'| = m < m+1$.
Done.
\end{proof}

Naturally, we wonder whether we can refine this and have $\O_m^\alpha <_\lo \O_{m+1}^\alpha$.
First we should notice that for $\alpha = \omega$ the examples $\O_m^\alpha$ collapse w.r.t. $\leq_\lo$, as follows.

\begin{calculation}
Let $1 \leq m < \omega$.
We have $\O_1^\omega \geq_\lo \O_m^\omega$.
(Hence $\O_1^\omega \cong_\lo \O_m^\omega$.)
\end{calculation}

\begin{proof}
To establish this, we construct a partial computable function $\zeta: \N^* \pto \N$ and for each $A \in \O_m^\omega$ a $(\zeta,\O_1^\omega,A)$-supporting sight.
Our construction will be a kind of diagonal arugment, as the reader will notice.

For $1 \leq i \leq m$, we shall denote by $\pi^m_i$ the `projection' function $\N \to \N: \la x_1,\ldots,x_m \ra \mapsto x_i$.
Given $a_1,\ldots,a_m \in \N$, define a set
$$S_{a_1,\ldots,a_m} = \{(c_1,\ldots,c_p) \in \N^* \mid p \leq m \en \forall 1 \leq i \leq p (c_i \neq \pi_i^m(a_i)\},$$
which is evidently a well-founded tree.
On the other hand, there is clearly a partial computable function $\zeta: \N^* \pto \N$ satisfying for each $(c_1,\ldots,c_p) \in \N^*$ the condition
$$\zeta(c_1,\ldots,c_p) =
\begin{cases}
\la 1,0 \ra                      & \textif p < m \\
\la 0,\la c_1,\ldots,c_m \ra \ra & \textif p = m.
\end{cases}$$

We now show that for $a_1,\ldots,a_m \in \N$, the sight $S_{a_1,\ldots,a_m}$ is $(\zeta,\O_1^\omega,\N\bs\{a_1,\ldots,a_m\})$-supporting.
Let $(c_1,\ldots,c_p) \in S_{a_1,\ldots,a_m}$.
If $(c_1,\ldots,c_p)$ is a non-leaf, then $p < m$, so $\zeta(c_1,\ldots,c_p) = \la 1,0 \ra$, as desired.
If $(c_1,\ldots,c_p)$ is a leaf, then $p = m$, so $\zeta(c_1,\ldots,c_p) = \la 0,\la c_1,\ldots,c_m \ra \ra$.
We need that $\la c_1,\ldots,c_m \ra \notin \{a_1,\ldots,a_m\}$.
Suppose, for contradiction, that for some $1 \leq i \leq m$ we have $a_i = \la c_1,\ldots,c_m \ra$.
As $(c_1,\ldots,c_m) \in S_{a_1,\ldots,a_m}$, we have
$$c_i \neq \pi^m_i(a_i) = \pi^m_i(\la c_1,\ldots,c_m \ra) = c_i,$$
an absurdity.
This proves that $S_{a_1,\ldots,a_m}$ is $(\zeta,\O_1^\omega,\N\bs\{a_1,\ldots,a_m\})$-supporting.

Let us spell out how this applies to our situation.
Let $A \in \O_m^\omega$.
Then we may write $A = \N \bs \{a_1,\ldots,a_m\}$.
Therefore we have a $(\zeta,\O_1^\omega,A)$-supporting sight, namely $S_{a_1,\ldots,a_m}$.
We are done.
\end{proof}

The situation becomes interesting for $\alpha < \omega$.
The following proposition, which belongs to the realm of the intersection property machinary we have built above, allows us to establish $\O_m^\alpha \not\geq_\lo \O_{m+1}^\alpha$ for \emph{almost} all $m$.

\begin{proposition}\label{IscnPropNleq}
Let $\A,\B \in \PPN$, and let $n \geq 1$. Suppose that
\begin{itemize}
 \item $\A$ does not have the $n$-intersection property, while
 \item $\B$ has the $n$-intersection property.
\end{itemize}
Then $\mathcal{A} \not\leq_\lo \mathcal{B}$.
\end{proposition}

\begin{proof}
Suppose, for contradiction, that some $z \in \forall_{A \in \A} \loF_\B(A)$.
Since $\A$ does not have the $n$-intersection property, we can choose $A_1,\ldots,A_n \in \A$ with $A_1 \cap \ldots \cap A_n = \emptyset$.
Choose, for each $1 \leq i \leq n$, a $(z,\B,A_i)$-dedicated sight $S_i$.
Since $\B$ has $n$-intersection property, there is a sequence $d$ that is a leaf on each $S_i$.
So $z[d] \in \bigcap_{i} A_i = \emptyset$, an absurdity.
\end{proof}

\begin{proposition}
Let $1 < 2m < \alpha < \omega$.
The ceiled number $\lceil \alpha / m \rceil$ is the least number $d$ for which there are $d$ sets in $\O_m^\alpha$ with empty intersection.
\end{proposition}

\begin{proof}
Elementary reasoning.
\end{proof}

\begin{calculation}
Let $1 < 2m < \alpha < \omega$.
Suppose $\lceil \alpha / (m+1) \rceil < \lceil \alpha / m \rceil$.
Then $\O_\alpha^{m+1} \not\leq_\lo \O_\alpha^{m}$.
\end{calculation}

\begin{proof}
Let $d = \lceil \alpha / (m+1) \rceil$, i.e., the least number $d$ for which there are $d$ sets in $\O_{m+1}^\alpha$ with empty intersection.
Let $A_1,\ldots,A_d \in \O_{m+1}^\alpha$ with $\bigcap_{1 \leq i \leq d} A_i = \emptyset$.\footnote{For example, one can take $A_i = \N \bs \{n \in \N \mid (i-1)(m+1) \leq n < i*(m+1)\}$.}
But $\O_m^\alpha$ has $n$-intersection property.
So the Proposition \ref{IscnPropNleq} applies, and we have $\O_\alpha^{m+1} \not\leq_\lo \O_\alpha^{m}$.
\end{proof}

\begin{openquiz}
Refine this: do we have $\O_{m+1}^\alpha \not\leq_\lo \O_m^\alpha$ for \emph{all} $1 < 2m < \alpha < \omega$?
\end{openquiz}

Next, instead of varying the subscript $m$ in the exmaples $\O_m^\alpha$, we can try to vary the superscript $\alpha$.
First we notice that increasing $\alpha$ means going lower in the order ($\leq_\mo$ and) $\leq_\lo$, as follows.

\begin{calculation}
Let $1 < 2m < \beta \leq \alpha \leq \omega$.
We have $\O_m^\alpha \leq_\mo \O_m^\beta$.
\end{calculation}

\begin{proof}
It is easy to see that $\id \realizes \forall A \in \O_m^\alpha \exists B \in \O_m^\beta (B \ri A)$.
\end{proof}

\begin{calculation}
Let $1 < 2m < \alpha < \omega$. We have $\O^\omega \not\geq_\lo \O_m^\alpha$. (Hence $\O^\omega <_\lo \O_m^\alpha$.)
\end{calculation}

\begin{proof}
Let $n = |\O_m^\alpha|$, which is finite.
Clearly $\O^\omega$ has $n$-intersection property, and $\bigcap \O^\omega = \emptyset$.
So by Proposition \ref{IscnPropNleq} we have `$\not\geq_\lo$', as desired.
\end{proof}

\begin{calculation}
Let $3 \leq \alpha < \omega$, and let $2m < \alpha$.
We have $\O_m^{\alpha+m} \not\geq_\lo \O_m^\alpha$.
(Hence $\O_m^{\alpha+m} <_\lo \O_m^\alpha$.)
\end{calculation}

\begin{proof}
Put $d = \lceil \alpha/m \rceil$.
Then $\O_m^\alpha$ contains $n$ sets whose intersection is empty, while $\O^m_{\alpha+m}$ has the $d$-intersection property.
Therefore by Proposition \ref{IscnPropNleq} we have `$\not\geq_\lo$'.
\end{proof}

\begin{openquiz}
Refine this: do we have $\O_m^{\alpha+1} \not\geq_\lo \O_m^\alpha$?
\end{openquiz}

\begin{openquiz}
Even if the last two questions are answered ``as desired'', so that all the `horizontal' and `vertical' $\leq_\lo$-inequalities in the triangular matrix $(\O_m^\alpha \mid 1 < 2m < \alpha < \omega$) are strict, it remains unknown how two non-horizontal\&non-vertical entries compare to each other.
For instance, for $1 \leq m < n$, how do $\O_m^{2m+1}$ and $\O_n^{2n+1}$ compare?
\end{openquiz}

\begin{openquiz}
Take two entries from the matrix $(\O_m^\alpha)$, and consider their join and their meet. What do we get - again some $\O_m^\alpha$? (Bear in mind that given $\A,\B \in \PPN$ their $\leq_\lo$-join is $\A \owedge \B$, which is finite if $\A,\B$ were finite.)
\end{openquiz}

\subsection{The example $\O^\omega$ is an atom among the basics}
\begin{proposition}\label{prop:atom}
$\O^\omega$ is the least basic topology $> \id$.
\end{proposition}

\begin{proof}
We have $\O_1^\omega >_\lo \bot$, because $\bigcap \O_1^\omega = \emptyset$.

Let $\A \in \PPN$ be such that $\A >_\lo \bot$.
To have $\O^\omega \leq_\lo \A$, it suffices that $\O^\omega \leq_\mo \A$, i.e. that
$$\realizes \forall_{n \in \N} \exists_{A \in \mathcal{A}} (A \ri \N\bs\{n\}).$$
But $\id$ realizes this; given $n \in \N$, since $\bigcap \A = \emptyset$, we can choose $A \in \A$ with $n \notin A$.
Done.
\end{proof}

\begin{remark}
However, $\O_1^\omega$ is not an atom in $(\PPN^\N,\leq_\lo)$:
if it were so, then for any $D \subseteq \N$ non-recursive we have $\O^\omega \leq_\lo \rho_S$, which is impossible.
\end{remark}

\begin{openquiz}
Is there the least topology $> \id$?
\end{openquiz}

\subsection{The topologies $\O_m^\alpha$ do not add effective sets.}
\begin{proposition}\label{OnlyRecEffe}
Let $1 < 2m < \alpha \leq \omega$ (as always).
Let $D \subseteq \N$.
If $D$ is effective in $\O_m^\alpha$, then $D$ is recursive.
\end{proposition}

\begin{proof}
Let us first do some prepatory work.
Choose a partial computable function $f: \N \pto \N$ subject to the following action.
\begin{quote}
Let $\la i,x \ra \in \N$.

If $i = 0$, then $f(\la i,x \ra) = x$.

If $i = 1$, then, writing $x = \la \_,e \ra$, find (with some algorithm) distinct $a_1,\ldots,a_{m+1} \in \alpha$ with $e(a_1),\ldots,e(a_{m+1})\cvg$ and $f(e(a_1)) = \ldots = f(e(a_{m+1}))$,
and (if found) put $f(\la i,x \ra) = f(e(a_1))$.
\end{quote}
We show the statement
\begin{equation}\label{ore1}
\text{\begin{tabular}{c}
For every sight $S$, for every $\la i,x \ra \in \N$, \\
if $S$ is $(\la i,x \ra,\O_m^\alpha,\{\chi_D(n)\})$-dedicated, \\
then $f(\la i,x \ra)$ is defined and is equal to $\chi_D(n)$.
\end{tabular}}
\end{equation}
by induction on $S$.

The case $S = \nil$.
Then, since $S$ is $(\la i,x \ra,\O_m^\alpha,\{\chi_D(n)\})$-dedicated, we have $i = 0$ and $x \in \{\chi_D(n)\}$.
So, by definition of $f$, we have $f(\la i,x \ra) = \chi_D(n)$ as desired.

The case $S = (A,\sigma)$.
We have $i = 1$, and $x = \la \_,e \ra$ for some $e \in \N$.
As $2m+1 \leq \alpha$, we have $m+1 \leq \alpha-m$.
So, since $A \in \O_m^\alpha$, we have $m+1 \leq |A|$.
For any $a \in A$, we have $e(a)\cvg$ and, by the IH, $f(e(a)) = \chi_D(n)$.
Let $a_1,\ldots,a_{m+1} \in \alpha$ as in the definition of $f$, so that $e(a_1),\ldots,e(a_{m+1})\cvg$ and $f(e(a_1)) = \ldots = f(e(a_{m+1}))$.
Since $A$ contains all but $m$ elements in $\alpha$, there must be some $1 \leq i_0 \leq m+1$ such that $a_{i_0} \in A$.
So, by definition of $f$, we have $f(\la i,x \ra) = f(e(a_1)) = f(e(a_{i_0})) = \chi_D(n)$, as desired.
Induction complete.

Let us now show the Proposition.
Assume $\rho_D \leq_\lo \O_m^\alpha$.
Take $\gamma \realizes \forall n \in \N . \{n\} \ri \loF_{\O_m^\alpha}(\{\chi_S(n)\})$.
From \eqref{ore1} it follows that
\begin{center}
for all $z \in \loF_{\O_m^\alpha}(\{\chi_D(n)\})$ we have $f(z) = \chi_D(n)$.
\end{center}
Therefore the computable function $f \circ \varphi_\gamma$ is the characteristic function of $D$.
So $D$ is decidable.
We are done.
\end{proof}

\section{The example $\F^*$}
\subsection{More on sights: sectors}
\begin{definition}
A sight $S$ is a \defn{sector} of a sight $T$ when, by induction on $S$,
\begin{itemize}
 \item if $S = \nil$, then $T = \nil$,
 \item if $S = (A,\sigma)$, then $T = (B,\tau)$ with $A \subseteq B$ and for each $a \in A$ the sight $\sigma(a)$ is a sector of $\tau(a)$.
\end{itemize}
A sight $S$ is \defn{finitary} when, inductively,
\begin{itemize}
 \item if $S = \nil$, then always,
 \item if $S = (A,\sigma)$, then $A$ is finite and for each $a \in A$ the sight $\sigma(a)$ is finitary.
\end{itemize}
For $n \in \N$, a sight $s$ is \defn{full $n$-ary} when, inductively,
\begin{itemize}
 \item if $S = \nil$, then always,
 \item if $S = (A,\sigma)$, then $|A| = n$ and for all $a \in A$ the sight $\sigma(a)$ is full $n$-ary.
\end{itemize}
\end{definition}

\begin{lemma}\label{HasFullBinary}
Let $n \in \N$, and let $\A \in \PPN$ such that each $A \in \A$ has at least $n$ elements.
Then a sight $S$ on $\A$ has a full $n$-ary sector.
\end{lemma}

\begin{proof}
By induction on $S$.

The case $S = \nil$.
Then $S$ is a full $n$-ary sector of itself.

The case $S = (A,\sigma)$.
By the premise, one can choose a $n$-elements subset $B$ of $A$.
For each $b \in B$, by the IH, one can choose a full $n$-ary sector $\tau(a)$ of $\sigma(a)$.
Then $(B,\tau)$ is a full $n$-ary sector of $(A,\sigma) = S$.
Done.
\end{proof}

\begin{proposition}[K\"onig's Lemma\footnote{But in our case the proof is constructive.} for sights]
A finitary sight $S$ has only finitely many nodes.
\end{proposition}

\begin{proof}
By induction on $S$.

The case $S = \nil$.
Then the empty sequence $()$ is the only node of $S$.
So good.

The case $S = (A,\sigma)$.
For each $a \in A$, let $\nu(a)$ be the number of nodes of $\sigma(a)$, which is finite by the IH.
Then $\nu(a)$ is the number of nodes of $(A,\sigma)$ with first component equal to $a$.
Since the nodes of $(A,\sigma)$ are the empty sequence $()$ and for each $a \in A$ the sequences $a \cons d$ with $d$ a node of $\sigma(a)$,
we have that the number of nodes of $(A,\sigma)$ is $1 + \sum_{a \in A} \nu(a)$, which is finite as $A$ is finite.
Done.
\end{proof}

\begin{corollary}\label{FinitaryFinite}
Let $S$ be a finitary sight, and let $z \in \N$ be r-defined on $S$.
Then the r-image $z[S]$ is finite.
\end{corollary}

\begin{proof}
Clearly, $S$ has only finitely many leaves. So by the definition of r-image, $z[S]$ is finite.
\end{proof}

\subsection{Comparison to the $\O_m^\alpha$}
We prove that the example $\F^*$ is incomparable (w.r.t. $\leq_\lo$) to $\O_m^\alpha$ ($\alpha < \omega$).
A special case of this result, namely the incomparability of the collections $\ups \N$ and $\{\{0,1\},\{1,2\},\{2,0\}\}$, was provided as an exercise to the author by Jaap van Oosten.\footnote{This problem was the trigger for the author to formulate sights and define the collections $\O_m^\alpha$.}

\begin{calculation}\label{fNotLeqO}
Let $3 \leq \alpha \leq \omega$, and let $1 < 2m < \alpha$.
We have $\F^* \not\leq_\lo \O_m^\alpha$.
\end{calculation}

\begin{proof}
Suppose, for contradiction, that some $z \in \forall_{A \in \F^*} \loF_{\O_m^\alpha}(A)$.
Choose a $(z,\O_m^\alpha,\N)$-dedicated sight $S$.
Each $B \in \O_m^\alpha$ has at least $m+1$ elements\footnote{by the assumption $2m < \alpha$}, so $S$ has\footnote{See Lemma \ref{HasFullBinary}.} a full $(m+1)$-ary sector $S'$, which is clearly $(z,\{\text{the $m+1$-tons} \subseteq \alpha\},\N)$-dedicated.
Note that the image $z(S')$ is\footnote{See Corollary \ref{FinitaryFinite}.} finite.
Choose a $(z,\O_m^\alpha,\N \bs z(S'))$-dedicated sight $T$.
Since
\begin{itemize}
\item the sight $S'$ is on $\{\text{the $m+1$-tons} \subseteq \alpha\}$
\item the sight $T$ is on $\O_m^\alpha$,
\item $\{\text{the $m+1$-tons} \subseteq \alpha\}$ and $\O_m^\alpha$ have the joint intersection property,
\end{itemize}
there is\footnote{by Lemma \ref{joint-intersection}} a sequence $d$ that is a node of both $S'$ and $T$ and a leaf of one of them.
But since $z$ is r-defined on both $S'$ and $T$, the sequence $d$ must be\footnote{by Lemma \ref{all-complete}} a leaf of both $S'$ and $T$.
Therefore $z[d] \in z[S'] \cap z[T] \subseteq z[S'] \cap (\N \bs z[S']) = \emptyset$, an absurdity.
\end{proof}

\begin{calculation}
Let $1 < 2m < \alpha < \omega$.
Then $\O_m^\alpha \not\leq_\lo \F^*$.
\end{calculation}

\begin{proof}
Clearly $\F^*$ clearly has $|\O_m^\alpha|$-intersection property.
Moreover $\bigcap \O_m^\alpha = \emptyset$.
So Proposition \ref{IscnPropNleq} applies, and we have `$\not\leq_\lo$'.
\end{proof}

\begin{proposition}\label{IscOwedge}
Let $\A_1,\ldots,\A_k \in \PPN$ and $n_1,\ldots,n_k \in \N$ such that $\A_i$ has $n_i$-intersection property for each $1 \leq i \leq k$.
Then $\A_1 \owedge \ldots \owedge \A_k$ has $\min(n_1,\ldots,n_k)$-intersection property.
\end{proposition}

\begin{proof}
By means of induction, it suffices to show that if $\A \in \PPN$ has $n$-intersection property and $\B \in \PPN$ has $m$-intersection property then $\A \owedge \B$ has $\min(n,m)$-intersection property.
Let $(A_1 \wedge B_1),\ldots,(A_{\min(n,n')} \wedge B_{\min(n,n')}) \in \A \owedge \B$.
We can choose some $a \in A_1 \cap \ldots \cap A_{\min(n,n')}$ and some $b \in B_1 \cap \ldots \cap B_{\min(n,n')}$.
Then $\la a,b \ra \in (A_1 \wedge B_1) \cap \ldots \cap (A_{\min(n,n')} \wedge B_{\min(n,n')})$.
Done.
\end{proof}

\begin{calculation}
Let $1 < 2m < \alpha \leq \omega$.
We have $\O_m^\alpha \vee_\lo \F^* <_\lo \negneg$.
\end{calculation}

\begin{proof}
Since $\negneg$ is an upper atom in $(\PPN^\N,\leq_\lo)$, it suffices to show that $\negneg \not\leq_\lo \O_m^\alpha \owedge \F^*$.
Since $\O_m^\alpha$ as well as $\F^*$ has $2$-intersection property, so does $\O_m^\alpha \owedge \F^*$.
By Proposition \ref{IscnPropNleq}, we have $\negneg \cong \{\{0\},\{1\}\} \not\leq_\lo \O_m^\alpha \owedge \F^*$ as desired.
\end{proof}

\begin{calculation}
Let $1 \leq k \in \N$.
We have $\bigvee_{1 \leq m \leq k} \O_m^{2m+1} <_\lo \negneg$.
\end{calculation}

\begin{proof}
Each $\O_m^{2m+1}$ has $2$-intersection property, so $\owedge_{1 \leq m \leq k} \O_m^{2m+1}$ does.
By Proposition \ref{IscnPropNleq}, we have $\bigvee_{1 \leq m \leq k} \cong \owedge_{1 \leq m \leq k} \O_m^{2m+1} \not\geq_\lo \negneg$.
Therefore we have `$<_\lo$'.
\end{proof}

\begin{proposition}\label{NoIscOwedge}
Let $n \geq 1$, and let $\B \in \PPN$ have $n$-intersection property.
Let $\A_1,\ldots,\A_k \in \PPN$ $(k \geq 1)$ be non-empty collections such that some $\A_i$ does not have $n$-intersection property.
Then $\A_1 \owedge \ldots \owedge \A_k \not\leq_\lo \B$.
\end{proposition}

\begin{proof}
Without loss of generality, assume that $\A_1$ does not have $n$-intersection property.
Choose $Z_1,\ldots,Z_n \in \A_1$ with $Z_1 \cap \ldots \cap Z_n = \emptyset$.
Choose $A_2 \in \A_2, \ldots, A_k \in \A_k$.
Then the $n$-sets $(Z_1 \wedge A_2 \wedge \ldots \wedge A_k),\ldots,(Z_n \wedge A_2 \wedge \ldots \wedge A_k)$ clearly have empty intersection.
So $\A_1 \owedge \ldots \owedge \A_k$ does not have $n$-intersection property.
Therefore by Proposition \ref{IscnPropNleq} we have `$\not\leq_\lo$'.
\end{proof}

\begin{calculation}
Let $1 \leq k \in \N$.
We have $\vee_{1 \leq m \leq k} \O_m^{2m+1} \not\leq_\lo \F^*$.
\end{calculation}

\begin{proof}
The collection $\F^*$ has $3$-intersection property, but $\O_1^3$ does not have $3$-intersection property.
Therefore Proposition \ref{NoIscOwedge} applies, and we have `$\not\leq_\lo$'.
\end{proof}

\begin{openquiz}
Can we (mimic the proof of Proposition \ref{fNotLeqO} to) show that
\begin{equation}\label{fNotLeqMowedge}
\F^* \not\leq_\lo \owedge_{1 \leq m \leq k} \O^m_{2m+1}?
\end{equation}
\end{openquiz}

\subsection{Arithmetical sets are effective in $\F^*$}\label{subs:arith}
In this subsection, we prove that any arithmetical set is effective in $\F^*$.
This problem was suggested to the author by Jaap van Oosten.
Clearly, it suffices to show that for each $k \in \N$ every $\Pi_k$-set is effective in $\F^*$.
Let $D \subseteq \N$, and let $\phi(x_0,x_1,\ldots,x_k)$ be a $\Delta_0$-formula such that for all $n \in \N$,
\begin{center}
$n \in D$ if and only if $\N \models \forall x_1 \exists x_n \cdots \forall x_{k-1} \exists x_k \phi(n,x_1,\ldots,x_k)$.
\end{center}
We present the proof that $D$ is effective in $\F^*$, as follows.

\begin{notation}
Let $(X,R)$ be a poset.
Let $S$ be a set of finite sequences in $X$, i.e. $S \subseteq X^*$.
We denote by $(S,R_\lex)$ the lexicographic order induced by $(X,R)$.
\end{notation}

\begin{proposition}
Let $(X,R)$ be a poset, and let $S \subseteq X^*$.
\claim{If $(X,R)$ is a total order, then so is $(S,R_\lex)$.
If $(X,R)$ is a well-order, then so is $(S,R_\lex)$.}
\end{proposition}

\begin{proof}
Elementary reasonings.
\end{proof}

\begin{definition}
Let $T \subseteq \N^*$.
We define a relation $\leq$ on $T$ by
\begin{center}
$t \leq t'$ $\iff$ $\len(t) < \len(t')$, or $\len(t) = \len(t')$ and $t \leq_\lex t'$.
\end{center}
\claim{This relation is a well-order.}
\end{definition}

\begin{proof}
Clearly, the relation $(T,\leq)$ is reflexive, antisymmetric and transitive.

Let us prove that $(T,\leq)$ is total.
Let $t,t' \in T$, and suppose that not $t \leq t'$.
Then $\len(t) \geq \len(t')$, and either $\len(t) \neq \len(t')$ or $t >_\lex t'$.
If the former is the case, then $\len(t) > \len(t')$, so that $t \geq t'$.
If the former is not the case, i.e. $\len(t) = \len(t')$, then the latter must hold, i.e., $t >_\lex t'$, so $t \geq t'$ as well.
This proves the totalness.

Last, we show that any non-empty $S \subseteq T$ has a least element.
Let $m = \min\{\len(s) \mid s \in S\}$, and consider the `level $m$ subset'
$$S_0 = \{s \in S \mid \len(s) = m\}$$
of $S$.
The set $S_0$ is clearly non-empty, so has a least element $s_0$ w.r.t. the well-ordering $\leq_\lex$.
Clearly $s_0 \leq s$ for all $s \in S$, so $s_0$ is a least element of $(S,\leq)$.
This proves the well-orderedness. We are done.
\end{proof}

\begin{notation}
Let $S \subseteq \N^*$.
For $m \in \N$, we write $S^{<m} = \{s \in S \mid \len(s) < m\}$.
Suppose that $S$ is a tree.
For $s \in S$, we write $\isLeaf(s)$ if $s$ is a leaf.
We write $\lLeaf(S)$ for the least leaf of $(S,\leq)$.
\end{notation}

\begin{construction}
Let $s = (c_1,\ldots,c_p) \in \N^*$.
We define a finite tree $T_s$, as follows by induction on $p$.
\begin{align*}
& T_{()} = \{()\} \\
& T_{(c_1,\ldots,c_p)} = T_{(c_1,\ldots,c_{p-1})} \cup \{\lLeaf(T_{(c_1,\ldots,c_{p-1})}) \cons x \mid 0 \leq x \leq c_p\}.
\end{align*}
We define the \defn{stage} of $s$, denoted $\stage(s)$, to be the least length of a leaf in the tree $T_s$.
\end{construction}

\begin{definition}\label{BN}
Let $n \in \N$.
Let $(d_1,\ldots,d_q) \in \N^*$ with $q < k$.
We define a number $b_n(d_1,\ldots,d_q) \in \N$ as follows.
\begin{itemize}
\item[(i)]
If $q$ is odd and $\N \models \exists x_{q+1} \forall x_{q+2} \ldots \exists x_k \phi(n,d_1,\ldots,d_q,x_{q+1},\ldots,x_k)$, \\
let $b_n(d_1,\ldots,d_q)$ be the least $y$ such that \\
$\N \models \forall x_{q+2} \ldots \forall x_{k-1} \exists x_k \phi(n,d_1,\ldots,d_q,y,x_{q+2},\ldots,x_k)$.
\item[(ii)]
If $q$ is odd and $\N \not\models \exists x_{q+1} \forall x_{q+2} \ldots \exists x_k \phi(n,d_1,\ldots,d_q,x_{q+1},\ldots,x_k)$, \\
let $b_n(d_1,\ldots,d_q) = 0$.
\item[(iii)]
If $q$ is even and $\N \models \forall x_{q+1} \exists x_{q+2} \ldots \forall x_{k-1} \exists x_k \phi(n,d_1,\ldots,y_q,x_{q+1},\ldots,x_k)$, \\
let $b_n(d_1,\ldots,d_q) = 0$.
\item[(iv)]
If $q$ is even and $\N \not\models \forall x_{q+1} \exists x_{q+2} \ldots \forall x_{k-1} \exists x_k \phi(n,d_1,\ldots,y_q,x_{q+1},\ldots,x_k)$, \\
let $b_n(_1,\ldots,d_q)$ be the least $y \in \N$ such that \\
$\N \not\models \exists x_{q+2} \ldots \forall x_{k-1} \exists x_k \phi(n,d_1,\ldots,d_q,y,x_{q+2},\ldots,x_k)$.
\end{itemize}
\end{definition}

\begin{construction}\label{SCon}
Let $n \in \N$.
We define $S_n \subseteq \N^*$, by the following inductive clauses.
\begin{itemize}
\item $() \in S_n$
\item Provided $(c_1,\ldots,c_p) \in S_n$, we have $(c_1,\ldots,c_p,c_{p+1}) \in S_n$ if
      \begin{itemize}
      \item $\stage(c_1,\ldots,c_p) < k$, and
      \item $b_n(\lLeaf(T_{(c_1,\ldots,c_p)})) \leq c_{p+1}$.
      \end{itemize}
\end{itemize}
We claim of this construction the following properties.
\claim{\begin{itemize}
 \item[(a)] $S_n$ is a sight on $\Pitts$.
 \item[(b)] The assignment $(n \in \N, s \in S_n) \mapsto (\isLeaf_{S_n}(s) \in 2)$ is effective.
\end{itemize}}
\end{construction}

\begin{proof}
(a)
Clearly, $S_n$ is a tree on $\Pitts$.
It remains to prove that the poset $(S_n,\preceq)$ has no infinite chain.
Suppose, for contradiction, that there is a sequence $(c_i \mid 1 \leq i)$ such that $(c_1,\ldots,c_p) \in S_n$ for each $p \in \N$.

We show the statement
\begin{equation}\label{sc1}\text{
For each $0 \leq m \leq k$, there is $p \in \N$ such that $\stage(c_1,\ldots,c_p) = m$.
}\end{equation}
by induction on $m$.

The case $m = 0$.
Clearly $\stage() = 0$, as desired.

The case $m > 0$.
By the IH, there is $r \in \N$ such that $\stage(c_1,\ldots,c_r) = m-1$.
Let $d$ be the number of length $m-1$ leaves in $T_{(c_1,\ldots,c_r)}$.
It is easy to see\footnote{but not very short if you work out formally} from the definitions of $S_n$ and of $T_\ph$ that $\stage(c_1,\ldots,c_{r+d}) = m$.
This proves \eqref{sc1}.

By \eqref{sc1}, there is $p \in \N$ such that $\stage(c_1,\ldots,c_p) = k$.
Since $(c_1,\ldots,c_p,c_{p+1}) \in S_n$, we have by definition of $S_n$ that $T_{(c_1,\ldots,c_p)}$ has a leaf of length $< k$, a contradiction.
We conclude that $S_n$ is a sight.

(b)
Convince yourself of this by understanding the definition of $S_n$.
\end{proof}

\begin{definition}\label{CNS}
Let $n \in \N$.
Let $s = (c_1,\ldots,c_p) \in S_n$.
Let $t = (d_1,\ldots,d_q) \in T_s$ with $q < \stage(s)$.
\claim{\begin{itemize}
\item[(a)] There is unique $r < p$ with $t = \lLeaf(T_{(c_1,\ldots,c_r)})$.
\end{itemize}}
\noindent Define $c_{n,s}(t) = c_{r+1}$.\footnote{Alternatively, we could define $c_{n,s}(t) = \max(\Out_{T_s}(t))$. But the description we have taken allows us to have the core property (b) immediately.}
\claim{\begin{itemize}
\item[(b)] We have $b_n(t) \leq c_s(t)$.
\item[(c)] The assignment $(n \in \N, s \in S_n, t \in T_s^{<\stage(s)}) \mapsto (c_{n,s}(t) \in \N)$ is effective.
\end{itemize}}
\end{definition}

\begin{proof}
All by definition of $S_n$.
\end{proof}

\begin{construction}\label{TauCon}
Let $n \in \N$, and let $s \in \Lvs(S_n)$.
We are about to construct an ordinary predicate $\tau_{n,s}$ on $T_s$.
For $(d_1,\ldots,d_q) \in T_s$, we define $\tau_{n,s}(d_1,\ldots,d_q) = \true$ if, by induction on $k-q$,
\begin{itemize}
\item[] (case $q = k$) $\N \models \phi(n,d_1,\ldots,d_k)$,
\item[] (case $q < k$ odd) there is $y \leq c_s(d_1,\ldots,d_q)$ with $\tau_{n,s}(d_1,\ldots,d_q,y)$,
\item[] (case $q < k$ even) for all $y \leq c_s(d_1,\ldots,d_q)$ one has $\tau_{n,s}(d_1,\ldots,d_q,y)$.
\end{itemize}
We claim of this construction the following properties.
\claim{\begin{itemize}
\item[(a)] We have $\tau_{n,s}() = \chi_D(n)$.
\item[(b)] The assignment $(n \in \N, s \in \Lvs(S_n), t \in T_s) \mapsto (\tau_{n,s}(t) \in 2)$ is effective.
\end{itemize}}
%Clearly, the predicate $\tau_s$ on $T_s$ is effective, and the assignment $\N^* \to \P(\N^*): s \mapsto \tau_s$ is effective.
\end{construction}

\begin{proof}
(a)
We will prove the statement
\begin{eqnarray}
&\text{For all $(d_1,\ldots,d_q) \in T_s$, we have $\tau_{n,s}(d_1,\ldots,d_q)$ if and only if} \nonumber \\
&\begin{cases}
\N \models \exists x_{q+1} \forall x_{q+2} \cdots \forall x_{k-1} \exists x_k \phi(n,d_1,\ldots,d_q,x_{q+1},\ldots,x_k) & \text{if $q$ is odd} \\
\N \models \forall x_{q+1} \exists x_{q+2} \cdots \forall x_{k-1} \exists x_k \phi(n,d_1,\ldots,d_q,x_{q+1},\ldots,x_k) & \text{if $q$ is even.}
\end{cases} \label{tc1}
\end{eqnarray}
by induction on $k - q$.
The goal statement follows from this statement, as then: $\tau_{n,s}()$ if and only if $\N \models \forall x_{1} \exists x_{2} \cdots \forall x_{k-1} \exists x_k \phi(n,x_1,\ldots,x_k)$ if and only if $n \in D$.

Let us carry out the induction.
The case $q = k$ is immediate.

The case $q < k$ odd.
`if':
Write $y = b(d_1,\ldots,d_q)$.
By Definition \ref{BN} clause (i), we have
$\N \models \forall x_{q+2} \ldots \exists x_k \phi(n,d_1,\ldots,d_q,y,x_{q+2},\ldots,x_k)$.
By the IH, we have $\tau_{n,s}(d_1,\ldots,d_q,y)$.
By the claim (b) in Definition \ref{CNS}, $y \leq c_s(d_1,\ldots,d_q)$.
Therefore $\tau_{n,s}(d_1,\ldots,d_q)$ holds, as desired.
`only if':
Choose $y \leq c_s(d_1,\ldots,d_q)$ such that 
$\tau_{n,s}(d_1,\ldots,d_q,y)$.
By the IH, $\N \models \forall x_{q+2} \ldots \exists x_k \phi(n,d_1,\ldots,d_q,y,x_{q+2},\ldots,x_k)$.
Thus \eqref{tc1} follows immediately, as desired.

The case $q < k$ even.
`if':
Let $y \leq c_s(d_1,\ldots,d_q)$.
By \eqref{tc1}, we have $\N \models \exists x_{q+2} \cdots \forall x_{k-1} \exists x_k \phi(n,d_1,\ldots,d_q,y,x_{q+2},\ldots,x_k)$.
By the IH, we have $\tau(d_1,\ldots,d_q,y)$ as desired.
`only if':
We prove the contrapositive.
Write $y = b(d_1,\ldots,d_q)$.
As we have the negation of \eqref{tc1}, we have by Definition \ref{BN} clause (i) that
$\N \not\models \exists x_{q+2} \ldots \forall x_{k-1} \exists x_k \phi(n,d_1,\ldots,d_q,y,x_{q+2},\ldots,x_k)$.
By the IH, this means that $\neg \tau(d_1,\ldots,d_q,y)$.
But by the claim (b) in Definition \ref{CNS}, we have $y \leq c_s(d_1,\ldots,d_q)$.
It follows that $\neg \tau(d_1,\ldots,d_q)$, as desired.
Induction done.

(b)
Convince yourself of this by staring at the definition\footnote{which is in fact designed for this to be clear} of $\tau_{n,s}(t)$, using
\begin{itemize}
 \item that formula $\phi$ is $\Delta_0$, and
 \item that\footnote{See the claim (c) in Definition \ref{CNS}.} the mapping $(n \in \N, s \in \Lvs(S_n), t \in T_s^{<k}) \mapsto (c_{n,s}(t) \in \N)$ is effective.
\end{itemize}
We are done.
\end{proof}

\begin{lemma}
Let $n \in \N$.
Define a function $\epsilon_n: S_n \to \N$ by
$$\epsilon_n(s) =
\begin{cases}
\la 1,0 \ra & \text{if $\stage(s) < k$}, \\
\la 0,0 \ra & \text{if $\stage(s) = k$ and $\tau_{n,s}()$}, \\
\la 0,1 \ra & \text{if $\stage(s) = k$ and not $\tau_{n,s}()$}.
\end{cases}$$
\claim{\begin{itemize}
\item[(a)] The function $\epsilon_n: S_n \to \N$ is computable.
\item[(b)] The assignment $(n \in \N) \mapsto (\epsilon_n \in \N^{S_n})$ is computable.
\item[(c)] $\epsilon_n \realizes \loS_{\Pitts}(\{\chi_D(n)\})$.
\end{itemize}}
\end{lemma}

\begin{proof}
For (a),(b), it suffices that the assignment $(n \in \N, s \in S_n) \mapsto (\epsilon_n(s) \in \N)$ is effective.
Convince yourself of this by staring at the definition of $\epsilon_n(s)$, using
\begin{itemize}
\item that\footnote{See the claim (b) of Construction \ref{SCon}.} the assignment $(n \in \N, s \in S_n) \mapsto (\isLeaf_{S_n}(s) \in 2)$ is effective, and
\item that\footnote{See the claim (b) of Construction \ref{TauCon}.} the assignment $(n \in \N, s \in S_n) \mapsto \tau_{n,s}()$ is effective.
\end{itemize}

(c)
To this end, we show that the sight $S_n$ is $(\epsilon_n,\Pitts,\{\chi_D(n)\})$-supporting.
If $s \in S_n$ is not a leaf, then $\stage(s) < k$, so $\epsilon_n(s) = \la 1,0 \ra$, as required.
If $s \in S_n$ is a leaf, then $\stage(s) = k$, so $\epsilon_n(s) = \la 0,\tau_{n,s}() \ra = \la 0,\chi_D(n) \ra$, as required.
This proves that $S_n$ is $(\epsilon_n,\Pitts,\{\chi_D(n)\})$-supporting, as desired.
\end{proof}

\begin{thm}
$D$ is effective in $\Pitts$.
\end{thm}

\begin{proof}
By the previous lemma, we have $\epsilon_\ph \realizes \forall n^\N . \{n\} \RI \loS_\Pitts(\{\chi_D(n)\})$.
\end{proof}

\begin{openquiz}
Two related questions:
\begin{itemize}
 \item (Jaap van Oosten) Are only the arithemtical sets effective in $\Pitts$?
 \item Consider the arithmetical degree $\Delta_1^1$, and choose $D \in \Delta_1^1$. Trivially we have $\Pitts \nleq_\lo \rho_D$. How about $\rho_D \leq_\lo \Pitts$?
\end{itemize}
\end{openquiz}

\bibliographystyle{amsalpha}   % this means that the order of references
                            % is determined by the order in which the
                            % \cite and \nocite commands appear
\bibliography{mas}
\addcontentsline{toc}{chapter}{\protect\numberline{}Bibliography}

\end{document}